\documentclass[final]{siamart0516}
\usepackage{braket,amsfonts,color,makecell}
\usepackage{xcolor}
\usepackage{framed}
\usepackage[justification=centering]{caption}
\usepackage{shadethm}
\usepackage{epstopdf}
\usepackage{algpseudocode}
\usepackage{url}
\usepackage{verbatim}
\usepackage{amsmath} 
\usepackage{amssymb}  
\usepackage{amsfonts}
\usepackage{acronym}
\usepackage[normalem]{ulem}
\usepackage{graphicx}
\usepackage{epsfig}
\usepackage{cite}
\usepackage{algorithm}

\newtheorem{assumption}{Assumption}
\newtheorem{remark}{Remark}
\def\bko{{\rm 1\kern-.17em l}}

\newcommand{\an}[1]{{\color{black}#1}}
\newcommand{\fy}[1]{{\color{black}#1}}
\newcommand{\vvs}[1]{{\color{black}#1}}
\newcommand{\vus}[1]{{\color{black}#1}}
\newcommand{\us}[1]{{\color{black}#1}}
\newcommand{\fys}[1]{{\color{black}#1}}
\newcommand{\fyfy}[1]{{\color{black}#1}}
\renewcommand{\theenumi}{\arabic{enumi} )}

\def\be{\begin{enumerate}}
\def\ee{\end{enumerate}}

\newtheorem{property}{Property}

 \newcommand{\remove}[1]{}
\newcommand{\EXP}[1]{\mathsf{E}\!\left[#1\right] }

\def\sF{\mathcal{F}}

\def\Real{\mathbb{R}}
\def\g{\gamma}

\def\e{\epsilon}
\def\a{\alpha}

\usepackage{tikz}
\usetikzlibrary{decorations.pathreplacing,calc}
\newcommand{\tikzmark}[1]{\tikz[overlay,remember picture] \node (#1) {};}

\newcommand*{\AddNote}[4]{%
    \begin{tikzpicture}[overlay, remember picture]
        \draw [decoration={brace,amplitude=0.5em},decorate,ultra thick,black]
            ($(#3)!(#1.north)!($(#3)-(0,1)$)$) --  
            ($(#3)!(#2.south)!($(#3)-(0,1)$)$)
                node [align=center, text width=4.5cm, pos=0.5, anchor=west] {#4};
    \end{tikzpicture}
}%

\begin{document}
\title{On stochastic and deterministic quasi-Newton methods for non-strongly convex optimization: Asymptotic convergence and rate analysis\thanks{A preliminary version of this paper appeared in the IEEE $55^\text{th}$ Conference on Decision and Control (see reference \cite{FarzadCDC16}).}
}

\author{Farzad Yousefian\thanks{School of Industrial Engineering \& Management, Oklahoma State University, Stillwater, OK 74078, USA
    (\email{farzad.yousefian@okstate.edu}, \url{https://sites.google.com/site/farzad1yousefian}).}
    \and 
     Angelia Nedi{\' c}\thanks{School of Electrical, Computer and Energy Engineering, Arizona State University, Tempe, AZ 85287, USA
    (\email{Angelia.Nedich@asu.edu}, \url{https://ecee.engineering.asu.edu/project/angelia-nedich}).  Nedi\'c gratefully acknowledges the support of the NSF through grant CCF-1717391.}
    \and  
    Uday V. Shanbhag\thanks{Industrial \& Manufacturing Engineering,
Pennsylvania State University,  University Park, State College, PA 16802, USA
    (\email{udaybag@psu.edu}, \url{http://www.personal.psu.edu/vvs3}).}
}

\maketitle

\begin{abstract} Motivated by applications arising from large scale
optimization and machine learning, we consider stochastic quasi-Newton
(SQN) methods for solving unconstrained convex optimization problems.
Much of the convergence analysis of SQN methods, in both full and
limited-memory regimes, requires the objective function to be strongly
convex. However, this assumption is fairly restrictive and does not
hold in many applications. To the best of our knowledge, no rate
statements currently exist for SQN methods in the absence of such an
assumption. Also, among the existing first-order methods for addressing
stochastic optimization problems with merely convex objectives, those
equipped with provable convergence rates employ averaging. However,
this averaging technique has a detrimental impact on inducing sparsity.
Motivated by these gaps, we consider optimization problems with
non-strongly convex objectives with Lipschitz but possibly unbounded
gradients.  The main contributions of the paper are as follows: (i) To
address large scale stochastic optimization problems, we develop an
iteratively regularized stochastic limited-memory BFGS (IRS-LBFGS)
algorithm, where the stepsize, regularization parameter, and the
Hessian inverse \fy{approximation} are updated iteratively. We
establish convergence of the iterates (with no averaging) to an optimal
solution of the original problem both in an almost-sure sense and in a
mean sense. The convergence rate is derived in terms of the objective
function values and is shown to be
$\mathcal{O}\left(1/k^{\left({1}/{3}-\e\right)}\right)$, where $\e$ is
an arbitrary small positive scalar; (ii) In deterministic regimes, we
show that \fys{the algorithm} displays a rate
\fys{$\mathcal{O}({1}/{k^{1-\e}})$. We present numerical experiments}
performed on a large-scale text classification problem \fy{and compare
IRS-LBFGS with standard SQN methods as well as first-order methods such
as SAGA and IAG}. 
\end{abstract}

\section{Introduction}\label{sec:intro}
\fy{We} consider the following stochastic optimization problem:
\begin{equation}\label{eqn:problem}
\min_{x \in \Real^n} f(x)\fy{\ \triangleq \ }\EXP{F(x,\xi(\us{\omega}))},\tag{SO}
\end{equation}
where $F: \Real^n\times\us{\Real^d}\to\mathbb{R}$ is a \fy{real-valued} function, the random
vector $\xi$ is defined as {$\xi:\Omega
\rightarrow \Real^d$}, \us{$(\Omega,{\cal F}, \mathbb{P})$ denotes the
associated probability space\fy{,} and the expectation \us{$\EXP{F(x,\xi)}$} is
taken with respect to $\mathbb{P}$}. Problem \eqref{eqn:problem} provides a general framework that can capture a wide range of applications in operations research, machine learning, statistics and control to name a few \fys{(\vus{cf.}\cite{bottou-2010,nocedal15}). Addressing problem \eqref{eqn:problem}} has led
	to significant \vus{progress via} Monte\vus{-}Carlo sampling techniques. \fy{\vus{Amongst such schemes}, stochastic approximation
	(SA) methods~\cite{robbins51sa,nemirovski_robust_2009} have \vus{proved particularly} popular.} \fy{The standard SA method, introduced by Robbins and
	Monro~\cite{robbins51sa}, \vus{for solving} \eqref{eqn:problem}, produces a sequence $\{x_k\}$ using the following update rule}
\begin{equation}\label{eqn:SA}\tag{SA}   
x_{k+1}:=x_k-\g_k\nabla \us{F}(x_k,\xi_k), \quad \hbox{for }k\geq 0,
\end{equation}
where \fy{$x_0 \in \Real^n$ is a randomly generated initial point,} $\g_k>0$ denotes the stepsize and $\nabla F(x_k,\xi_k)$ denotes
\fys{a \vvs{sampled} gradient of $f$} with respect to $x$
at $x_k$. \fy{SA} schemes are
	characterized by several disadvantages, including \vus{a} poorer
		rate of convergence (than their deterministic counterparts) and
		the detrimental impact of conditioning on their performance.
		\fys{In deterministic regimes, the BFGS} method, named after Broyden, Fletcher, Goldfarb, and Shanno, is \fys{amongst the \vus{most} popular quasi-Newton methods} \cite{goldfarb70, fletcher70}, \vus{displaying} a \fy{superlinear} convergence rate without requiring \fy{second-order} information. Addressing large scale deterministic problems, the limited-memory variant of the BFGS method, denoted by LBFGS, was developed \fy{and attains} an $\vus{R}$-linear convergence rate under strong convexity of the objective function (see Theorem 6.1 in~\cite{nocedal89}). Recently, there has been a growing interest in applying \fys{stochastic quasi-Newton (SQN) methods} for solving  large-scale optimization and machine learning problems. In these methods, $x_k$ is \fy{updated} by the following rule:
\begin{equation}\label{eqn:SQN}\tag{SQN}
x_{k+1}:=x_k-\g_kH_k\nabla \us{F}(x_k,\xi_k), \quad \hbox{for }k\geq 0,
\end{equation}
where $H_k \succeq 0$ is an approximation of the inverse of the
\fy{Hessian} at iteration $k$ that incorporates the curvature
information of the objective function \us{within} the algorithm. The
convergence of this class of algorithms can be derived under a careful
choice of \fy{$H_k$} and the stepsize sequence \fy{$\{\g_k\}$}. In
particular, \fy{the} boundedness of the eigenvalues of $H_k$ is an
important factor in achieving global convergence in convex and
nonconvex problems (cf. \cite{Fukushima01,Bottou09}). While in
\cite{Schraudolph07} the performance of SQN methods \vus{was found} to be
favorable in solving high dimensional problems, Mokhtari et
al.~\cite{mokh14} considered stochastic optimization problems with
strongly convex objectives and developed a regularized BFGS method
(RES) \an{by updating $H_k$ according to} a modified version
of BFGS update rule to assure convergence. To address large scale
applications, limited-memory variants were employed to ascertain
scalability in terms of the number of
variables~\cite{Mokhtari15,nocedal15}. \vus{Recent extensions have included} a stochastic quasi-Newton method~\cite{wang14} for solving nonconvex stochastic optimization problems \vus{and  a constant stepsize} variance reduced SQN method~\an{\cite{Lucchi15} for smooth strongly convex problems characterized by a linear convergence rate}. \fy{\vus{Finally}, an incremental quasi-Newton (IQN) method with a local superlinear convergence rate has been recently developed
for addressing the sum of a large number of strongly convex functions~\cite{Iqn17}.}

\noindent {\bf Motivation:} \fy{In} \vus{both the full and limited
memory variants of the SQN methods \fy{in the
literature}~\cite{nocedal15, Mokhtari15,Iqn17}}, it is uniformly
assumed that the objective function is strongly convex. This assumption
plays an important role in deriving the rate of convergence of the
algorithm. However, in many applications, the objective function is
convex, but not strongly convex \fys{such as, \vus{when considering}
the logistic regression function.} While \vus{a} lack of strong
convexity might lead to \vvs{slower} convergence \vus{in practice}, no
\vus{rigorous support for} the convergence rate is \vus{currently}
\vvs{available} in the literature of SQN methods.  A \an{simple} remedy
to address this challenge is to regularize the objective function with
the term $\frac{1}{2}\mu\|x\|^2$ and solve the approximate problem of
the form \fys{$\min_{x \in \Real^n} f(x)+\frac{\mu}{2}\|x\|^2$}, where
$\mu>0$ is the regularization parameter. Several challenges arise in
applying this technique. A drawback of this technique is that the
optimal solution to the \fys{regularized problem} is not
\fy{necessarily} an optimal of the original problem
\eqref{eqn:problem}. \fy{\vus{Yet, a}nother challenge arises from} the
choice of $\mu$. While \vus{larger} values of $\mu$ may result in large
deviations from the true optimal \fy{solution(s)}, 
choosing a small $\mu$ leads to a deterioration of \fy{the constant
factor} in the convergence rate of the algorithm. \fy{This issue has
been addressed to some extent with the help of averaging techniques. In
particular, \vus{under mere convexity, most} first-order methods admit convergence rate
guarantees under averaging.  For example, averaging SA schemes achieve
a} rate of $\mathcal{O}\left(\frac{M}{\sqrt{k}}\right)$, where $M$ is
an upper bound on the norm of the subgradient
\an{(see~\cite{nemirovski_robust_2009,Nedic2014}).} \fy{In past few
years, fast incremental gradient methods with improved rates of
convergence have been developed (see~\cite{Saga14,Sag17,Piag16,Iag17}).
Of these, addressing the merely convex case, SAGA with averaging
achieves a sublinear convergence rate
$\mathcal{O}\left(\frac{N}{k}\right)$ \vus{where $N$ denotes the number
of blocks}, while in the presence of strong convexity, non-averaging
variants of SAGA and IAG admit a linear convergence rate assuming that
the function satisfies some smoothness conditions. 

A \vus{crucial concern that plagues} the aforementioned schemes is that
the averaging technique has a detrimental impact on inducing sparsity.
In the case of incremental methods such as SAGA and IAG, despite the
fast convergence speed, the application of these methods is \vus{impaired by} the excessive memory requirements. For standard SAGA and IAG, the
memory requirements are $\mathcal{O}\left(nN\right)$.} \fy{Accordingly,
in this paper, our main goal lies in addressing such shortcomings in
absence of strong convexity and developing a first-order method
equipped with a rate of convergence for the generated non-averaged
iterates}. 

\noindent \textbf{Related research on regularization:} In \fy{optimization}, in order to obtain solutions with desirable properties, it is common to regularize the problem \eqref{eqn:problem} as follows
\begin{equation}\label{eqn:problemReg2}
\min_{x \in \Real^n} f_\mu(x):=f(x)+\mu R(x),
\end{equation} where $R:\Real^n\to \Real$ is a proper convex function
and the scalar $\mu>0$ is the regularization parameter. The properties
of the regularized problem and its relation \vus{to} the original problem
have been investigated by different researchers. Mangasarian and his
colleagues appear among the first researchers who studied exact
regularization of linear and nonlinear programs
\cite{Mangasarian79nonlinearperturbation, Mangasarian91}. A
regularization is said to be exact when an optimal solution of
\eqref{eqn:problemReg2}, is also optimal for problem
\eqref{eqn:problem} if $\mu$ is small enough. Tseng et al.
\cite{Tseng08,Tseng09} established the necessary and sufficient
conditions of exact regularization for convex programs and derived
error bounds for inexact regularized convex problems. \fy{In a similar veing}, exact regularization of variational \fy{inequality problems} has been studied
in \cite{Charitha17}. A challenging question is concerned about the
choice of the regularization parameter $\mu$. A common approach to find
an acceptable value for $\mu$ is through a two-loop scheme where in the
inner loop, problem \eqref{eqn:problemReg2} is solved for a fixed value
of $\mu$, while \fy{$\mu$ is tuned in the outer loop}. The main
drawback of this approach is that\fy{,} in general, there is no
guidance on the tuning rule for $\mu$. In addition, this approach is
computationally inefficient. 
\vus{Furthermore}, tuning rules may result in losing the desired properties of the sample
path of the solutions to regularized problems.  \fys{In
this work, we address this issue through employing} an iterative
single-loop algorithm where we update the regularization parameter
$\mu$ at each iteration of the scheme and reduce it iteratively to
converge to zero~\cite{Farzad3,koshal13}. 

\noindent {\bf Contributions:} We consider stochastic optimization problems with non-strongly convex objective functions and Lipschitz but possibly unbounded gradient mappings. Our main contributions are as follows:\\
(i) \vvs{\bf Asymptotic convergence}:  We develop \fy{an iteratively} regularized SQN method \fy{where} the stepsize, regularization parameter, and the \fy{Hessian inverse approximation denoted by $H_k$} are updated iteratively. We assume that $H_k$ satisfies a set of general assumptions on its eigenvalues and its dependency \vvs{on} the uncertainty. The \vvs{asymptotic} convergence of the \fy{method} is established under a \fy{suitable} choice of \fy{an error} function. \fy{For the sequence of the iterates $\{x_k\}$ produced by the algorithm}, we obtain a set of suitable conditions on the stepsize and regularization sequences for which $f(x_k)$ converges to the optimal objective value, i.e., $f^*$,  of \eqref{eqn:problem} \vvs{both} in an almost sure sense and in \fy{a mean sense}. We also derive an upper bound for \fy{$f(x_k)-f^*$}.\\
(ii) \vvs{\bf Rate of convergence for regularized LBFGS methods}: To
address large scale stochastic optimization problems, motivated by our
earlier work~\cite{FarzadCDC16} on SQN methods for small scale
stochastic optimization problems with non-strongly convex objectives,
\fy{we develop an iteratively regularized stochastic limited-memory
BFGS scheme (see Algorithm \ref{algorithm:IR-S-BFGS})}. We show that
under a careful choice of the \fy{update rules for the stepsize and}
regularization parameter, \fy{Algorithm \ref{algorithm:IR-S-BFGS}}
displays a convergence rate
$\mathcal{O}\left(k^{-\left(\frac{1}{3}-\e\right)}\right)$ in terms of
the objective \fy{function} values, where $\e$ is an arbitrary small
positive scalar. \fy{Similar to standard stochastic LBFGS schemes, the
memory requirement is independent of $N$ and is \vus{$\mathcal{O}\left(mn\right)$}, where $m \ll n$ \vus{denotes} the memory parameter \vus{in the LBFGS scheme}.
In the deterministic case, we} show that the convergence rate
\vvs{improves} to $\mathcal{O}\left(\frac{1}{k^{1-\e}}\right)$.
\fy{Both of these convergence rates appear to be new for the class of
deterministic and stochastic quasi-Newton methods.}\\ \textbf{Outline
of the paper:} The rest of the paper is organized as follows.
Section \ref{sec:alg} presents the general framework of the proposed SQN
algorithm and the sets of main assumptions. In Section \ref{sec:conv}, we \vvs{prove the asymptotic
convergence of the iterates produced by} the scheme in both almost sure and \vvs{a mean sense} and
derive the a general error bound. In Section \ref{sec:BFGS}, we develop \fy{an iteratively} regularized stochastic LBFGS method (\fy{Algorithm \ref{algorithm:IR-S-BFGS}}) and derive \fy{its} convergence rate. The rate
analysis is also provided for the deterministic variant of this scheme.  We \vvs{then}
present the \fy{numerical experiments} performed on a large scale
classification problem in Section \ref{sec:num}. The paper ends with some
concluding remarks in Section \ref{sec:conc}.

 \textbf{Notation:} A vector $x$ is assumed to be a column vector and
$x^T$ denotes its transpose, \an{while} $\|x\|$ {denotes} the Euclidean
vector norm, i.e., $\|x\|=\sqrt{x^Tx}$.  We write \textit{a.s.} as the
abbreviation for ``almost surely''.  \an{For a symmetric matrix $B$, 
$\lambda_{\min}(B)$ and $\lambda_{\max}(B)$ \vus{denotes the
smallest and largest eigenvalue of $B$}, respectively.} We use $\EXP{z}$ to
denote the expectation of a random variable~$z$. A function $f:X
\subset \mathbb{R}^n\rightarrow \mathbb{R}$ is said to be strongly
convex with parameter $\mu>0$, if  $f(y)\geq f(x)+\nabla
f(x)^T(y-x)+\frac{\mu}{2}\|x-y\|^2,$ for any $x,y \in X$.  A mapping
$F:X \subset \mathbb{R}^n\rightarrow \mathbb{R}$ is \vus{said to be} Lipschitz
continuous with parameter $L>0$ if for any $x, y \in X$, we have
$\|F(x)-F(y)\|\leq L\|x-y\|$. For a continuously differentiable
function $f$ with Lipschitz gradients with parameter $L>0$, we have
$f(y)\leq f(x)+\nabla f(x)^T(y-x)+\frac{L}{2}\|x-y\|^2,$ for any $x,y
\in X$. For a vector $x \in \Real^n$ and a \fys{nonempty} set $X
\subset \Real^n$, the Euclidean distance of $x$ from $X$ is denoted by
$dist(x,X)$. We denote the optimal objective value of problem
\eqref{eqn:problem} by $f^*$ and the set of the optimal solutions by
$X^*$.

\section{Outline of the SQN scheme}\label{sec:alg}
In this section, we describe a general SQN scheme for solving problem~\eqref{eqn:problem} and present the main assumptions. Let $x_0\in \Real^n$ be an arbitrary initial point, and $x_k$ be generated by the following recursive rule
\begin{align}\label{eqn:LM-cyclic-reg-BFGS}\tag{IR-SQN}
x_{k+1}:=x_k -\gamma_kH_k\left(\nabla F(x_k,\xi_k)+ \mu_k(x_k-x_0)\right), \quad \hbox{for all } k \geq 0.
\end{align}
Here, $\g_k$ and $\mu_k$ are the steplength and the regularization parameter, respectively. $H_k \in \Real^{n\times n}$ is a matrix that \fy{contains} the curvature information of the objective function. The \eqref{eqn:LM-cyclic-reg-BFGS} scheme can be seen as a regularized variant of \fy{the} classical stochastic SQN \fy{method}. Here we regularize the gradient map by the term $\mu_k(x_k-x_0)$ to \fy{induce} the strong monotonicity property. In \fy{the} absence of strong convexity of $f$, unlike the classical schemes where $\mu_k$ is maintained fixed, we let $\mu_k$ be updated and \fy{decreased} to zero. Throughout, we let $\sF_k$ denote the history of the method up to time $k$, i.e., 
$\sF_k\fy{\ \triangleq \ }\{x_0,\xi_0,\xi_1,\ldots,\xi_{k-1}\}$ for $k\ge 1$, and $\sF_0\triangleq\{x_0\}$.  
\begin{assumption}\label{assum:convex}  Consider problem~\eqref{eqn:problem}. Let the following hold:
\begin{enumerate}
\item [(a)]  \an{The function $f(x)$ is convex over $\Real^n$.}  
\item[(b)] $f(x)$ is continuously differentiable \us{with Lipschitz
	continuous gradients} over $\Real^n$ with parameter $L>0$. 
	\item[(c)] \ The optimal solution set of \fy{the problem} is nonempty. 
\end{enumerate}
\end{assumption}
Next, we state the assumptions on the random variable $\xi$ and the properties of the stochastic estimator of the gradient mapping.  
\begin{assumption}[Random variable $\xi$]\label{assum:main}
\begin{enumerate}
\item[(a)]  \fy{Vectors $\xi_k \in \mathbb{R}^d$ are i.i.d.\ realizations of the random variable $\xi$} for any $k \geq 0$; 
\item[(b)] The stochastic gradient mapping $\nabla F(x, \xi)$ is an unbiased estimator of $\nabla f(x)$, i.e.
$\EXP{\nabla F(x,\xi)}=\nabla f(x)$ for all $x$, and has a bounded variance,
	i.e., there exists a \us{scalar} $\nu>0$ such that $\EXP{\|\nabla F(x,\xi)-\nabla f(x)\|^2} \leq \nu^2,$ for all $x \in \mathbb{R}^n$.
\end{enumerate}
\end{assumption} 
The next assumption {pertains to} the properties of $H_k$.
\begin{assumption}[Conditions on matrix $H_k$]\label{assump:Hk}
Let the following hold for all $k\geq 0$:\\ 
(a) \  The matrix $H_k\in \Real^{n\times n}$ is $\sF_{k}$-measurable, i.e., $\EXP{H_k \mid \sF_{k}}=H_k$. \\
(b) \ Matrix $H_k$ is symmetric and positive definite and satisfies the following condition: There exist positive scalars $\lambda_{\min},\lambda$ and scalar $\alpha\leq 0$ such that $$\lambda_{\min}\mathbf{I} \preceq H_k \preceq \lambda \mu_k^\alpha\mathbf{I},\qquad  \hbox{for all }k\geq 0,$$ \fys{where $\mu_k$ is the regularization parameter in \eqref{eqn:LM-cyclic-reg-BFGS}.}
\end{assumption}
Assumption \ref{assump:Hk} holds for the stochastic gradient method where $H_k$ is the identity matrix, $\lambda_{\min}=\lambda=1$ and $\alpha=0$. In the case of employing an appropriate LBFGS update rule that will be discussed  in Section \ref{sec:BFGS}, the maximum eigenvalue is obtained in terms of  the regularization parameter.

\section{Convergence analysis}\label{sec:conv}
In this section, we present the convergence analysis of the \eqref{eqn:LM-cyclic-reg-BFGS} method. Our discussion starts by some preliminary definitions and properties. After obtaining a recursive error bound for the method in Lemma \ref{lemma:main-ineq}, we show a.s. convergence in Proposition \ref{prop:a.s}, establish convergence in mean, and derive an error bound in Proposition \ref{prop:mean}.
\begin{definition}[Regularized function and gradient mapping]\label{def:regularizedF}
Consider the sequence $\{\mu_k\}$ of positive \an{scalars} and the starting point of the algorithm \eqref{eqn:LM-cyclic-reg-BFGS}, i.e., $x_0$.
The regularized function $f_k$ and its gradient are defined as follows \fys{for all $k\geq 0$}:
\begin{align*}
f_k(x)\triangleq  f(x)+\frac{\mu_k}{2}{\|x-x_0\|^2},\quad \nabla f_k(x)\triangleq\nabla f(x)+\mu_k(x-x_0).
\end{align*}
\end{definition}
In a similar way, we denote the regularized stochastic function $F(x,\xi)$ and its gradient with $F_k$ and $\nabla F_k$ for any $\xi$, respectively. 
\begin{property}[{Properties of a regularized function}]\label{proper:propsfk} We have:
\begin{itemize}
\item [(a)] $f_k$ is strongly convex with a parameter $\mu_k$.
\item [(b)] $f_k$ has Lipschitzian gradients with parameter $L+\mu_k$.
 \item [(c)] $f_k$ has a unique minimizer over $\Real^n$, denoted by $x^*_k$. Moreover, for any $x \in \Real^n$, 
 \[2\mu_k (f_k(x)-f_k(x^*_k)) \leq \|\nabla f_k(x)\|^2\leq 2(L+\mu_k) (f_k(x)-f_k(x^*_k)).\]
\end{itemize}
\end{property}
The existence and uniqueness of $x^*_k$ in Property~\ref{proper:propsfk}(c) 
is due to the strong convexity of the function $f_k$ (see, for example, Section 1.3.2 in~\cite{Polyak87}), 
while the relation for the gradient is known to hold for a 
strongly convex function with a parameter $\mu$ that also has Lipschitz gradients with a parameter 
$L$ (see Lemma 1 on page 23 in~\cite{Polyak87}). 
In the convergence analysis, we make use of the following result, which can be found
in~\cite{Polyak87} \fy{(see Lemma 10 on page 49)}.

\begin{lemma}\label{lemma:probabilistic_bound_polyak}
Let $\{v_k\}$ be a sequence of nonnegative random variables, where 
$\EXP{v_0} < \infty$, and let $\{\a_k\}$ and $\{\beta_k\}$
be deterministic scalar sequences such that:
\begin{align*}
& \EXP{v_{k+1}|v_0,\ldots, v_k} \leq (1-\alpha_k)v_k+\beta_k
\quad a.s. \ \hbox{for all }k\ge0, 
\end{align*}
where $0 \leq \alpha_k \leq 1$, $\beta_k \geq 0$, $\sum_{k=0}^\infty \alpha_k =\infty$, $\sum_{k=0}^\infty \beta_k < \infty$, and $\lim_{k\to\infty}\,\frac{\beta_k}{\alpha_k} = 0.$ Then, $v_k \rightarrow 0$ almost surely.
\end{lemma} 

Throughout, we denote the stochastic error of the gradient estimator by 
\begin{align}
\label{def:wk}w_k\triangleq \nabla F(x_k,\xi_k)-\nabla f(x_k), \an{\quad\hbox{for all }k\ge0.} 
\end{align}
Note that under Assumption \ref{assum:main}, from the definition of $w_k$ in \eqref{def:wk}, we obtain $\EXP{w_k\mid \sF_k}=0$ and $\EXP{\|w_k\|^2\mid \sF_k}\fy{\leq}  \nu^2$. \fy{The following result plays a key role in the convergence and rate analysis of the proposed schemes.}
\begin{lemma}[A recursive error bound]\label{lemma:main-ineq}
Consider the \eqref{eqn:LM-cyclic-reg-BFGS} method and suppose Assumptions \ref{assum:convex}, \ref{assum:main}, and \ref{assump:Hk} hold. Also, assume $\mu_k$ is a non-increasing sequence and let \begin{align}\label{mainLemmaCond}\g_k\mu_k^{2\alpha} \leq \frac{\lambda_{\min}}{\lambda^2(L+\mu_0)},\quad \hbox{for all }k\geq 0.
\end{align}Then, the following inequality holds \fy{for all $k\geq 0$}:
\begin{align}\label{ineq:cond-recursive-F-k}
\EXP{f_{k+1}(x_{k+1})\mid\sF_k}-f^* &\leq (1-\lambda_{\min}\mu_k\g_k)(f_k(x_k)-f^*)+\frac{\lambda_{\min}\mbox{dist}^2(x_0,X^*)}{2}\mu_k^2\g_k\notag \\ &+\frac{ (L+\mu_k)\lambda^2\nu^2}{2}\mu^{2\alpha}_k\g_k^2.
\end{align}
\end{lemma}
\begin{proof}
The Lipschitzian property of $\nabla f_k$ and  
the update rule \eqref{eqn:LM-cyclic-reg-BFGS} imply that 
\begin{align*}
f_k(x_{k+1}) &\leq f_k(x_k)+\nabla f_k(x_k)^T(x_{k+1}-x_k)+\frac{ (L+\mu_k)}{2}\|x_{k+1}-x_k\|^2 \cr
&= f_k(x_k)-\g_k\nabla f_k(x_k)^TH_k\left(\nabla F(x_k,\xi_k) +\mu_k(x_k-x_0)\right)\\&+ \frac{ (L+\mu_k)}{2}\g_k^2\|H_k\left(\nabla F(x_k,\xi_k)+\mu_k(x_k-x_0)\right)\|^2.\end{align*}
\fy{Invoking the definition of the stochastic error $w_k$ in} \eqref{def:wk} \fy{and} Definition~\ref{def:regularizedF}, we obtain
\begin{align}\label{ineq:term1-2}
f_k(x_{k+1}) &\leq f_k(x_k)-\g_k\nabla f_k(x_k)^TH_k(\nabla f(x_k)+w_k +\mu_k(x_k-x_0))\notag\\&+ \frac{ (L+\mu_k)}{2}\g_k^2\|H_k(\nabla f(x_k)+w_k+\mu_k(x_k-x_0))\|^2\\ 
\notag = f_k(x_k)&-\g_k\underbrace{\nabla f_k(x_k)^TH_k(\nabla f_k(x_k)+w_k)}_{\tiny\hbox{Term } 1}+ \frac{ (L+\mu_k)}{2}\g_k^2\underbrace{\|H_k(\nabla f_k(x_k)+w_k)\|^2}_{\tiny\hbox{ Term } 2},\end{align}
where in the last \fy{equation,} we used the definition of $f_k$. Next, we estimate the conditional expectation of Term 1 and 2. From Assumption \ref{assump:Hk}, we have 
\begin{align*}
\hbox{Term }1 &= \nabla f_k(x_k)^TH_k\nabla f_k(x_k)+\nabla f_k(x_k)^TH_kw_k \\&\geq \lambda_{\min}\|\nabla f_k(x_k)\|^2+\nabla f_k(x_k)^TH_kw_k.
\end{align*}
Taking expectations conditioned on $\sF_k$, from the preceding inequality, we obtain
\begin{align}\label{equ:Term1}
\EXP{\hbox{Term } 1\mid\sF_k} &\geq \lambda_{\min}\|\nabla f_k(x_k)\|^2+\EXP{\nabla f_k(x_k)^TH_kw_k\mid\sF_k}\\ \notag&=\lambda_{\min}\|\nabla f_k(x_k)\|^2+\nabla f_k(x_k)^TH_k\EXP{w_k\mid\sF_k} =\lambda_{\min}\|\nabla f_k(x_k)\|^2,\end{align}
where we recall that $\EXP{w_k\mid\sF_k}=0$ and $\EXP{H_k\mid\sF_k}=H_k$. Similarly, in Term 2, invoking Assumption \ref{assump:Hk}(b), we may write
\begin{align*}
\hbox{Term } 2&= (\nabla f_k(x_k)+w_k)^TH_k^2(\nabla f_k(x_k)+w_k)\leq (\lambda\mu^{\alpha}_k)^2\|\nabla f_k(x_k)+w_k\|^2 \\ &=\lambda^2\mu^{2\alpha}_k\left(\|\nabla f_k(x_k)\|^2+\|w_k\|^2+2\nabla f_k(x_k)^Tw_k\right).\end{align*}
Taking conditional expectations from the preceding inequality, and using Assumption \ref{assum:main}, we obtain
\begin{align}\label{equ:Term2}
\EXP{\hbox{Term } 2\mid\sF_k}&\leq\lambda^2\mu^{2\alpha}_k\left(\|\nabla f_k(x_k)\|^2+\EXP{\|w_k\|^2\mid\sF_k}+2\nabla f_k(x_k)^T\EXP{w_k\mid\sF_k}\right)\notag\\ & \leq \lambda^2\mu^{2\alpha}_k\left(\|\nabla f_k(x_k)\|^2+\nu^2\right).\end{align}
Next, taking conditional expectations in \eqref{ineq:term1-2}, and using \eqref{equ:Term1} and \eqref{equ:Term2}, we obtain
\begin{align*}
&\EXP{f_k(x_{k+1})\mid\sF_k} \leq f_k(x_k)-\g_k\lambda_{\min}\|\nabla f_k(x_k)\|^2\notag \\
&+\lambda^2\mu^{2\alpha}_k\frac{ (L+\mu_k)}{2}\g_k^2\left(\|\nabla f_k(x_k)\|^2+\nu^2\right)\notag \\
&\leq f_k(x_k)-\frac{\g_k\lambda_{\min}}{2}\|\nabla f_k(x_k)\|^2\left(2-\frac{\lambda^2\mu_k^{2\alpha}\g_k(L+\mu_k)}{\lambda_{\min}}\right)+\lambda^2\mu^{2\alpha}_k\frac{ (L+\mu_k)}{2}\g_k^2\nu^2.
\end{align*}
From the assumption that $\g_k$ and $\mu_k$ satisfy $\g_k\mu^{2\alpha}_k \leq \frac{\lambda_{\min}}{\lambda^2(L+\mu_0)}$ for any $k \geq 0$ and that $\mu_k$ is non-increasing, we have $\g_k\mu^{2\alpha}_k \leq \frac{\lambda_{\min}}{\lambda^2(L+\mu_k)}$. As a consequence, we get $2-\frac{\lambda^2\mu_k^{2\alpha}\g_k(L+\mu_k)}{\lambda_{\min}}\geq 1$.  Therefore, from \fy{the} preceding inequality, we obtain
\begin{align*}
\EXP{f_k(x_{k+1}) \mid\sF_k}&\leq f_k(x_k)-\frac{\g_k\lambda_{\min}}{2}\|\nabla f_k(x_k)\|^2+\lambda^2\mu^{2\alpha}_k\frac{ (L+\mu_k)}{2}\g_k^2\nu^2.
\end{align*}
Employing Property \ref{proper:propsfk}(c), we have
\begin{align*}
\EXP{f_k(x_{k+1})\mid\sF_k} &\leq f_k(x_k)-\lambda_{\min}\mu_k\g_k(f_k(x_k)-f_k(x^*_k))+\lambda^2\mu^{2\alpha}_k\frac{ (L+\mu_k)}{2}\g_k^2\nu^2.
\end{align*}
Note that, since $\mu_k$ is a non-increasing sequence, Definition \ref{def:regularizedF} implies that \[\EXP{f_{k+1}(x_{k+1})\mid \sF_k} \leq\EXP{f_{k}(x_{k+1})\mid \sF_k}.\] Therefore, we obtain 
\begin{align}\label{ineq:lemmaLastIneq}
\EXP{f_{k+1}(x_{k+1})\mid\sF_k} &\leq f_k(x_k)-\lambda_{\min}\mu_k\g_k(\underbrace{f_k(x_k)-f_k(x^*_k)}_{\tiny{\hbox{Term } 3}})+\lambda^2\mu^{2\alpha}_k\frac{ (L+\mu_k)}{2}\g_k^2\nu^2.
\end{align}
Next, we derive a lower bound for Term 3. Since $x_k^*$ is the unique minimizer of $f_k$, we have $f_k(x_k^*) \leq f_k(x^*)$. Therefore, invoking Definition \ref{def:regularizedF}, for an arbitrary optimal solution $x^* \in X^*$, we have 
\[f_k(x_k)-f_k(x^*_k)\geq f_k(x_k)-f_k(x^*)=f_k(x_k)-f^*-\frac{\mu_k}{2}\|x^*-x_0\|^2.\]
From the preceding relation and \eqref{ineq:lemmaLastIneq}, we have
\begin{align*}
\EXP{f_{k+1}(x_{k+1})\mid\sF_k} &\leq f_k(x_k)-\lambda_{\min}\mu_k\g_k(f_k(x_k)-f^*)+\frac{\lambda_{\min}\|x^*-x_0\|^2}{2}\mu_k^2\g_k\\
&+\frac{(L+\mu_k)\lambda^2\nu^2}{2}\mu^{2\alpha}_k\g_k^2.
\end{align*}
Since $x^*$ is an arbitrary optimal solution, taking minimum from the right-hand side of the preceding inequality over $X^*$, we can replace $\|x^*-x_0\|$ by $\mbox{dist}(x_0,X^*)$. Then, subtracting $f^*$ from both sides of the resulting relation yields the desired inequality.
\end{proof}
Next, we show the convergence of the scheme. {In order to apply Lemma \ref{lemma:probabilistic_bound_polyak} to inequality \eqref{ineq:cond-recursive-F-k} and prove the almost sure convergence, we use the following definitions: 
\begin{align}\label{def:lemma3}
& v_k := f_k(x_k)-f^*, \quad \alpha_k := \lambda_{\min}\g_k\mu_k, \notag\\ &\beta_k := \frac{\lambda_{\min}\mbox{dist}^2(x_0,X^*)}{2}\mu_k^2\g_k+\frac{(L+\mu_k)\lambda^2\nu^2}{2}\mu^{2\alpha}_k\g_k^2.
\end{align}
\an{To satisfy the conditions of Lemma \ref{lemma:probabilistic_bound_polyak}, we 
identify a set of sufficient conditions on $\{\gamma_k\}$ and $\{\mu_k\}$ in the following assumption.}
Later in the subsequent sections, for each class of algorithms, we provide a set of sequences that meet these assumptions.}
\begin{assumption}[\fy{Conditions} on sequences for a.s. convergence]\label{assum:sequences}
Let the sequences $\{\gamma_k\}$ and $\{\mu_k\}$ be positive and satisfy the following conditions:
\begin{flalign*}
\begin{array}{lll}
\vspace{.2in}(a)\  \fy{\lim_{k \to \infty }\g_k\mu_k^{2\alpha-1} =0;}& &(b)\ \{\mu_k\}\hbox{ is non-increasing and }\mu_k \to 0;\\
\vspace{.2in}(c)\  \lambda_{\min}\g_k\mu_k \leq 1\hbox{ for }k \geq 0; && (d)\ \sum_{k=0}^\infty \g_k\mu_k =\infty;\\
(e)\ \sum_{k=0}^\infty \mu_k^2\g_k <\infty;&&(f)\  \sum_{k=0}^\infty \g_k^2\mu_k^{2\alpha} <\infty\fy{.}
\end{array}
\end{flalign*}
\end{assumption}

\begin{proposition}[Almost sure convergence]\label{prop:a.s}
Consider the \eqref{eqn:LM-cyclic-reg-BFGS} scheme. Suppose Assumptions~\ref{assum:convex},
\ref{assum:main}, \ref{assump:Hk}, and~\ref{assum:sequences} hold. Then,
 $\lim_{k \to \infty }f(x_k)= f^*$ almost surely.
\end{proposition}
\begin{proof} First, \fy{note that from Assumption \ref{assum:sequences}(a,b), we have $\lim_{k \to \infty }\g_k\mu_k^{2\alpha} =0$}. Thus, there exists $K\geq 1$ such that for any $k\geq K$, we have $\g_k\mu_k^{2\alpha} \leq \frac{\lambda_{\min}}{\lambda^2(L+\mu_0)}$
implying that \an{condition~\eqref{mainLemmaCond} of Lemma~\ref{lemma:main-ineq} holds for all $k\geq K$.
Hence, relation~\eqref{ineq:cond-recursive-F-k} holds for any $k\geq K$.}
Next, we apply Lemma \ref{lemma:probabilistic_bound_polyak} to prove a.s.\ convergence of 
the \eqref{eqn:LM-cyclic-reg-BFGS} scheme. 
{Consider the definitions in~\eqref{def:lemma3} for any $k \geq K$.}
The non-negativity of $\alpha_k$ and $\beta_k$ is implied by the definition and that $\lambda_{\min}$, $\g_k$ and $\mu_k$ are positive.
From  \eqref{ineq:cond-recursive-F-k}, we have 
\begin{align*}
& \EXP{v_{k+1}\mid \sF_k} \leq (1-\alpha_k)v_k+\beta_k
\quad  \hbox{for all }k\geq K.
\end{align*}
Since $f^* \leq f(x)$ for any $x \in \Real^n$, we can write $v_k= (f(x_k)-f^*)+\frac{\mu_k}{2}\|x_k-x_0\|^2  \geq 0.$
From Assumption \ref{assum:sequences}(c), we obtain $\alpha_k \leq 1$. Also, from Assumption \ref{assum:sequences}(d), we get $\sum_{k=K}^\infty \alpha_k =\infty$. 
\an{Using Assumption~\ref{assum:sequences}(b,e,f) and the definition of $\beta_k$ in~\eqref{def:lemma3}, for an arbitrary solution $x^*$, we \vus{may prove the summability of $\beta_k$ as follows.}} 
\begin{align*}&\sum_{k=K}^\infty \beta_k\leq  \frac{\lambda_{\min}\mbox{dist}^2(x_0,X^*)}{2}\sum_{k=K}^\infty \mu_k^2\g_k+\frac{(L+\mu_0)\lambda^2\nu^2}{2}\sum_{k=K}^\infty\mu^{2\alpha}_k\g_k^2 <\infty.\end{align*}
Similarly, we can write
\begin{align*} \lim_{k\to \infty}\frac{\beta_k}{\alpha_k}& \leq \frac{\mbox{dist}^2(x_0,X^*)}{2}\lim_{k\to \infty}{\mu_k}+ \frac{(L+\mu_0)\lambda^2\nu^2}{2}\lim_{k\to \infty}\mu^{2\alpha-1}_k\g_k =0, \end{align*}
where the last equation is implied \an{by Assumption \ref{assum:sequences}(a,b).}
Therefore, all conditions of Lemma \ref{lemma:probabilistic_bound_polyak} hold \fy{(with an index shift)} and we conclude that 
\an{$v_k:=f_k(x_k)-f^*$ converges to 0 a.s.} 
Let us define $v'_k:= f(x_k)-f^*$ and $v''_k := \frac{\mu_k}{2}\|x_k-x_0\|^2$, \an{ 
so that $v_k=v'_k+v''_k$. Since $v'_k$ and $v''_k$ are nonnegative, and  $v_k \to 0$ a.s., 
 it follows that $v'_k \to 0$ and $v''_k \to 0$ a.s., implying that
$\lim_{k \to \infty }f(x_k)= f^*$ a.s.}
\end{proof}

In the following, our goal is to state the assumptions on the sequences $\{\g_k\}$ and $\{\mu_k\}$ under which we can show the convergence in mean.
\begin{assumption}[\fy{Conditions} on sequences for convergence in mean]\label{assum:sequences-ms-convergence}
Let the sequences $\{\gamma_k\}$ and $\{\mu_k\}$ be positive and satisfy the following conditions:
\begin{itemize}
\item [(a)] $\lim_{k \to \infty }\g_k\mu_k^{2\alpha-1} =0$;
\item [(b)] $\{\mu_k\}\hbox{ is non-increasing and }\mu_k \to 0$;
\item [(c)] $\lambda_{\min}\g_k\mu_{k} \leq 1$ for $k \geq 0$;
\item [(d)] There exist $K_0$ and $0<\beta<1$ such that \begin{align*}\g_{k-1}\mu_{k-1}^{2\alpha-1}\leq \g_{k}\mu_{k}^{2\alpha-1}(1+\beta \lambda_{\min}\g_{k}\mu_{k}),\quad \hbox{for all }k \geq K_0;\end{align*}
\item [(e)] There exists a scalar $\rho>0$ such that $ \mu_k^{2-2\alpha} \leq \rho\g_k$ for all $k \geq 0$\fy{.}
\end{itemize}
\end{assumption}
Next, we use Assumption \ref{assum:sequences-ms-convergence} to establish the convergence in mean.
\begin{proposition}[Convergence in mean]\label{prop:mean}
Consider the \eqref{eqn:LM-cyclic-reg-BFGS} scheme. 
Suppose Assumptions~\ref{assum:convex},
\ref{assum:main}, \ref{assump:Hk}, and \ref{assum:sequences-ms-convergence} hold. 
Then, there exists $K\geq 1$ such that:\begin{align}\label{ineq:bound}
 \EXP{f(x_{k+1})}-f^*\leq \theta\g_k\mu_k^{2\alpha-1}, \quad \hbox{for all } k \geq K,
 \end{align}
  where $f^*$ is the optimal value of problem \fy{and}
\begin{align}\label{def:theta}\theta := \max \bigg\{\frac{\EXP{f_{K+1}(x_{K+1})}-f^*}{\g_K\mu_K^{2\alpha-1}},\frac{\rho\lambda_{\min}\mbox{dist}^2(x_0,X^*)+ (L+\mu_0)\lambda^2\nu^2}{2\fy{\lambda_{\min}}(1-\beta)}\bigg\}.\end{align}
 Moreover, $\lim_{k\to \infty}\EXP{f(x_{k})}=f^*.$
\end{proposition}
\begin{proof}
Note that Assumption \ref{assum:sequences-ms-convergence}(a,b) imply that \eqref{ineq:cond-recursive-F-k} holds for a large enough $k$, say after $\hat K$. Then, since the conditions of Lemma \ref{lemma:main-ineq} are met \fy{(with an index shift)}, taking expectations on both sides of \eqref{ineq:cond-recursive-F-k}, we obtain for any $k\geq \hat K$:
\begin{align*}
\EXP{f_{k+1}(x_{k+1})-f^*} &\leq (1-\lambda_{\min}\mu_k\g_k)\EXP{f_k(x_k)-f^*}+\frac{\lambda_{\min}\mbox{dist}^2(x_0,X^*)}{2}\mu_k^2\g_k\\ &+\frac{ (L+\mu_0)\lambda^2\nu^2}{2}\mu^{2\alpha}_k\g_k^2.
\end{align*}
 Using Assumption \ref{assum:sequences-ms-convergence}(e), we have $\mu_k^2\g_k\leq \rho \g_k^2\mu_k^{2\alpha}$. Thus, we obtain
\begin{align}\label{ineq:cond-recursive-F-k-expected2}
\EXP{f_{k+1}(x_{k+1})-f^*} &\leq (1-\lambda_{\min}\mu_k\g_k)\EXP{f_k(x_k)-f^*}\notag\\ &+\left(\frac{\rho\lambda_{\min}\mbox{dist}^2(x_0,X^*)+ (L+\mu_0)\lambda^2\nu^2}{2}\right)\mu^{2\alpha}_k\g_k^2.
\end{align}
Let us define \fy{$K\triangleq\max\{\hat K, K_0\}$, where $K_0$ is from Assumption \ref{assum:sequences-ms-convergence}(d)}. Using the preceding relation and by induction on $k$, we show the desired result. To show \eqref{ineq:bound}, we show the following relation first: 
\begin{align}\label{ineq:bound_v2}
 \EXP{f_{k+1}(x_{k+1})}-f^*\leq \theta\g_k\mu_k^{2\alpha-1}, \quad \hbox{for all } k \geq K,
 \end{align}
 Note that \eqref{ineq:bound_v2} \fy{implies} the relation \eqref{ineq:bound} since we have $ \EXP{f(x_{k+1})}\leq  \EXP{f_{k+1}(x_{k+1})}$.
First, we show that \eqref{ineq:bound_v2} holds for \fy{$k=K$}.  Consider the term $\EXP{f_{K+1}(x_{K+1})}-f^*$. Multiplying and dividing by $\g_K\mu_K^{2\alpha-1}$, we obtain
\begin{align*}\EXP{f_{K+1}(x_{K+1})}-f^*=\left(\frac{\EXP{f_{K+1}(x_{K+1})}-f^*}{\g_k\mu_K^{2\alpha-1}}\right)\g_K\mu_K^{2\alpha-1}
 \leq \theta \g_K\mu_K^{2\alpha-1},\end{align*}
where the last inequality is obtained by invoking the definition of $\theta$ in \eqref{def:theta}. This implies that \eqref{ineq:bound_v2} holds for $k=K$. Now assume that  \eqref{ineq:bound_v2} holds for some $k \geq K$. We show that it also holds for $k+1$.  From the induction hypothesis and \eqref{ineq:cond-recursive-F-k-expected2} we have
\begin{align*}
\EXP{f_{k+1}(x_{k+1})-f^*} &\leq (1-\lambda_{\min}\mu_k\g_k)\theta\g_{k-1}\mu_{k-1}^{2\alpha-1}\\ &+\left(\frac{\rho\lambda_{\min}\mbox{dist}^2(x_0,X^*)+ (L+\mu_0)\lambda^2\nu^2}{2}\right)\mu^{2\alpha}_k\g_k^2.
\end{align*}
Using Assumption \ref{assum:sequences-ms-convergence}(d) we obtain
\begin{align}\label{ineq:thm2-rel1}
\EXP{f_{k+1}(x_{k+1})-f^*} &\leq \theta\g_{k}\mu_{k}^{2\alpha-1}(1-\lambda_{\min}\mu_k\g_k)(1+\beta \lambda_{\min}\g_{k}\mu_{k})\notag \\ &+\left(\frac{\rho\lambda_{\min}\mbox{dist}^2(x_0,X^*)+ (L+\mu_0)\lambda^2\nu^2}{2}\right)\mu^{2\alpha}_k\g_k^2.
\end{align}
Next we find an upper bound for the term $(1-\lambda_{\min}\mu_k\g_k)(1+\beta \lambda_{\min}\g_{k}\mu_{k}) $ as follows
\begin{align*}
(1-\lambda_{\min}\mu_k\g_k)(1+\beta \lambda_{\min}\g_{k}\mu_{k}) &= 1-\lambda_{\min}\mu_k\g_k+\beta\lambda_{\min}\mu_k\g_k-\fy{\beta\lambda_{\min}^2\mu_k^2\g_k^2}\\ & \leq 1-(1-\beta)\lambda_{\min}\mu_k\g_k.
\end{align*}
Combining this relation with \eqref{ineq:thm2-rel1}, it follows \fy{that}
\begin{align*}
&\EXP{f_{k+1}(x_{k+1})-f^*} \leq \theta\g_{k}\mu_{k}^{2\alpha-1} -\theta\fy{\lambda_{\min}}(1-\beta)\mu_k^{2\alpha}\g_k^2  \\ &+\left(\frac{\rho\lambda_{\min}\mbox{dist}^2(x_0,X^*)+ (L+\mu_0)\lambda^2\nu^2}{2}\right)\mu^{2\alpha}_k\g_k^2\\ & =\theta\g_{k}\mu_{k}^{2\alpha-1}-\left(\underbrace{\theta\fy{\lambda_{min}}(1-\beta)-\frac{\rho\lambda_{\min}\mbox{dist}^2(x_0,X^*)+ (L+\mu_0)\lambda^2\nu^2}{2}}_{\tiny{\hbox{Term }1}}\right)\mu^{2\alpha}_k\g_k^2.
\end{align*}
Note that the definition of $\theta$ in \eqref{def:theta} implies that Term 1 is nonnegative. Therefore, 
\begin{align*}
\EXP{f_{k+1}(x_{k+1})-f^*} &\leq \theta\g_{k}\mu_{k}^{2\alpha-1}.
\end{align*}
Hence, the induction statement holds for $k+1$. We conclude that \eqref{ineq:bound_v2} holds for all $k\geq K$. As a consequence, \eqref{ineq:bound} holds for all $k\geq K$ as well. To complete the proof, we need to show $\lim_{k\to \infty}\EXP{f(x_{k})}=f^*$. This is an immediate result of \eqref{ineq:bound} and Assumption \ref{assum:sequences-ms-convergence}(a).
\end{proof}

\section{Iteratively regularized stochastic and deterministic LBFGS methods}\label{sec:BFGS}
\fy{In this section, our main goal is to develop an efficient update rule for matrix $H_k$ of} the \eqref{eqn:LM-cyclic-reg-BFGS} scheme \fy{and establish a convergence rate result}.  
\subsection{Background}\label{sec:LBFGS-background}
Stochastic gradient methods are known to be sensitive to the choice of
stepsizes. In our \vus{prior} work \cite{Farzad1,Farzad2,FarzadTAC19},
we address this challenge \vus{in part} by developing self-tuned
stepsizes \fy{under the strong convexity assumption}. Another avenue to
enhance the robustness of \vus{this} scheme \vus{lies in incorporating}
curvature information of the objective function. A well-known updating
rule for the matrix $H_k$ that uses the curvature estimates is \vus{the
BFGS update}. The deterministic BFGS \fy{scheme}, achieves a
superlinear convergence rate (cf. Theorem 8.6 \cite{Nocedal2006NO})
outperforming the deterministic gradient/subgradient method. In the
classical deterministic BFGS scheme, the curvature information is
incorporated within the algorithm using two terms: the first term is
the displacement factor $s_k=x_{k+1}-x_k$, while the other is the
change in the gradient mapping, $y_k=\nabla f(x_{k+1})-\nabla f(x_k)$,
where  $\nabla f$ denotes the gradient mapping of the deterministic
objective function \fy{$f$}. To have a well-defined update rule, it is
essential that at each iteration, the \fy{curvature condition}
\fy{$s_k^Ty_k>0$} is satisfied. By maintaining this condition at each
iteration, the positive definiteness of the \fy{approximate Hessian},
denoted by $B_k$, is preserved. The BFGS update rule in deterministic
regime also ensures that $B_k$ satisfies a secant equation given by
$B_{k+1}s_k=y_k$\fy{, which ensures} that the \fy{approximate Hessian}
maps $s_k$ into $y_k$.  

To address optimization problems in \fy{the} stochastic \fy{regime}, a regularized BFGS update rule, namely (RES), \fy{was} developed for problem \eqref{eqn:problem} under the strong convexity assumption~\cite{mokh14}. \fys{In problems with a large dimension (see \cite{nocedal15} for some examples), the implementation of this scheme becomes challenging}. This is because the computation \fy{of} $B_k$ and its inverse become \fy{expensive}. Moreover, at each iteration, a matrix of size $n\times n$ needs to be stored. To address these issues in \fy{large scale} optimization problems, limited-memory variants of stochastic BFGS scheme, denoted by stochastic LBFGS, have been developed \cite{Mokhtari15,nocedal15}. The key idea in LBFGS update rule is that instead of storing the full $n\times n$ matrix at each iteration, a fixed number of vectors of size $n$ are stored and used to update the \fy{approximate Hessian inverse}. 

\subsection{Outline of the stochastic LBFGS scheme}\label{sec:LBFGS-outline}
\fy{The} strong convexity property assumed in \cite{Mokhtari15,nocedal15} plays a key role in developing the LBFGS update rules and establishing the convergence. Note that in \fy{the} absence of strong convexity, the \fy{curvature condition} does not hold. To address this issue, a standard approach is to employ a damped variant of the BFGS update rule \cite{Nocedal2006NO}. A drawback of this class of update rules is that there is no guarantee on the rate statements under such rules. Here, we resolve this issue through employing the properties of the regularized gradient map. This is carried out in \eqref{eqn:LM-cyclic-reg-BFGS} by adding the regularization term $\mu_k(x_k-x_0)$ to the stochastic gradient mapping $\nabla F(x_k,\xi_k)$. \fys{To maintain the curvature condition}, \fys{we consider updating} the matrix $H_k$ and the parameter $\mu_k$ \fy{in alternate steps}. Keeping the regularization parameter constant in one iteration \fy{allows for maintaining the \fy{curvature condition}}. After \fy{updating} $H_k$, in the \fy{subsequent} iteration, we keep this matrix fixed and drop the value of the regularized parameter.  \fy{Accordingly, the update rule for} the regularization parameter $\mu_k$ \fy{is based on the} following general procedure:
\begin{align}\label{eqn:mu-k}
  \begin{cases}
   \mu_{k}\fy{:=}\mu_{k-1},      &   \text{if } k \text{ is odd},\\
    \mu_{k}<\mu_{k-1},  &    \text{otherwise}\fy{.}
  \end{cases}
\end{align}
Note that we allow for updating the stepsize sequence at each iteration.
We construct the update rule in terms of the following two factors defined for any odd \fy{$k\geq 1$:\begin{align}\label{equ:siyi-LBFGS}&s_{\lceil k/2\rceil}:= x_{k}-x_{k-1},\cr 
&y_{\lceil k/2\rceil}:=  \nabla F(x_{k},\xi_{k-1}) -\nabla F(x_{k-1},\xi_{k-1})  +  \fys{\tau}\mu_k^\delta s_{\lceil k/2\rceil},\end{align}}where \fys{$\tau >0$ and} $0<\delta\leq 1$ \fys{are parameters to control} the level of regularization in the matrix $H_k$. Here, $\delta$ only controls the regularization for matrix $H_k$, but not that of the gradient direction. It is assumed that $\delta>0$ to \vus{ensure} that the perturbation term \fy{$\mu_k^\delta s_{\lceil k/2\rceil} \to 0$, as $k\to\infty$}. The update policy for $H_k$ is \fy{defined} as follows: 
\begin{align}\label{eqn:H-k}H_{k}\fy{\triangleq}
  \begin{cases}
    H_{k,m},  &    \text{if } k \text{ is odd}, \\
    H_{k-1},  &   \text{otherwise},
  \end{cases}
\end{align}
where $m<n$ (in the large scale settings, \fys{$m\ll n$}) is \fy{the memory parameter and represents} the number of pairs \fy{$(s_i,y_i)$} to be \fys{stored} to estimate $H_k$. Matrix $H_{k,m}$, for any \fy{odd $k\geq 2m-1$}, is updated using the following recursive formula:
\begin{align}\label{eqn:H-k-m}
H_{k,j}:=\left(\mathbf{I}-\frac{y_is_i^T}{{y_i^Ts_i}}\right)^TH_{k,j-1}\left(\mathbf{I}-\frac{y_is_i^T}{y_i^Ts_i}\right)+\frac{s_is_i^T}{y_i^Ts_i}, \quad 1 \leq j\leq m,
\end{align}
where \fy{$i\triangleq \lceil k/2\rceil-(m-j)$ and we set $H_{k,0}:=\frac{s_{\lceil k/2\rceil}^Ty_{\lceil k/2\rceil}}{y_{\lceil k/2\rceil}^Ty_{\lceil k/2\rceil}}\mathbf{I}$}. Here, at odd iterations, matrix $H_k$ is obtained recursively from $H_{k,0},H_{k,1},\ldots, H_{k,m-1}$. Note that \fy{computation of} $H_k$ at an odd $k$ \fy{needs} $m$ pairs \fy{of} \fy{$(s_i,y_i)$}. More precisely, $H_k$ uses the following curvature information: $\left\{(s_i,y_i)\mid i=\lceil k/2\rceil-m+1, \lceil k/2\rceil-m+2,\ldots,\lceil k/2\rceil\right\}.$ 
For convenience, in the first \fys{$2m-2$} iterations, we \fy{let $H_k$ be the identity matrix}. \fy{This allows for collecting the first set of $m$ pairs $(s_i,y_i)$, where $i=1,2,\ldots,m$, that is used at iteration $k:=2m-1$ to obtain $H_{2m-1}$}. \fys{The main differences between update rule \eqref{eqn:H-k-m} and that of the standard SQN schemes~\cite{Mokhtari15,nocedal15} are as follows:} (i) The first distinction is with respect to the definition of $y_i$ in \eqref{equ:siyi-LBFGS}. Here the term $\fy{\mu_k^\delta s_{\lceil k/2\rceil}}$ \vus{compensates} for the lack of strong monotonicity of the gradient mapping and \vus{aids in} establishing the \fy{curvature condition}. (ii) Second, instead of obtaining the pair $(s_i,y_i)$ at every iteration, we \fy{evaluate these terms only at} odd iterations to allow for updating the regularization parameter satisfying \eqref{eqn:mu-k}. 

\fy{Implementation of this stochastic LBFGS scheme requires computing the term $H_k\nabla F_k(x_k,\xi_k)$ at the $k$th iteration.} This can be performed through a two-loop recursion with $\mathcal{O}\left(mn\right)$ number of operations (see Ch. 7, Pg. 178 in \cite{Nocedal2006NO}). \fy{This will be shown for Algorithm \ref{algorithm:IR-S-BFGS} in Theorem \ref{thm:rate}(b).}

In this section, we consider a stronger variant of Assumption \ref{assum:convex} stated as follows:
\begin{assumption}\label{assum:convex2}  
\begin{enumerate}
\item [(a)]  \an{The function $F(x,\xi)$ is convex over $\Real^n$ for any $\xi \in \Omega$.}  
\item[(b)] For any $\xi \in \Omega$, $F(\cdot,\xi)$ is continuously differentiable \us{with Lipschitz
	continuous gradients} over $\Real^n$ with parameter $L_\xi>0$. Moreover, $L:=\sup_{\xi \in \Omega}L_\xi < \infty$. 
	\item[(c)] \ The optimal solution set $X^*$ of problem \eqref{eqn:problem} is nonempty. 
\end{enumerate}
\end{assumption}
\fy{Next, in Lemma \ref{LBFGS-matrix}, we derive bounds on the eigenvalues of the matrix $H_k$ and show that at iterations where $H_k$ is updated, both the \fy{curvature condition} and the secant equation hold. In \fys{the} proof of Lemma \ref{LBFGS-matrix}, we will make use of the following result.
\begin{lemma}\label{sumProductBounds}
Let $0 < a_1 \leq a_2 \leq \ldots \leq a_n$, and $P$ and $S$ be positive scalars such that \fys{$\sum_{i=1}^n a_i \leq S$ and $\prod_{i=1}^na_i \geq P$. Then, we have 
$a_1 \geq (n-1)!P/S^{n-1}$.}
\end{lemma} 
\begin{proof}
See Appendix \ref{app:sumProductBounds}.
\end{proof}}
\begin{lemma}[Properties of  update rule \eqref{eqn:H-k}-\eqref{eqn:H-k-m}]\label{LBFGS-matrix}
Consider the \eqref{eqn:LM-cyclic-reg-BFGS} method. Let $H_k$ be given by the update rule \eqref{eqn:H-k}-\eqref{eqn:H-k-m}, where $s_i$ and $y_i$ are defined in \eqref{equ:siyi-LBFGS} and $\mu_k$ is updated according to the procedure \eqref{eqn:mu-k}. Let \fy{Assumption \ref{assum:convex2}(a,b)} hold. Then, the following results hold:
\begin{itemize}
\item [(a)]  For any odd \fy{$k \geq 2m-1$}, the \fy{curvature condition} holds, i.e., \fy{$s_{\lceil k/2\rceil}^T{y_{\lceil k/2\rceil}} >0$}.
\item [(b)]  For any odd \fy{$k \geq 2m-1$}, the secant equation holds, i.e., \fy{$H_{k}{y}_{\lceil k/2\rceil}=s_{\lceil k/2\rceil}$}. 
\item [(c)] For any \fy{$k \geq 2m-1$}, $H_k$ satisfies Assumption \ref{assump:Hk} \fy{with} the following values:
\begin{align}\label{equ:valuesForAssumH_k}&\lambda_{\min}=\frac{1}{(m+n)\left(L+\fys{\tau}\mu_0^\delta\right)}, \quad \fy{\lambda= \frac{(m+n)^{n+m-1}\left(L+\fys{\tau}\mu_0^\delta\right)^{n+m-1}}{(n-1)!\fys{\tau^{(n+m)}}}}, \notag\\ & \fy{\hbox{and }\alpha=-\delta(n+m).}
\end{align}
More precisely, $H_k$ is symmetric, $\EXP{H_k\mid\sF_k}=H_k$ and
\begin{align}\label{proof:H_kbounds} \frac{1}{(m+n)\left(L+\fys{\tau}\mu_0^\delta\right)}\mathbf{I} \preceq H_{k} \preceq \frac{(m+n)^{n+m-1}\left(L+\fys{\tau}\mu_0^\delta\right)^{n+m-1}}{(n-1)!\fys{\left(\tau\mu_k^\delta\right)^{(n+m)}}}\ \mathbf{I}.
\end{align}
\end{itemize}
\end{lemma}
\fy{\begin{proof}
See Appendix \ref{app:lemmaLBFGSmatrix}.
\end{proof}}

\fy{In the following two lemmas, we provide update rules for the stepsize and the regularization parameter to ensure \fys{convergence in \vus{both} an a.s. and mean sense} for the proposed LBFGS scheme.}
\begin{lemma}[Feasible tuning sequences for a.s. convergence \fy{(Proposition \ref{prop:a.s})}]\label{lemma:a.s.sequences}
Let the sequences $\g_k$ and $\mu_k$ be given by the following rules:
\begin{align}\label{equ:seq}
 \g_k=\frac{\g_0}{(k+1)^a}, \quad \mu_{k}=\frac{\fy{2^b\mu_0}}{\left(k+\kappa\right)^b},
\end{align}
where $\kappa=2$ if $k$ is even and $\kappa=1$ otherwise, $\g_0$and $\mu_0$ are positive scalars such that $\g_0\mu_0 \leq L(m+n)$, \fy{and $a$, $b$, and $\delta\leq 1$ are positive scalars \fys{satisfying}:
\begin{align*}& \fy{\frac{a}{b}>1+2\delta(n+m)}, \quad a+b \leq 1,\quad a+2b > 1, \quad\hbox{and}\quad a-\delta b(m+n)>0.5.\end{align*}
Then, $\g_k$ and $\mu_k$ satisfy Assumption \ref{assum:sequences} with $\lambda_{\min}$ and $\alpha$ in \eqref{equ:valuesForAssumH_k} and $\mu_k$ satisfies \eqref{eqn:mu-k}.}
\end{lemma}
\fy{\begin{proof}
See Appendix \ref{app:feasibleSeqASconv}.
\end{proof}}

\fy{\begin{remark}[An example \fy{of} feasible sequences]
The conditions on parameters $a$, $b$, $\g_0$, and $\mu_0$ in Lemma \ref{lemma:a.s.sequences} hold for $\g_0=\mu_0\leq\sqrt{L}$, $a=\frac{5}{6}$ and $b=\frac{1}{6}$, and $\delta=\frac{1}{m+n}$. 
\end{remark}}
\begin{lemma}[Feasible tuning sequences for convergence in mean \fy{(Proposition \ref{prop:mean})}]\label{lemma:mean-sequences}
Let the sequences $\g_k$ and $\mu_k$ be given by \eqref{equ:seq},
where $\g_0$ and $\mu_0$ are positive scalars such that $\g_0\mu_0 \leq L(m+n)$. \fy{Let, $0 <\delta \leq 1$ $a>0$ and $b>0$ \fys{satisfying}:
\begin{align*}& \frac{a}{b}>1+2\delta\left(m+n\right), \quad a+b < 1,
\quad \frac{a}{b} \leq
2\left(1+\delta\left(m+n\right)\right).\end{align*}Then, \vus{$\mu_k$
satisfies \eqref{eqn:mu-k}} and $\g_k$ and $\mu_k$ satisfy Assumption
\ref{assum:sequences-ms-convergence} with any arbitrary $0<\beta<1$,
$\rho\triangleq \g_0^{-1}\left(\mu_02^b\right)^{2+2\delta(m+n)}$,
and with $\lambda_{\min}$ and $\alpha$ given by
\eqref{equ:valuesForAssumH_k}.} \end{lemma}
\fy{\begin{proof}
See Appendix \ref{app:feasibleSeqMeanConv}.
\end{proof}}

\subsection{An efficient implementation with rate analysis}\label{sec:LBFGS-rate} 
\fy{\begin{algorithm}
  \caption{Iteratively Regularized Stochastic Limited-memory BFGS}
\label{algorithm:IR-S-BFGS}
    \begin{algorithmic}[1]
    \State\textbf{Input:} LBFGS memory parameter $m\geq 1$, Lipschitzian parameter $L>0$, random initial point $x_0 \in \Real^n$, initial stepsize $\g_0>0$, and initial regularization parameter $\mu_0>0$ such that $\g_0\mu_0 \leq  (m+n) L$, \fys{scalars $0<\epsilon < \frac{1}{3}$, $\delta \in \left(0,\frac{1.5\epsilon}{n+m}\right)$, and $\tau>0$}; 
    \State Set $a:=\frac{2}{3}-\epsilon+\frac{2\delta(n+m)}{3}$, $b:=\frac{1}{3}$;
    \For {$k=0,1,\ldots,$}
    		\State Compute $\g_k :=\frac{\g_0}{(k+1)^a}$ and $\mu_k:=\frac{\mu_02^b}{\left(k+1+\text{mod}(k+1,2)\right)^b}$;
        \State  {Evaluate the stochastic gradient $\nabla F(x_k,\xi_k)$;}
     \If {mod$(k,2)=1$}
    \State  Compute index $i:=\lceil k/2\rceil$;
     \State   Compute vector $s_i:=x_k-x_{k-1}$; 
     \State   Compute vector  $y_i:= \nabla F(x_k,\xi_{k-1})-\nabla F(x_{k-1},\xi_{k-1})+\fys{\tau}\mu_k^\delta s_i$;
      \If {$k> 2m$}
      \State Discard the vector pair $\{s_{i-m},y_{i-m}\}$ from storage;
      \EndIf
   \EndIf
 \If {$k < 2m-1$}
   \State Update solution iterate $x_{k+1}:=x_k-\g_k\left(\nabla F(x_k,\xi_k)+\mu_k(x_k-x_0)\right)$;
     \Else
     \State Initialize Hessian inverse $H_{k,0}:=\frac{s_i^Ty_i}{y_i^Ty_i}\mathbf{I}$;\hspace{.4in} \tikzmark{right}\tikzmark{top}
     \State Initialize $q:=\nabla F(x_k,\xi_k)+\mu_k(x_k-x_0)$;
     \For {$t=i:i-m+1$}
     \State Compute scalar $\alpha_{i-t+1}:=\frac{s_t^Tq}{s_t^Ty_t}$;
     \State Update vector $q:=q-\alpha_{i-t+1}y_t$;
     \EndFor
     \State Initialize vector $r:=H_{k,0}q$;
     \For {$t=i-m+1:i$}
     \State  Update vector $r:=r+\left(\alpha_{i-t+1}-\frac{y_t^Tr}{s_t^Ty_t}\right)s_t$;
     \EndFor\tikzmark{bottom}
             \State Update solution iterate $x_{k+1}:=x_k-\g_kr$;\Comment{\footnotesize{LBFGS update}} 
   
\EndIf
    
    \EndFor
   \end{algorithmic}
   \AddNote{top}{bottom}{right}{\footnotesize{LBFGS two-loop recursion}}
\end{algorithm}}
Algorithm \ref{algorithm:IR-S-BFGS} presents an efficient implementation of the proposed stochastic LBFGS scheme. Note that update rules for the stepsize and regularization parameter are specified in line $\#2$ and  $\#4$. \fys{Before presenting the complexity analysis in Theorem \ref{thm:rate}, we make some comments on the choice of parameter $\tau$.
\begin{remark} As mentioned, the parameter $\tau>0$ in line \#9 in
Algorithm \ref{algorithm:IR-S-BFGS} is used to control the level of the
iterative regularization employed in the computation of matrix $H_k$.
Intuitively, it may seem that a \vus{small} choice for $\tau$ \vus{may} reduce
the distortion caused by the term $\mu_k^\delta s_i$ in approximating
the Hessian inverse, and consequently, improve the performance of the
algorithm. However, this may not be always the case. To see this, first
note that the relation \eqref{proof:H_kbounds} shows the dependency of
eigenvalues of $H_k$ on the choice of $\tau$. Recall that the relation
\eqref{ineq:cond-recursive-F-k} is a key assumption used in the
convergence and rate analysis of the proposed method. It can be seen
that when $\tau \to 0$, the right-hand side of
\eqref{ineq:cond-recursive-F-k} will decrease to zero. This indicates
that a small $\tau$ enforces a small value for the term
$\gamma_0/\mu_0^{2\delta(n+m)}$. For example, assuming a fixed value
for $\mu_0$, this would lead to a small $\g_0$. This may have a
negative impact on the performance of the algorithm. As such, it is not
clear if a small $\tau$ can be always beneficial. A closer look into
this trade off calls for a more detailed analysis of the finite-time
performance of the algorithm, which is not the focus of our current
work and remains as a future direction to our study.  \end{remark}}
In Theorem \ref{thm:rate}(a), we establish the convergence rate of Algorithm \ref{algorithm:IR-S-BFGS}. Moreover, in Theorem \ref{thm:rate}(b), we show that the term $H_k\nabla F_k(x_k,\xi_k)$ is computed efficiently using \fyfy{the LBFGS two-loop recursion} in the algorithm with $\mathcal{O}\left(mn\right)$ complexity per iteration.

\begin{theorem}[\fy{Rate analysis for Algorithm \ref{algorithm:IR-S-BFGS}}]\label{thm:rate}
\fy{Consider Algorithm \ref{algorithm:IR-S-BFGS}. The following statements hold:}
\begin{itemize}
\item [(a)] Suppose Assumptions \ref{assum:main} and \ref{assum:convex2} hold. Then, there exists \fy{$K\geq2m-1$} such that 
\[\EXP{f(x_{k })}-f^*\leq \left(\frac{\theta\g_0}{\left(\mu_0\sqrt[3]{2}\right)^{1-2\alpha}}\right)\frac{1}{k^{\frac{1}{3}-\e}}, \quad \hbox{for all } k > K,\] 
where $\theta$ is given by \eqref{def:theta}, and $\lambda_{\min}$, $\lambda$, and $\alpha$ are given by \eqref{equ:valuesForAssumH_k}. 
\item [(b)] \fy{Let the scheme be at the $k$th iteration where $k\geq 2m-1$. Then, by the end of the LBFGS \fys{two-loop} recursion, i.e., line $\#$26 in Algorithm \ref{algorithm:IR-S-BFGS}, we have 
\begin{align}\label{eqn:r_Lemma}
  r=H_k\left(\nabla F(x_k,\xi_k)+\mu_k(x_k-x_0)\right),
\end{align}
where $H_{k}$ is defined by \eqref{eqn:H-k}.}
\end{itemize}
\end{theorem}
\begin{proof}
\noindent (a) First, we show that the conditions of Proposition~\ref{prop:mean} are satisfied. Assumption \ref{assum:convex} holds as a consequence of Assumption \ref{assum:convex2}. From Lemma \ref{LBFGS-matrix}(c), Assumption \ref{assump:Hk} holds for any \fy{$k \geq 2m-1$} as well. To show that Assumption \ref{assum:sequences-ms-convergence} holds, we apply Lemma \ref{lemma:mean-sequences}. We have 
\[\frac{a}{b}=\frac{\frac{2}{3}-\e+\frac{2\delta(n+m)}{3}}{1/3}=2-3\e+2\delta(n+m)>1+2\delta\left(m+n\right),\]
where we used $\e<\frac{1}{3}$. Moreover, since $\delta < \frac{1.5\e}{n+m}$, we have $a+b=1-\e +\frac{2\delta(n+m)}{3}<1$. Also, from the values of $a$ and $b$ we have $2b\left(1+\delta\left(m+n\right)\right)=a+\e>a.$
Thus, the conditions of Lemma \ref{lemma:mean-sequences} hold. This implies that there exists $K_0>0$ such that for any $k\geq K_0$,  the sequences $\g_k$ and $\mu_k$ satisfy Assumption \ref{assum:sequences-ms-convergence} with any arbitrary $0<\beta<1$ and for $\rho=\g_0^{-1}\left(\mu_02^b\right)^{2+2\delta(m+n)}$, and with $\lambda_{\min}$, $\lambda$ and $\alpha$ given by \eqref{equ:valuesForAssumH_k}. Let us define $K:=\max\{K_0,2m-1\}$. Since all conditions in Proposition~\ref{prop:mean} are met, \an{from \eqref{ineq:bound}, \eqref{equ:seq}, and by substituting values of $a$, $b$, and $\alpha$, for any $k \geq K$ we obtain}
\begin{align*}\EXP{f(x_{k+1})}-f^*&\leq \theta\g_{k-1}\mu_{k-1}^{2\alpha-1}= \frac{\theta\g_0(k+\kappa-1)^{(1-2\alpha)/3}}{\left(\mu_0\sqrt[3]{2}\right)^{1-2\alpha}k^{2/3-\e-\frac{2}{3}\alpha}}\\ &\leq  \frac{\theta\g_0(k+1)^{(1-2\alpha)/3}}{\left(\mu_0\sqrt[3]{2}\right)^{1-2\alpha}(k+1)^{2/3-\e-\frac{2}{3}\alpha}}=\left(\frac{\theta\g_0}{\left(\mu_0\sqrt[3]{2}\right)^{1-2\alpha}}\right)\frac{1}{(k+1)^{\frac{1}{3}-\e}}.\end{align*}
Through a change of variable from $k+1$ to $k$, we conclude the result. \\

\fy{\noindent (b)  To show \eqref{eqn:r_Lemma}, it suffices to show that
 \begin{align*}
  r=\begin{cases}
    H_{k,m}\left(\nabla F(x_k,\xi_k)+\mu_k(x_k-x_0)\right),  &    \text{if } k \text{ is odd}, \\
    H_{k-1,m}\left(\nabla F(x_k,\xi_k)+\mu_k(x_k-x_0)\right),  &   \text{otherwise},
  \end{cases}
\end{align*}
where $H_{k,m}$ is defined by the recursion \eqref{eqn:H-k-m} for an odd $k$. First, consider the case that $k\geq 2m-1$ is an odd number. As such, at the $k$th iteration, from line $\#$7, we have $i:=\lceil k/2 \rceil$. For clarity of the presentation, throughout \vus{this} proof, we use $K$ (instead of $i$), i.e., $K\triangleq \lceil k/2 \rceil$ \vus{and} $q_{K-t+1}$ \vus{is used to} denote the value of the vector $q \in \mathbb{R}^n$ after being updated at iteration $t$ in line $\#$21. Similarly, we use $r_{t-K+m}$ to denote the value of the vector $r \in \mathbb{R}^n$ after being updated at iteration $t$ in line $\#$25. Also, we use the following definitions:
\begin{align*}
& q_0\triangleq \nabla F(x_k,\xi_k)+\mu_k(x_k-x_0), \quad r_0\triangleq H_{k,0}q_{m},\cr 
&\rho_j\triangleq \frac{1}{y_j^Ts_j}, \hbox{ and }V_j\triangleq \mathbf{I}-\rho_{j}y_js_j^T, \quad \hbox{for all } j=K-(m-1),\ldots, K.
\end{align*}
Consider relation \eqref{eqn:H-k-m}. By applying this recursive relation repeatedly, we obtain
\begin{align}\label{equ:H_kClosedForm}
H_{k,m} &=\left(\prod_{j=1}^{m}V_{K-(m-j)}\right)^TH_{k,0}\left(\prod_{j=1}^{m}V_{K-(m-j)}\right)\\
&+ \rho_{K-m+1}\left(\prod_{j=2}^{m}V_{K-(m-j)}\right)^Ts_{K-m+1}s^T_{K-m+1}\left(\prod_{j=2}^{m}V_{K-(m-j)}\right) \notag\\
& + \rho_{K-m+2}\left(\prod_{j=3}^{m}V_{K-(m-j)}\right)^Ts_{K-m+2}s^T_{K-m+2}\left(\prod_{j=3}^{m}V_{K-(m-j)}\right)\notag\\
& + \ldots\notag\\
& +\rho_{K-1}V^T_{K}s_{K-1}s_{K-1}^TV_{K}\notag\\
&+\rho_{K}s_{K}s_{K}^T.\notag
\end{align}
Next, we derive a formula for $q_t$. From lines $\#$20-21 in the algorithm, we have
\begin{align*}
q_{K-t+1}&=q_{K-t} -\alpha_{K-t+1}y_t =q_{K-t} - \rho_t\left(s_t^Tq_{K-t}\right)y_t = q_{K-t} - \rho_t\left(y_ts_t^T\right)q_{K-t} \\
&= \left( \mathbf{I}-\rho_{t}y_ts_t^T\right)q_{K-t}=V_tq_{K-t}, \quad \hbox{for all } t= K, K-1, \ldots, K-m+1.
\end{align*}
From the preceding relation, we obtain
\begin{align}\label{equ:q_ell}
q_{\ell}=\left(\prod_{j=m-\ell+1}^{m}V_{K-(m-j)}\right)q_0, \quad \hbox{for all } \ell= 1,2, \ldots, m.
\end{align}
From the update rule for $\alpha_{i-t+1}$ in line $\#$20, using the definition of $\rho_t$, and applying the previous relation, we have $\alpha_1=\rho_{K}s_K^Tq_0$ and 
\begin{align}\label{equ:alpha_ell}
\alpha_{\ell}=\rho_{K-\ell+1}s_{K-\ell+1}^T\left(\prod_{j=m-\ell+2}^{m}V_{K-(m-j)}\right)q_0, \quad \hbox{for all } \ell= 2,3, \ldots, m.
\end{align}
\fys{Multiplying both sides of} \eqref{equ:H_kClosedForm} by $q_0$ and employing \eqref{equ:q_ell} and \eqref{equ:alpha_ell}, we obtain
\begin{align}\label{equ:H_kq_0}
H_{k,m}q_0 &=\left(\prod_{j=1}^{m}V_{K-(m-j)}\right)^TH_{k,0}q_m + \left(\prod_{j=2}^{m}V_{K-(m-j)}\right)^Ts_{K-m+1}\alpha_m \\
& + \left(\prod_{j=3}^{m}V_{K-(m-j)}\right)^Ts_{K-m+2}\alpha_{m-1}  +\ldots+V^T_{K}s_{K-1}\alpha_2+s_{K}\alpha_1.\notag
\end{align}
Next, we derive a formula for $r_t$. From line $\#$25 in the algorithm, we have 
\begin{align*}
r_{t-K+m}&=r_{t-K+m-1}+\left(\alpha_{K-t+1}-\rho_ty_t^Tr_{t-K+m-1}\right)s_t \\
&= r_{t-K+m-1}-\rho_ts_t y_t^Tr_{t-K+m-1}+\alpha_{K-t+1}s_t\\
& = V_t^Tr_{t-K+m-1}+\alpha_{K-t+1}s_t, \quad \hbox{for all } t= K-m+1,\ldots,K-1,  K.
\end{align*}
Combining the preceding two relations, we obtain
\begin{align*}
r_{\ell}&= V_{K-(m-\ell)}^Tr_{\ell-1}+\alpha_{m-\ell+1}s_{K-(m-\ell)}, \quad \hbox{for all } \ell= 1,2,\ldots,m.
\end{align*}
Using the preceding equation repeatedly, we obtain
\begin{align}\label{equ:r_m}
r_m &=\left(\prod_{j=1}^{m}V_{K-(m-j)}\right)^Tr_0 +\alpha_m \left(\prod_{j=2}^{m}V_{K-(m-j)}\right)^Ts_{K-m+1} \\
& + \alpha_{m-1} \left(\prod_{j=3}^{m}V_{K-(m-j)}\right)^Ts_{K-m+2} +\ldots+\alpha_2V^T_{K}s_{K-1}+\alpha_1s_{K}.\notag
\end{align}
From \eqref{equ:r_m} and \eqref{equ:H_kq_0}, and the definition of  $r_0$, we obtain $r_m=H_{k,m}q_0$. Taking to account the definition of $q_0$, the desired result holds for any odd $k\geq 2m-1$. Now, consider the case where $k\geq 2m-1$ is an even number. This implies that the ``if'' condition in line $\#$6 is skipped and as such, the value of $i$ is not updated from the iteration $k-1$, i.e.,  $i= \lceil (k-1)/2 \rceil$. Consequently, the LBFGS two-loop recursion at an even $k$ uses the following pairs 
$$\left\{(s_\ell,y_\ell)\mid \ell=\lceil (k-1)/2\rceil-m+1, \lceil (k-1)/2\rceil-m+2,\ldots,\lceil (k-1)/2\rceil\right\}.$$
\fys{Now,} considering the definition \eqref{eqn:H-k-m} for $k-1$, the desired relation can be shown following the same steps discussed for the case the iteration number is odd.}
\end{proof}
\subsection{Analysis \vus{of} the deterministic case}\label{sec:LBFGS-deterministic}
Our goal in the remainder of this section lies in establishing the convergence and rate statement for the deterministic LBFGS scheme. Consider the following regularized deterministic LBFGS method:
\begin{align}\label{eqn:R-L-BFGS}\tag{IR-LBFGS}
x_{k+1}:=x_k -\gamma_kH_k\left(\nabla f(x_k)+ \mu_k \left(x_k-x_0\right)\right), \quad \hbox{for all } k \geq 0, 
\end{align}
where $H_k$ is given by the update rule \eqref{eqn:H-k}, $\mu_k$ is updated according to \eqref{eqn:mu-k}\fy{, and for an odd $k \geq 1$ we set} \begin{align}\label{equ:siyi-DLBFGS}&s_{\lceil k/2\rceil}:= x_{k}-x_{k-1},\cr 
&{y_{\lceil k/2\rceil}}:=  \nabla f(x_{k}) -\nabla f(x_{k-1})  +  \fys{\tau}\mu_k^\delta s_{\lceil k/2\rceil}.\end{align} 
\begin{theorem}[Convergence and rate analysis of \fy{iteratively} regularized deterministic LBFGS method]
Let $x_k$ be generated by the \ref{eqn:R-L-BFGS} method. Suppose Assumption \ref{assum:convex} holds. Let $\lambda_{\min}$, $\lambda$ and $\alpha$ be given by \eqref{equ:valuesForAssumH_k}. Then \vus{the following hold.} 

\noindent (a) Let $\mu_k$ satisfies \eqref{eqn:mu-k}. If $\g_k$ and $\mu_k$ satisfy the following relation: \begin{align}\label{DLBFGScond}\g_k\mu_k^{2\alpha} \leq \frac{\lambda_{\min}}{\lambda^2(L+\mu_0)},\quad \hbox{for all }k\geq 0,
\end{align}then, for any $k \geq 0$, we have
\begin{align}\label{ineq:DLBFGSbound}
f_{k+1}(x_{k+1})-f^* &\leq (1-\lambda_{\min}\mu_k\g_k)(f_k(x_k)-f^*)+\frac{\lambda_{\min}\mbox{dist}^2(x_0,X^*)}{2}\mu_k^2\g_k.
\end{align}
\noindent (b) Let $\g_k$ and $\mu_k$ be given by the update rule \eqref{equ:seq} where \fy{$a, b>0$ and $0<\delta \leq 1$ satisfy
\[\frac{a}{b}>2\delta(n+m) \quad a+b\leq 1, \quad a+2b>1.\]Then, $\lim_{k\to \infty}f(x_k)=f^*$. Specifically, for $a=\frac{4}{5}$, $b=\frac{1}{5}$, and $\delta=\frac{1}{m+n}$, this result holds.}

\noindent (c)  Let $\e \in (0,1)$ be an arbitrary small scalar. Let $\g_k$ and $\mu_k$ be given by the update rule \eqref{equ:seq} where $a=\e$, $b=1-\e$. Also, assume \fy{$\delta \in \left(0,\frac{\e}{2(n+m)(1-\e)}\right)$.} Let $\g_0$ and $\mu_0$ satisfy the following condition:
\begin{align}\label{DLBFGScondmu0gamma0}\g_0\mu_0\geq (n+m)\left(L+\fys{\tau}\mu_0^\delta\right).\end{align}
Then, there exists $K$ such that  
\begin{align}\label{rateRDLBGFS}
f(x_k)-f^*\leq \frac{\Gamma}{(k+1)^{1-\e}}, \quad \hbox{for all } k \geq K,
\end{align}
where $\Gamma \triangleq \max\bigg\{(K+1)^{1-\e}\left(f_K(x_K)-f^*\right), \frac{\lambda_{\min}\g_0\mu_0^2{\mbox{dist}}^2(x_0,X^*)}{4^{a}(\lambda_{\min}\g_0\mu_0-b)} \bigg\}.$ 
\end{theorem}
\begin{proof}
\noindent (a)\  The conditions of Lemma \ref{LBFGS-matrix} are met indicating that Assumption \ref{LBFGS-matrix} holds. Assumption \ref{assum:main} is clearly met with $\nu=0$ as the problem is deterministic. Therefore, all of the conditions of Lemma \ref{lemma:main-ineq} are satisfied and thus \eqref{ineq:cond-recursive-F-k} holds. Substituting $\nu=0$ in \eqref{ineq:cond-recursive-F-k} and eliminating the expectation operator yields the desired inequality.\\
\noindent (b)\  First, we show that \eqref{ineq:DLBFGSbound} holds. We can write  \begin{align*} \g_k\mu_k^{2\alpha} & =\frac{\g_0}{(2^b\mu_0)^{-2\alpha}}(k+1)^{-a}(k+\kappa)^{-2\alpha b}\leq  \frac{\g_0}{(2^b\mu_0)^{-2\alpha}}(k+1)^{-a-2\alpha b}.\end{align*}
\fy{Note that the assumption that $a > 2b\delta(n+m)$, implies that $-a-2\alpha b<0$}. Therefore, $\g_k\mu_k^{2\alpha}\to 0$ showing that there exists $K_0$ such that for any $k\geq K_0$, \eqref{ineq:DLBFGSbound} holds. We apply Lemma \ref{lemma:probabilistic_bound_polyak} to the inequality \eqref{ineq:DLBFGSbound} by setting 
\[\alpha_k:=\lambda_{\min}\g_0\mu_0, \quad \beta_k:=\frac{\lambda_{\min}\mbox{dist}^2(x_0,X^*)}{2}\mu_k^2\g_k, \quad v_k:=f_k(x_k)-f^*.\] 
\fy{From} $a+b\leq 1$, we have $\sum_{k=0}^\infty \alpha_k =\infty$. Also, $a+2b >1$ indicates that $\sum_{k=0}^\infty \beta_k <\infty$. Since all conditions of Lemma \ref{lemma:probabilistic_bound_polyak} are met, we have $f_k(x_k)\to f^*$. Recalling Definition \ref{def:regularizedF}, this implies that $f(x_k)\to f^*$.\\
\noindent (c)\ First, we show that by the given update rules for $\g_k$ and $\mu_k$, relation \eqref{ineq:DLBFGSbound} holds. Note that \fy{$\alpha=-\delta(n+m)$}. Therefore, we can write \fy{
\begin{align}\label{RDBFGSeps}\g_k\mu_k^{2\alpha}&=\frac{\g_0(k+\kappa)^{2(m+n)\delta b}}{\left(\mu_02^b\right)^{2(m+n)\delta }(k+1)^a}\leq \frac{\g_0(k+2)^{2(m+n)\delta b}}{\left(\mu_02^b\right)^{2(m+n)\delta  }(k+1)^a}\notag \\&= \frac{\g_0(1+\frac{1}{k+1})^{2(m+n)\delta b}}{\left(\mu_02^b\right)^{2(m+n)\delta}(k+1)^{a-2(m+n)\delta b}}.\end{align}
Using the condition on $\delta$, we have $a-2(m+n)b\delta= \e-2(1-\e)\delta(m+n)>0.$}
Thus, relation \eqref{RDBFGSeps} indicates that there exists $K_1$ such that for any $k\geq K_1$, \eqref{ineq:DLBFGSbound} holds. Besides, since $a$ and $b$ are positive, there exits $K_2$ such that for any $k\geq K_2$, we have $(1-\lambda_{\min}\g_k\mu_k)>0$. Let us now define \fy{$K:=\max\{K_1,K_2,2m-1\}$}.
Next, we use induction on $k$ to show \eqref{rateRDLBGFS}.
 For $k=K$, it clearly holds. Let us assume \eqref{rateRDLBGFS} holds for $k>K$. Let $e_k$ denote $f_k(x_k)-f^*$. From \eqref{ineq:DLBFGSbound} and the update rules of $\g_k$ and $\mu_k$ we can write
\begin{align}\label{ineq:DLBFGSbound2}
e_k & \leq \left(1-\frac{\lambda_{\min}\g_0\mu_02^b}{k^a(k+\kappa-1)^b}\right)e_{k-1}+\frac{\lambda_{\min}\mbox{dist}^2(x_0,X^*)\g_0\mu_0^22^{2b-1}}{k^a(k+\kappa-1)^{2b}} \notag\\
& \leq  \left(1-\frac{\lambda_{\min}\g_0\mu_02^b}{k^a(k+1)^b}\right)e_{k-1}+\frac{\lambda_{\min}\mbox{dist}^2(x_0,X^*)\g_0\mu_0^22^{2b-1}}{k^{a+2b}}\notag\\
& \leq  \left(1-\frac{\lambda_{\min}\g_0\mu_0}{k}\right)e_{k-1}+\frac{\lambda_{\min}\mbox{dist}^2(x_0,X^*)\g_0\mu_0^22^{2b-1}}{k^{a+2b}},
\end{align}
where $\kappa$ is defined in \eqref{equ:seq}, and the last inequality is implied by $\frac{k^a(k+1)^b}{k^{a+b}}\leq 2^b$ for $k\geq 1$. Note that since $k\geq K_2$, the term $ \left(1-\frac{\lambda_{\min}\g_0\mu_0}{k^{a+b}}\right)$ in \eqref{ineq:DLBFGSbound2} is nonnegative. Therefore, we can replace $e_{k-1}$ by its upper bound \fys{$\frac{\Gamma}{k^b}$} in \eqref{ineq:DLBFGSbound2}. Doing so  and noticing that $a+b=1$, we obtain
\begin{align}\label{ineq:DLBFGSbound3}
e_k & \leq   \left(1-\frac{C_1}{k}\right)\frac{\Gamma}{k^b}+\frac{C_2}{k^{b+1}},
\end{align}
where we \fys{define} $C_1\fys{\triangleq}\lambda_{\min}\g_0\mu_0$ and  $ C_2\fys{\triangleq} \lambda_{\min}\mbox{dist}^2(x_0,X^*)\g_0\mu_0^22^{2b-1}.$ Using \eqref{ineq:DLBFGSbound3}, to show that $e_{k}\leq \frac{\Gamma}{(k+1)^{1-a}}$, it is enough to show that 
\[\Gamma\left(\frac{1}{k^b}-\frac{1}{(k+1)^b}\right)\leq \frac{C_1\Gamma -C_2}{k^{b+1}}.\]
Rearranging the terms, we need to verify that $\Gamma \geq \frac{C_2}{C_1-C_3}$ and $C_3<C_1$, where \fys{$C_3$ is an upper bound on} $\sup_{k\geq1}\bigg\{k^{b+1}\left(\frac{1}{k^b}-\frac{1}{(k+1)^b}\right)\bigg\}$.
 We claim that \fys{$C_3:=b$ is a feasible choice}. To prove this, we need to show that $k^{b+1}\left(\frac{1}{k^b}-\frac{1}{(k+1)^b}\right)\leq b$, \fys{or equivalently,}
\[\left(1-\frac{1}{k+1}\right)^b\geq 1-\frac{b}{k}, \quad \hbox{for all } k\geq 1.\]
Consider the function $g(x):=(1-\frac{1}{1+x})^b+\frac{b}{x}-1$ for $x\geq 1$. we have 
$$g'(x)=\frac{b}{(1+x)^2}\left(1-\frac{1}{1+x}\right)^{1-b}-\frac{b}{x^2}=\frac{b}{(1+x)^2}\left(\left(\frac{x+1}{x}\right)^{1-b}-\left(\frac{x+1}{x}\right)^{2}\right)\leq 0,$$
due to $0<b<1$. Hence, $g$ is non-increasing implying that it suffices to show $g(1)\geq0$, i.e., $2^b(1-b)\leq 1$.  Let us define $h(x):=2^x(1-x)$ for $0<x<1$. We have $h'(x)=2^x\left(\ln(2)(1-x)-1\right)$. This indicates that $h'(x)<0$ over $x\in(0,1)$, implying that $h(b)\leq h(0)=1$. Hence, we conclude that $C_3:=b$ \fys{is a feasible choice}. To show that $C_3<C_1$ holds, we need to verify that $C_1>b$. This is true due to \eqref{DLBFGScondmu0gamma0}. To complete the proof we need to show $\Gamma \geq \frac{C_2}{C_1-b}$. This holds from the definition of $\Gamma$. 
\end{proof}
\section{Numerical experiments}\label{sec:num}
In this section, we present the implementation results of \fy{Algorithm \ref{algorithm:IR-S-BFGS}} \vus{on} a \fys{classification} application. The Reuters Corpus Volume I (RCV1) data set \cite{lewis2004rcv1} is a collection of news-wire stories produced by Reuters. After the tokenization process, each article is converted to a sparse binary vector, in that 1 denotes the existence and 0 denotes nonexistence of a token in the corresponding article. We consider a subset of the data with 
\fy{$N=100,000$} articles and $n=138,921$ tokens. The articles are categorized into \fys{different} hierarchical groups. Here we focus our attention on the binary classification of the articles with respect to the \fy{Markets} class. We consider the logistic regression loss minimization problem given as follows:
\begin{align}\label{logistic}\tag{LRM}
\min_{x \in \Real^n} f(x):=\frac{1}{N}\sum_{i=1}^N\ln \left(1+\exp\left(-u_i^Txv_i\right)\right),
\end{align}
where $u_i \in \Real^n$ is the input binary vector associated with article $i$ \fys{and} $v_i \in \{-1,1\}$ \fys{denotes} the class of the $i$th article. We run three experiments. Of these, in Section \ref{sec:num-reg}, we compare the performance of Algorithm \ref{algorithm:IR-S-BFGS} \fys{(with $\tau=1$)} with that of standard SQN methods. \fys{In Sections \ref{sec:num-SAGA} and \ref{sec:num-IAG}, we provide comparisons of Algorithm \ref{algorithm:IR-S-BFGS} with SAGA~\cite{Saga14} and with IAG~\cite{Iag17} applied to regularized problems, respectively.}

\begin{table}[t]
\setlength{\tabcolsep}{0pt}
\centering{
 \begin{tabular}{m{1cm}  || c  c  c}
$\gamma_0$ & $\mu_0=1$ & $\mu_0=0.5$ & $\mu_0=0.1$ \\ \hline\\
10
&
\begin{minipage}{.3\textwidth}
\includegraphics[scale=.18, angle=0]{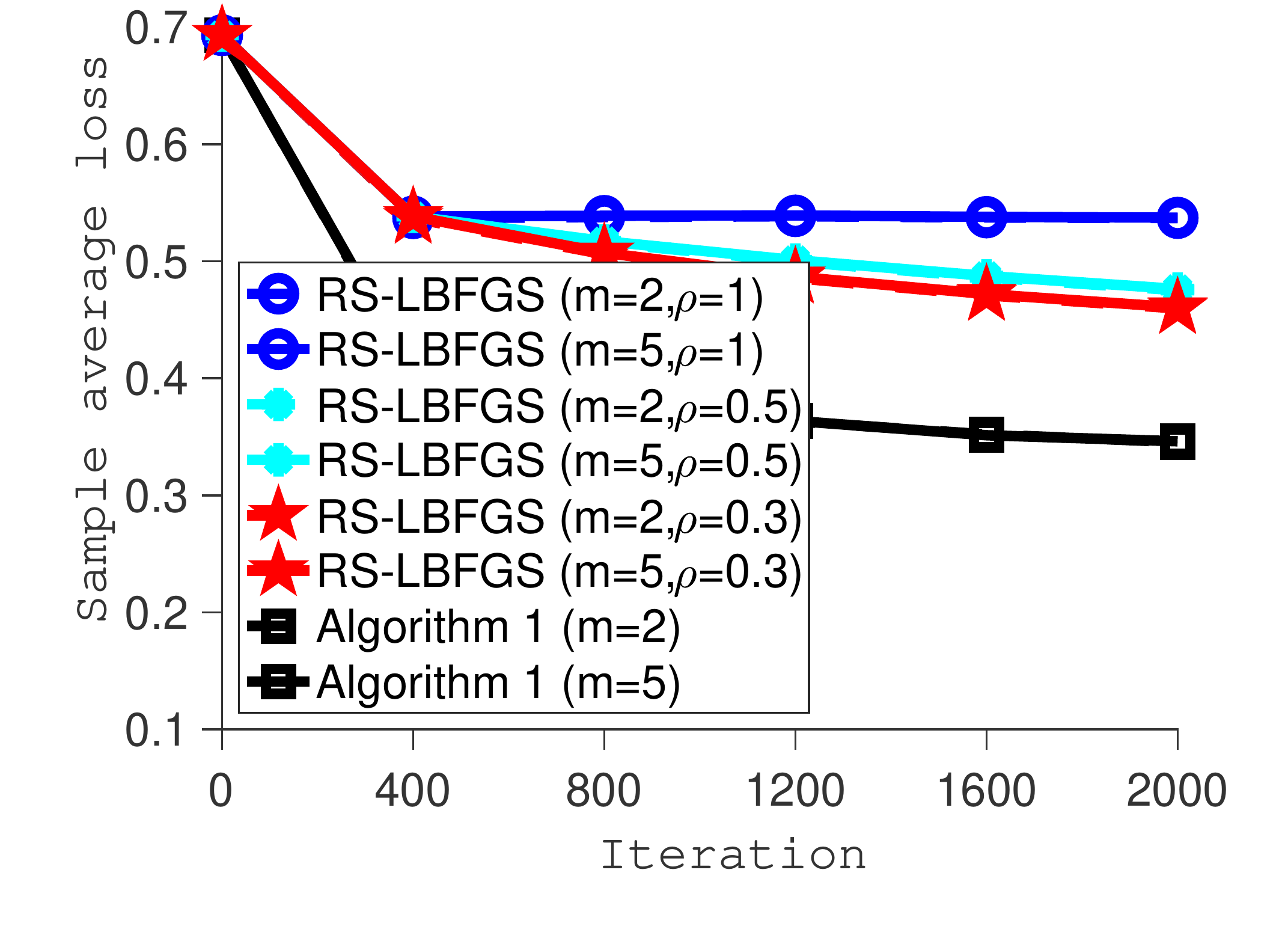}
\end{minipage}
&
\begin{minipage}{.3\textwidth}
\includegraphics[scale=.18, angle=0]{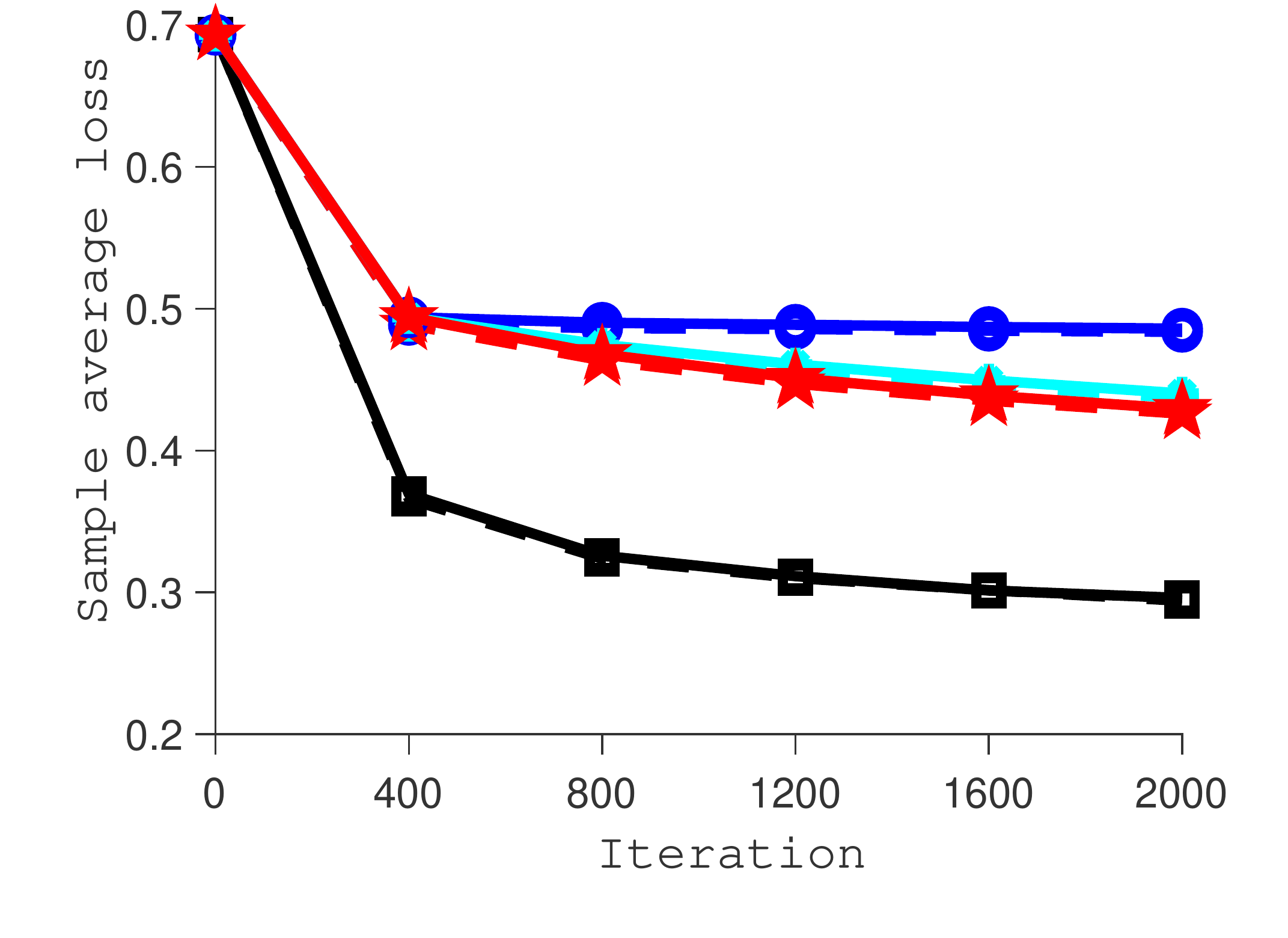}
\end{minipage}
	&
\begin{minipage}{.3\textwidth}
\includegraphics[scale=.18, angle=0]{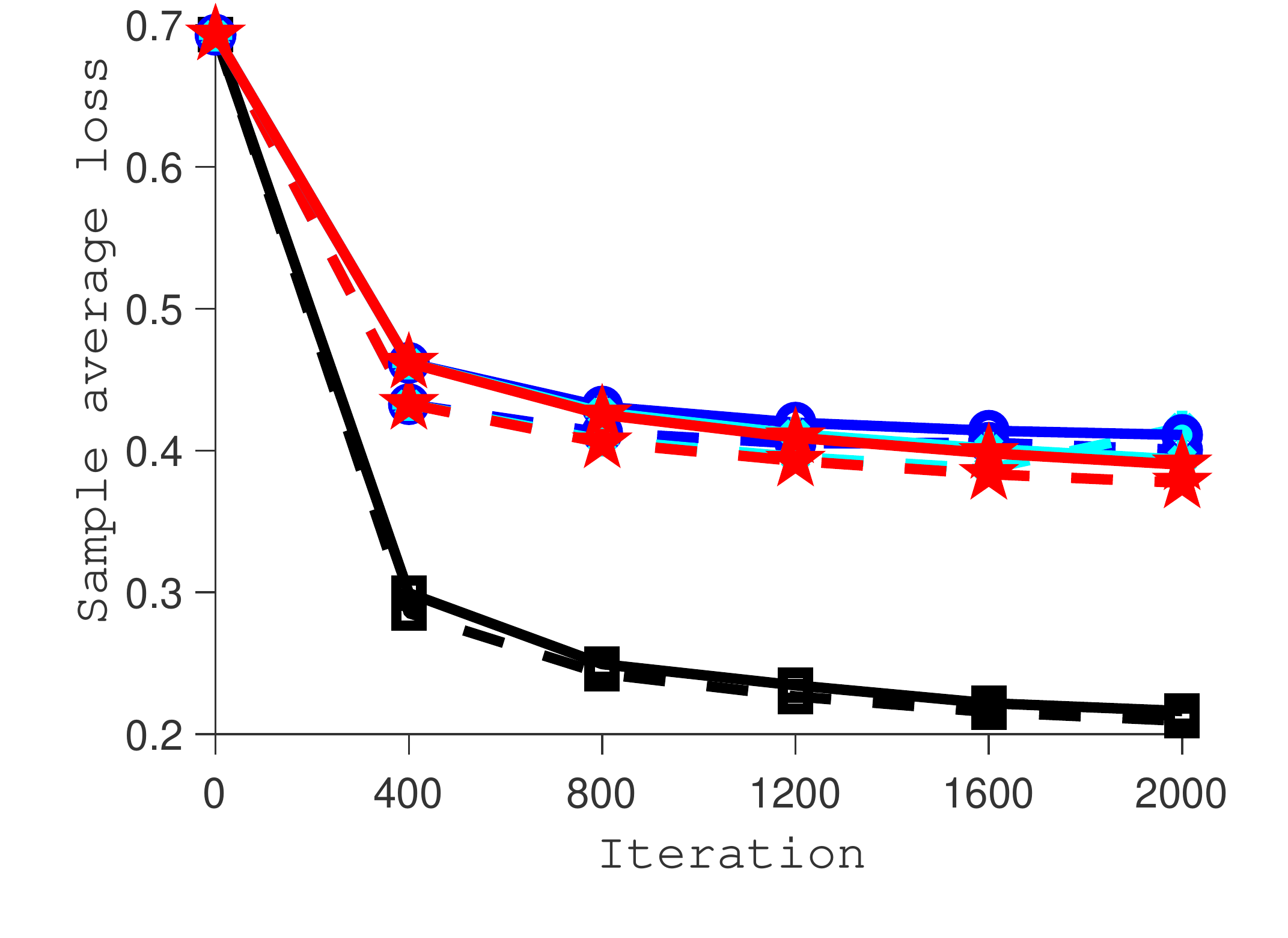}
\end{minipage}
\\
0.5
&
\begin{minipage}{.3\textwidth}
\includegraphics[scale=.18, angle=0]{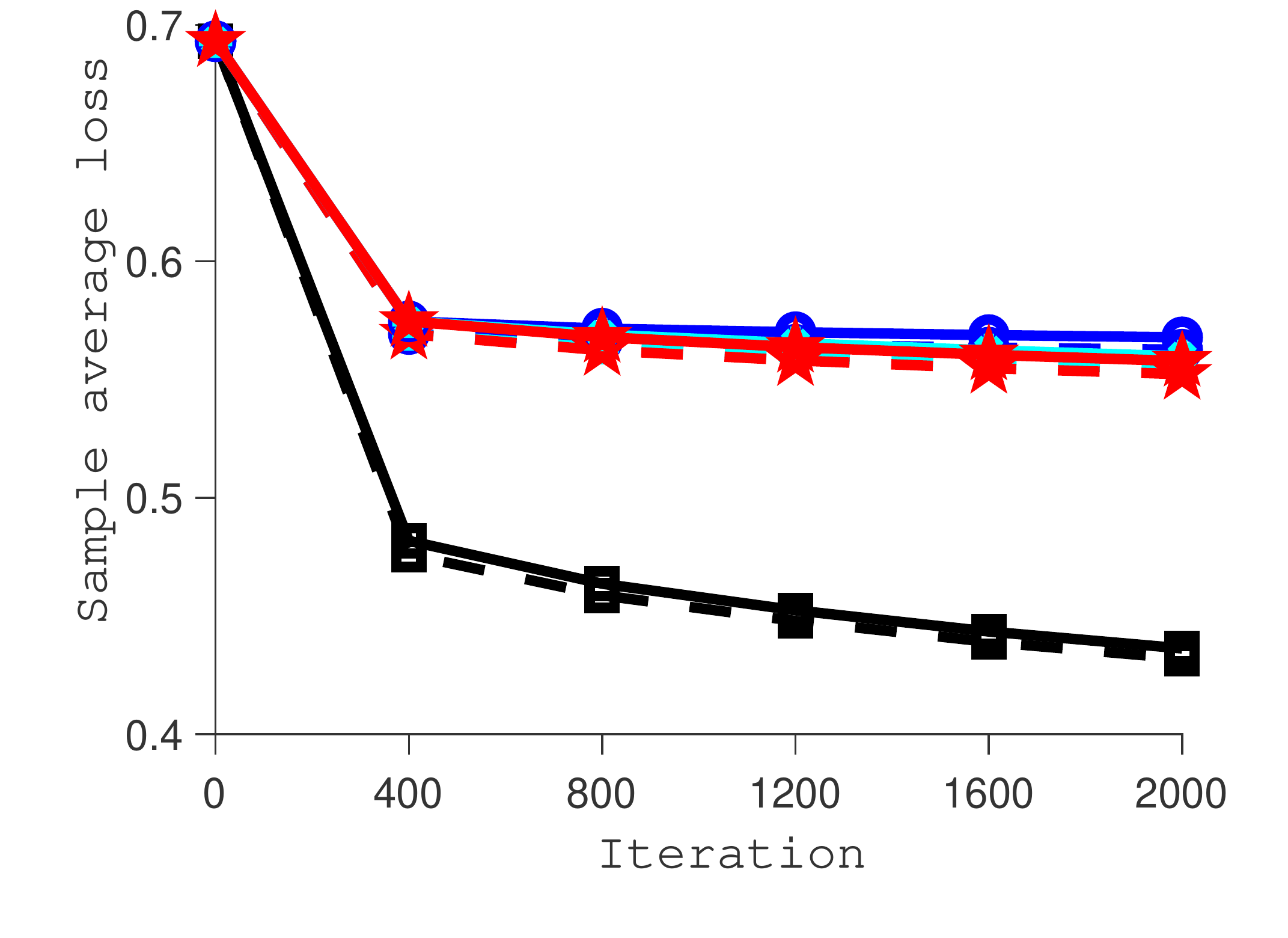}
\end{minipage}
&
\begin{minipage}{.3\textwidth}
\includegraphics[scale=.18, angle=0]{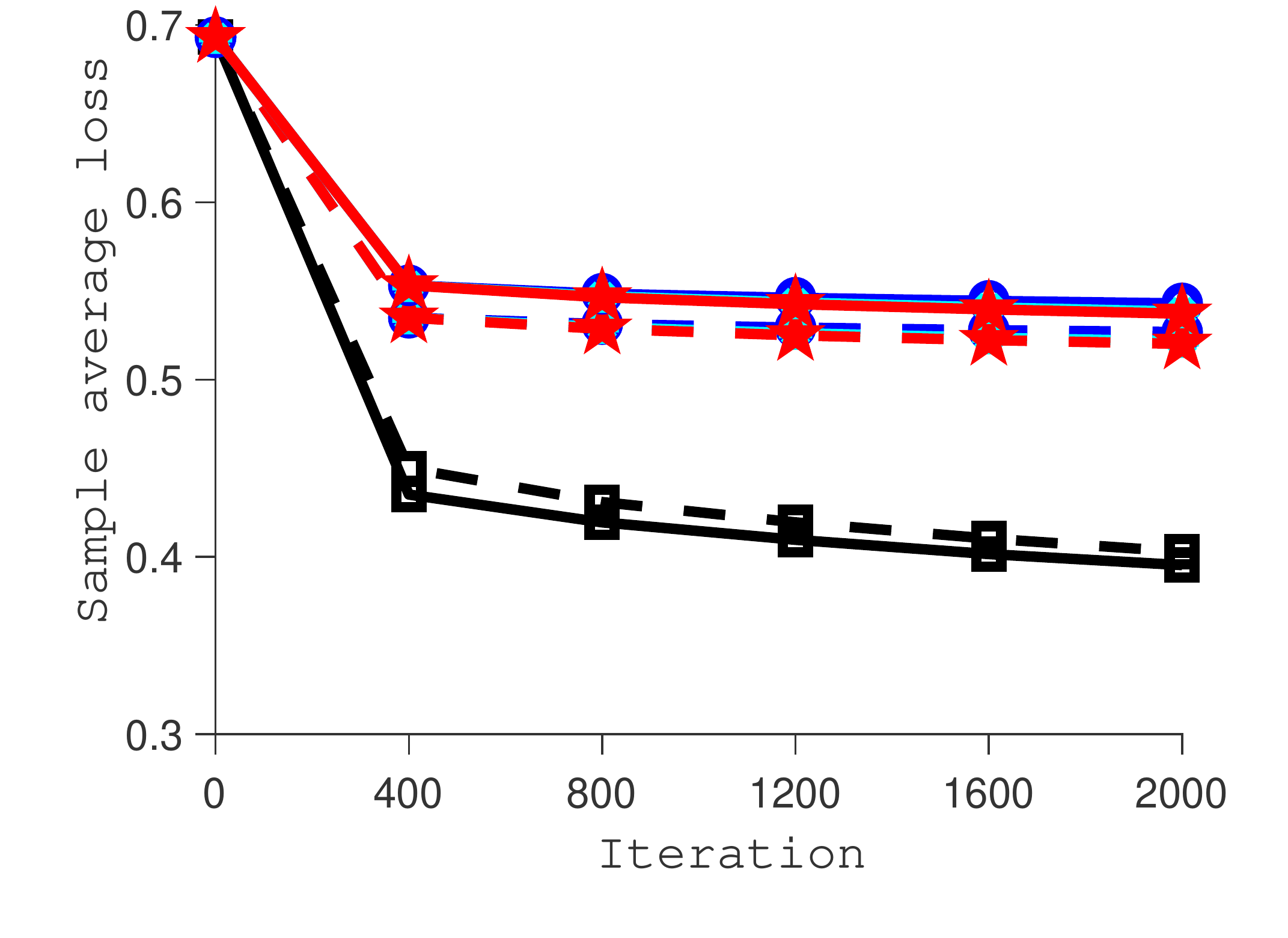}
\end{minipage}
&
\begin{minipage}{.3\textwidth}
\includegraphics[scale=.18, angle=0]{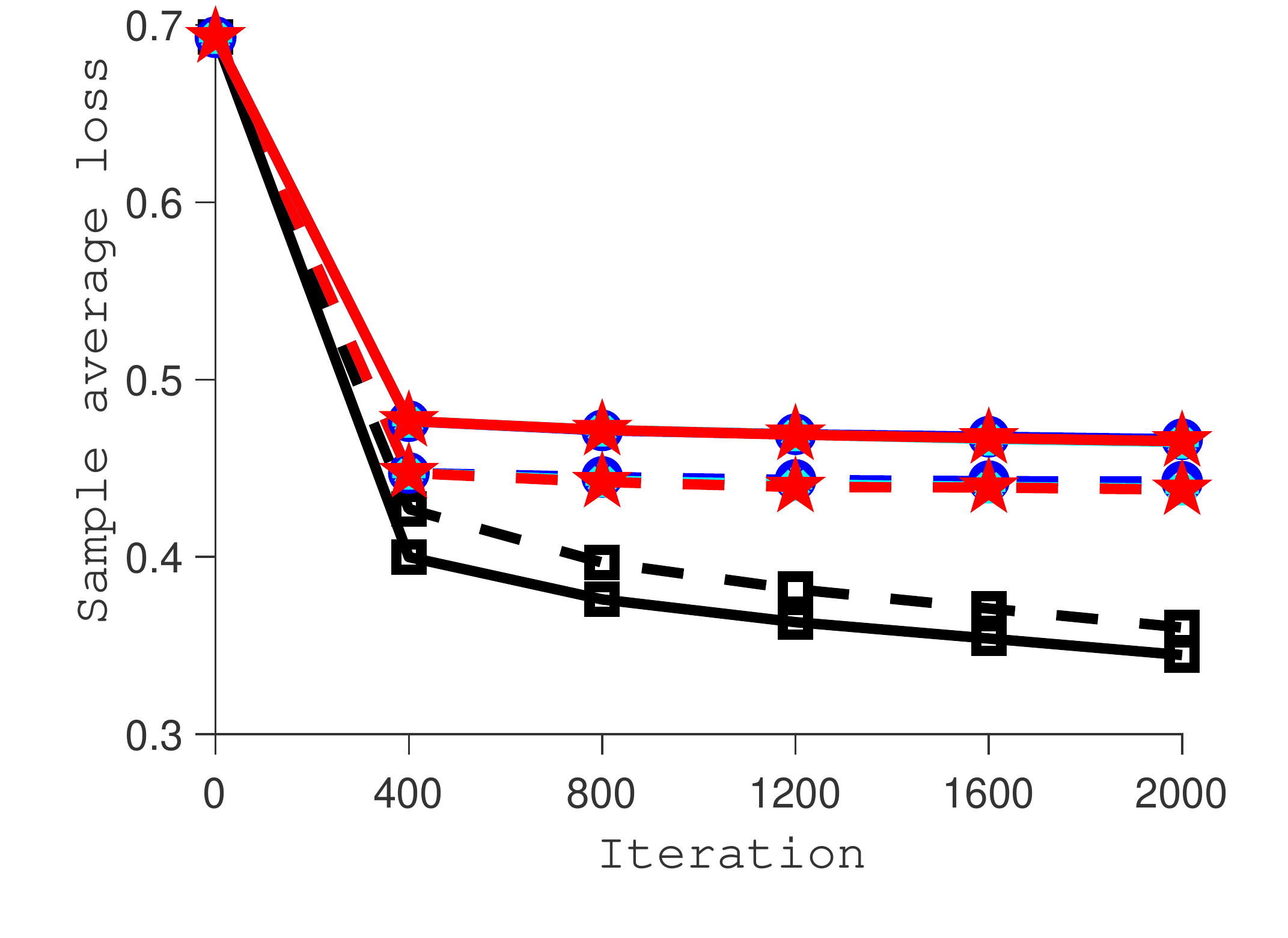}
\end{minipage}
\\
0.1
&
\begin{minipage}{.3\textwidth}
\includegraphics[scale=.18, angle=0]{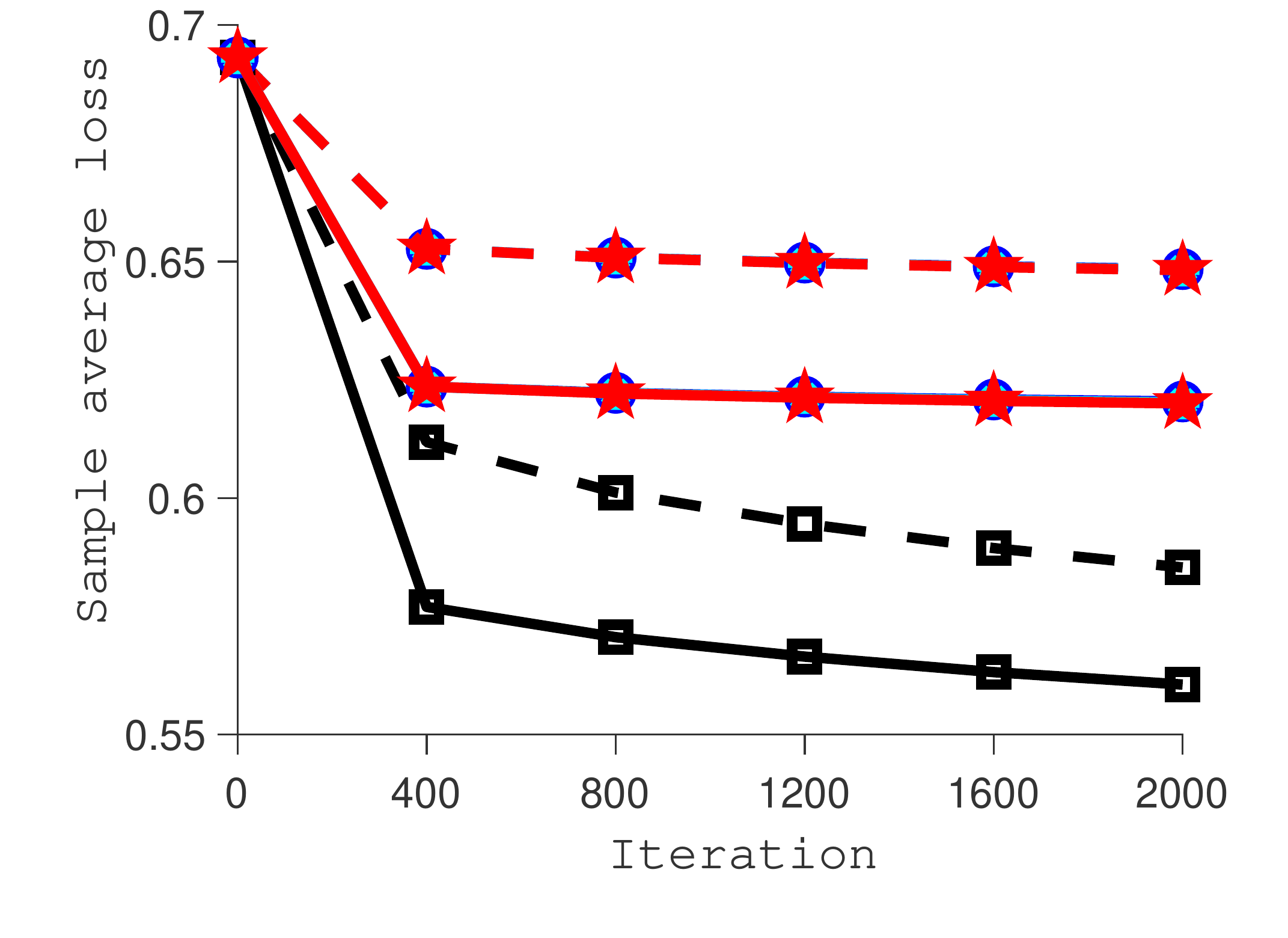}
\end{minipage}
&
\begin{minipage}{.3\textwidth}
\includegraphics[scale=.18, angle=0]{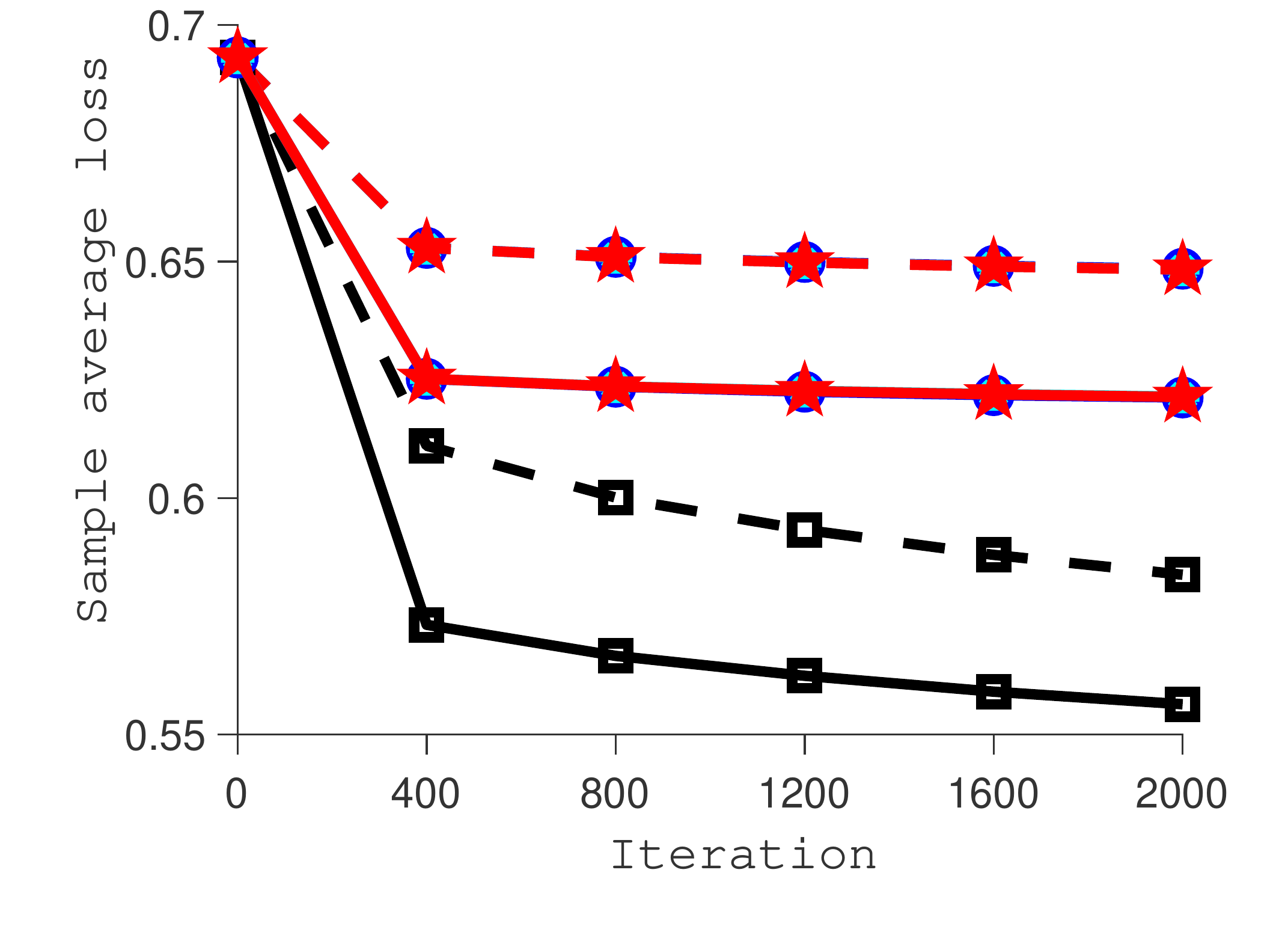}
\end{minipage}
&
\begin{minipage}{.3\textwidth}
\includegraphics[scale=.18, angle=0]{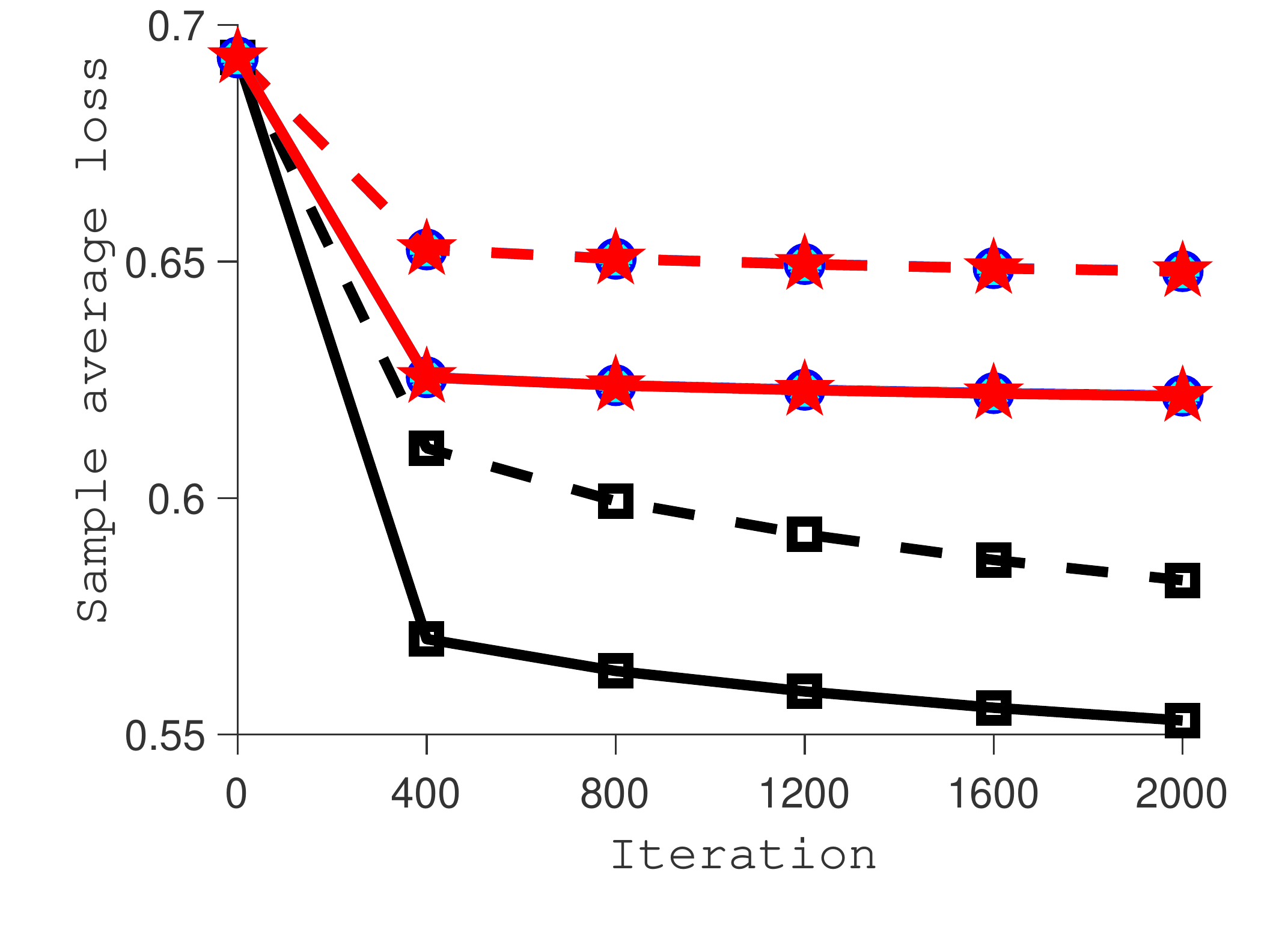}
\end{minipage}
\end{tabular}}
\captionof{figure}{\footnotesize{Algorithm \ref{algorithm:IR-S-BFGS} vs. stochastic LBFGS under constant regularization (i.e., $\rho=1$), and under piece-wise \fys{constant} regularization (i.e., $\rho=0.5, 0.3$), where $\rho$ is the decay ratio of the regularization parameter at each epoch of $400$ iterations.\\}}
\label{m2}
\end{table} 

\begin{table}[t]
\setlength{\tabcolsep}{0pt}
\centering
 \begin{tabular}{m{1cm} || c  c  c}
$N$ & initial cond. 1& initial cond. 2& initial cond. 3 \\ \hline\\
$10^3$
&
\begin{minipage}{.3\textwidth}
\includegraphics[scale=.18, angle=0]{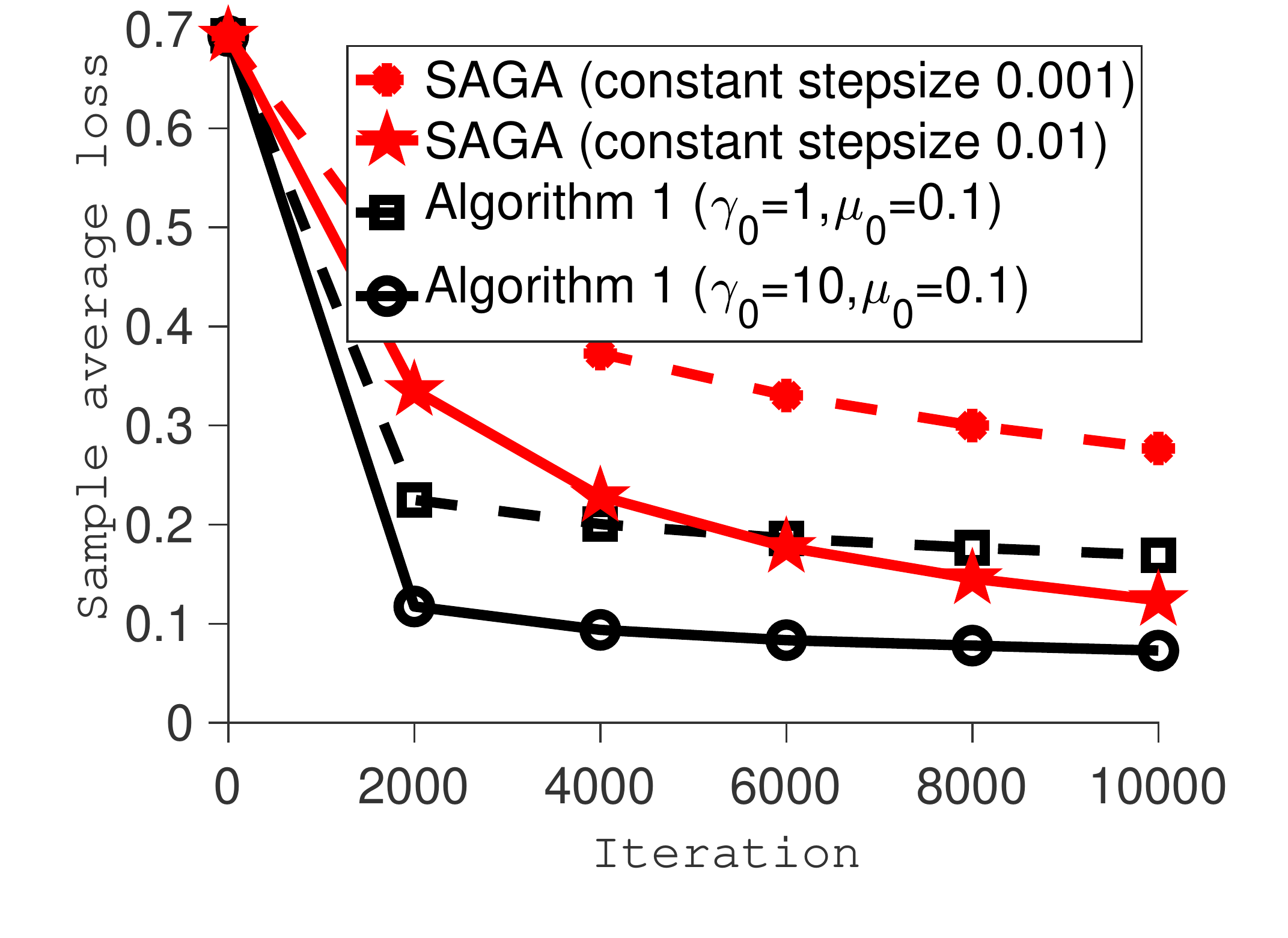}
\end{minipage}
&
\begin{minipage}{.3\textwidth}
\includegraphics[scale=.18, angle=0]{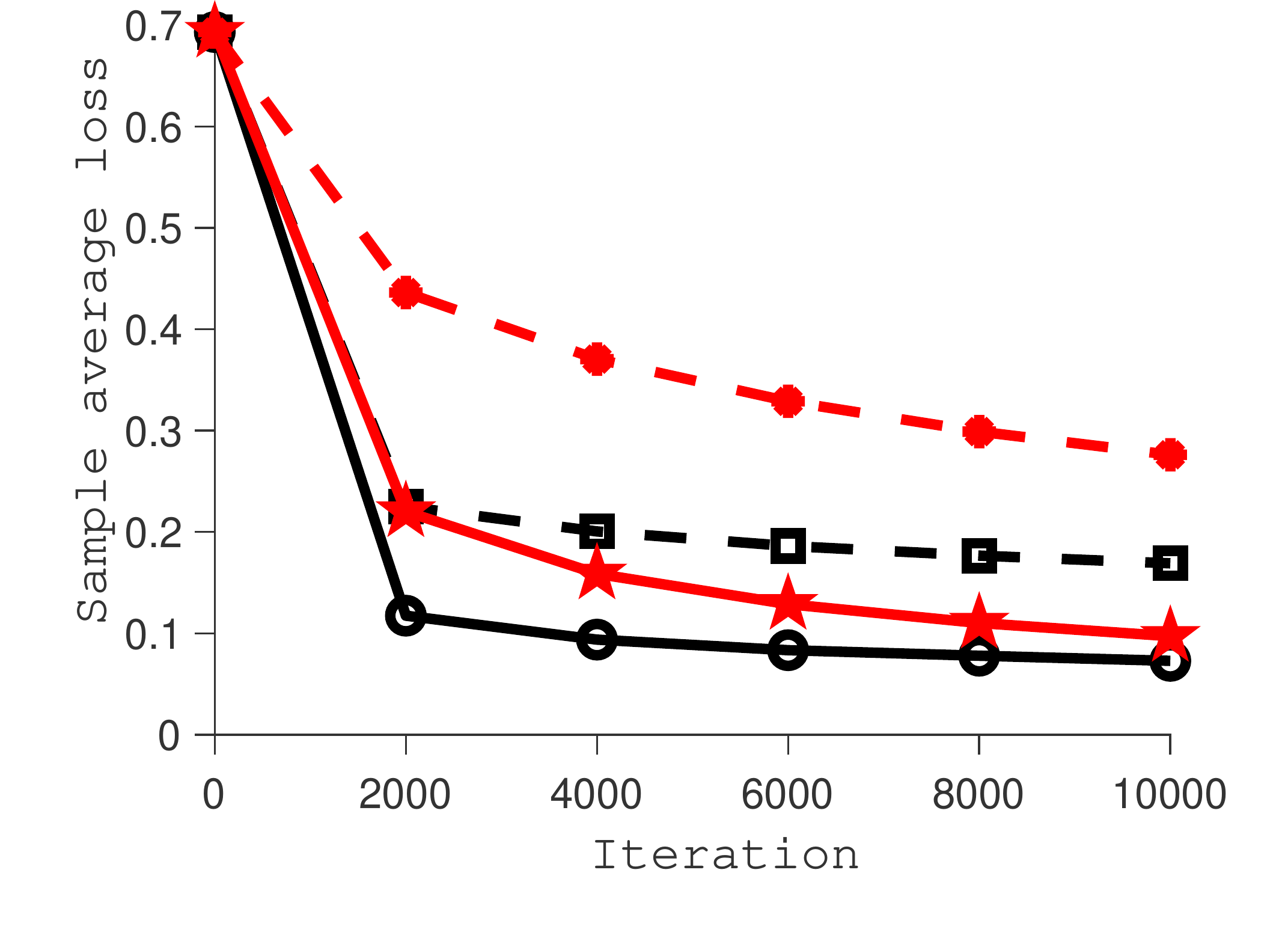}
\end{minipage}
	&
\begin{minipage}{.3\textwidth}
\includegraphics[scale=.18, angle=0]{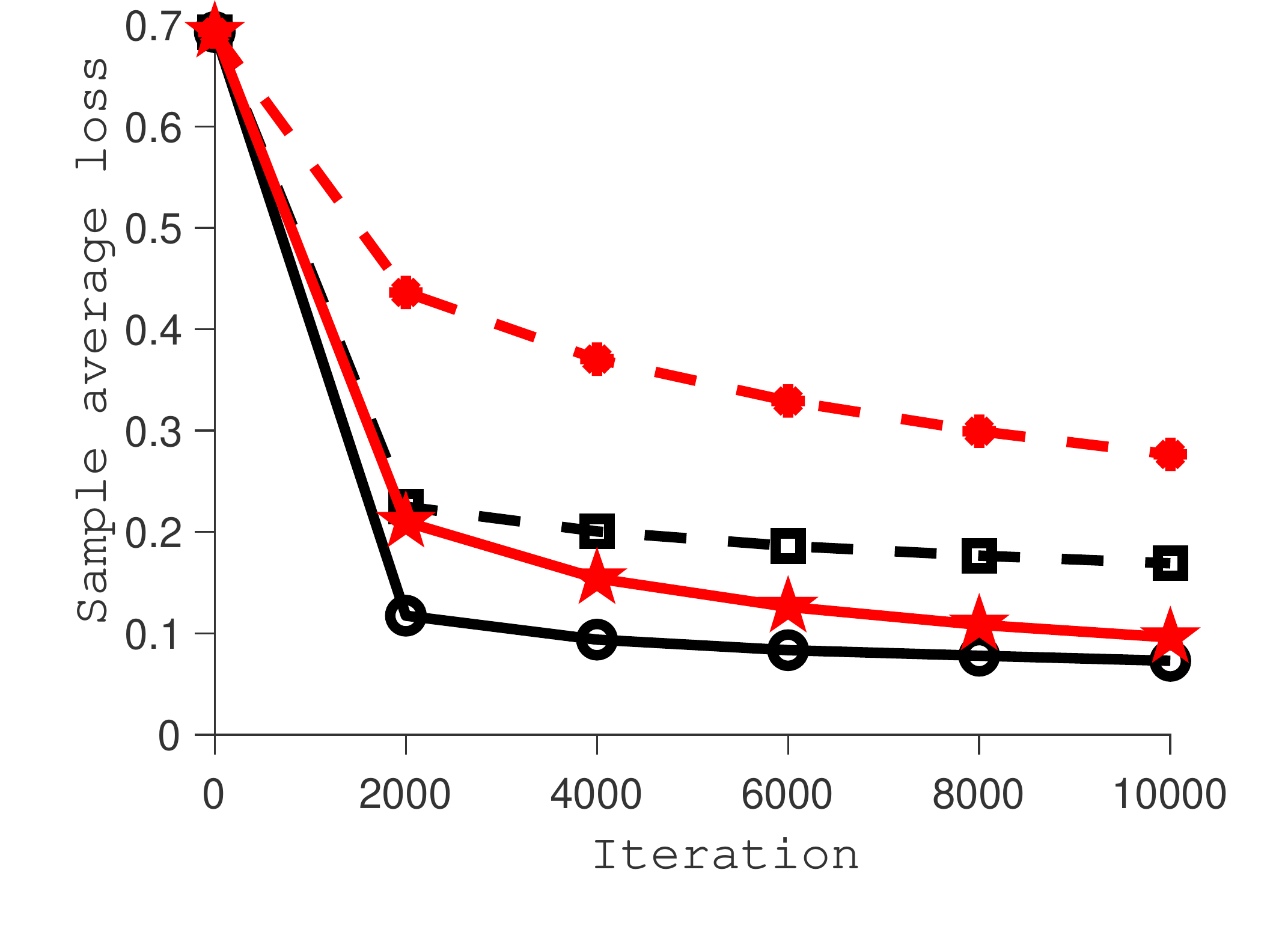}
\end{minipage}
\\
$10^4$
&
\begin{minipage}{.3\textwidth}
\includegraphics[scale=.18, angle=0]{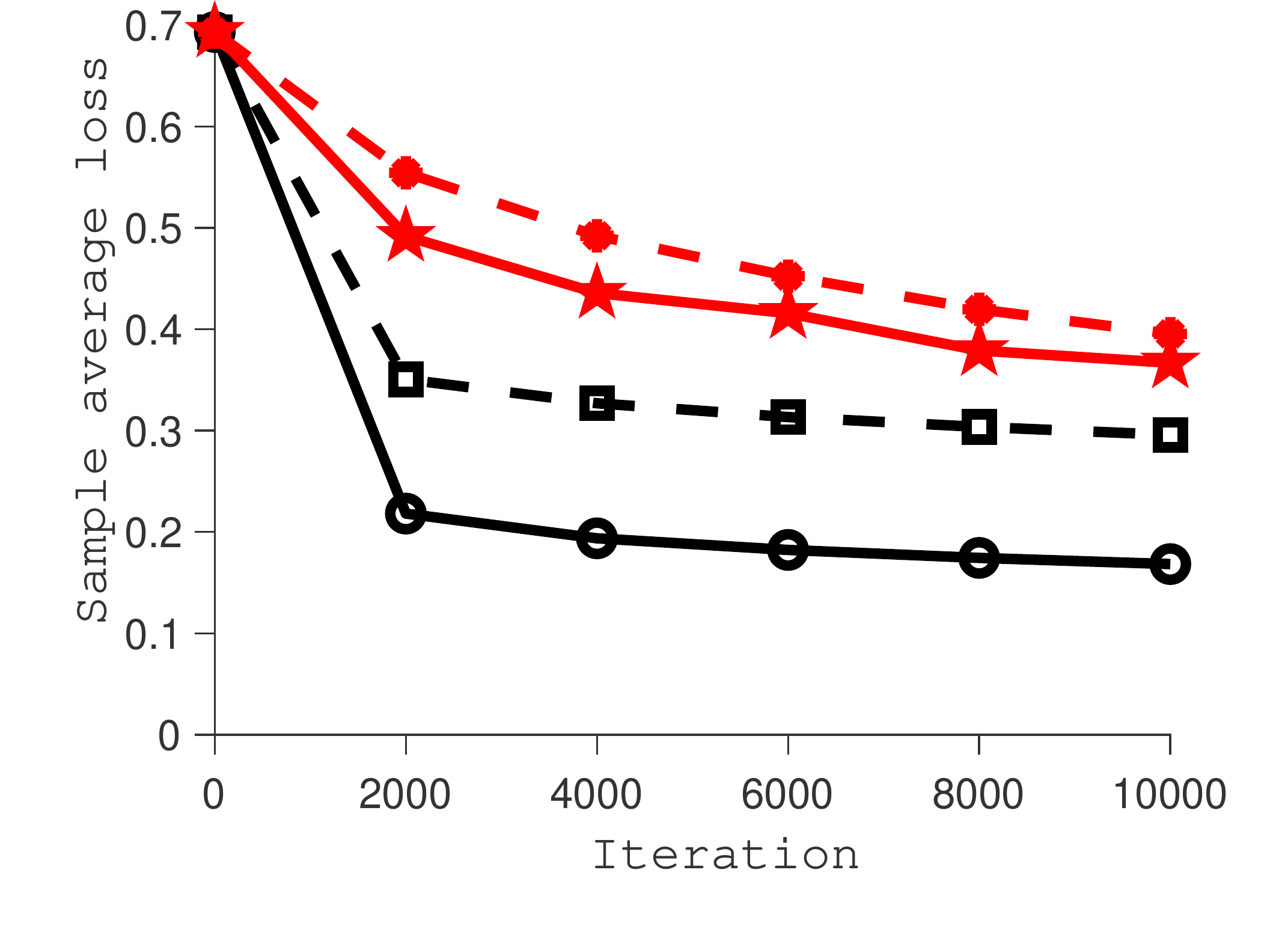}
\end{minipage}
&
\begin{minipage}{.3\textwidth}
\includegraphics[scale=.18, angle=0]{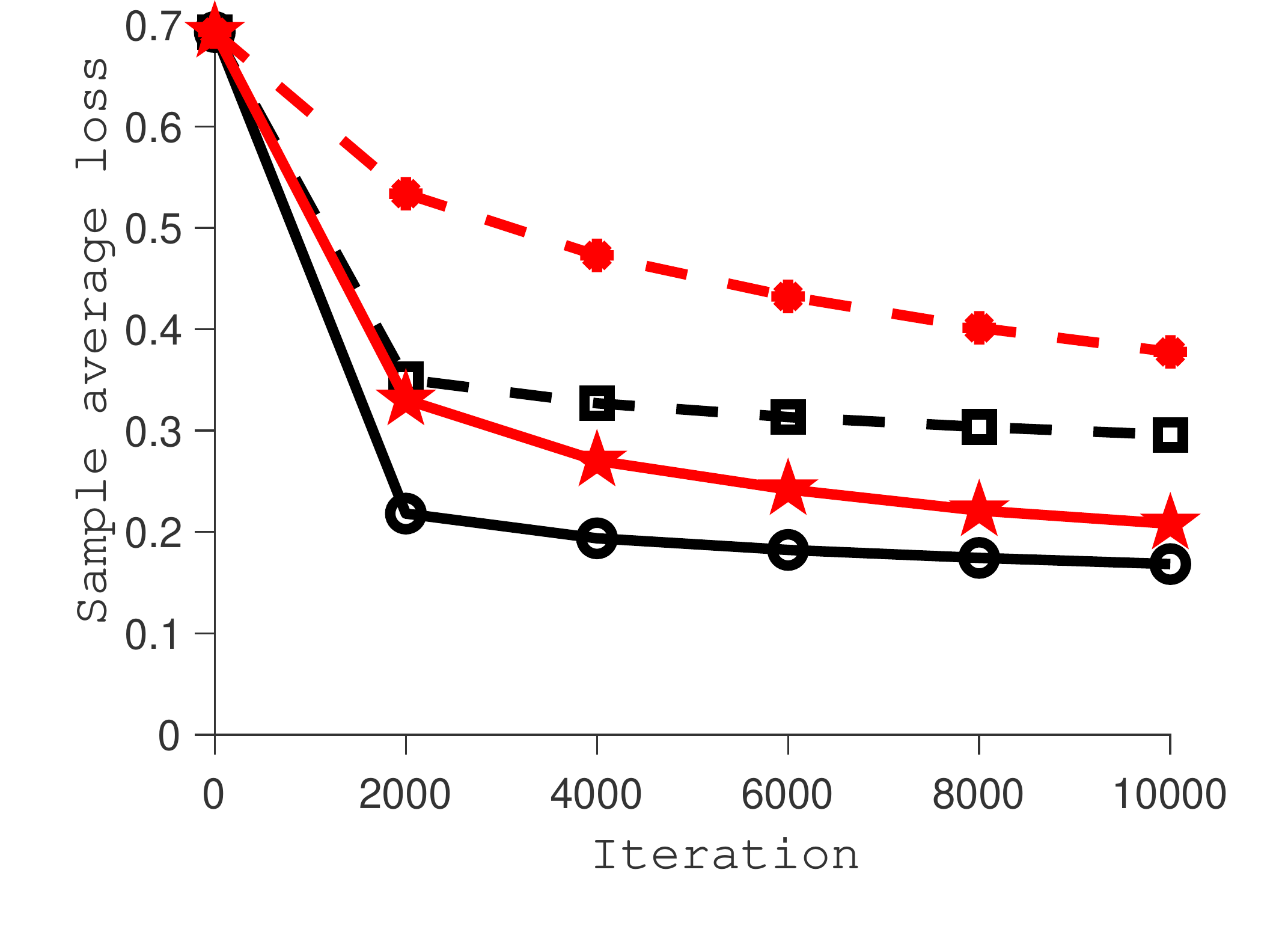}
\end{minipage}
&
\begin{minipage}{.3\textwidth}
\includegraphics[scale=.18, angle=0]{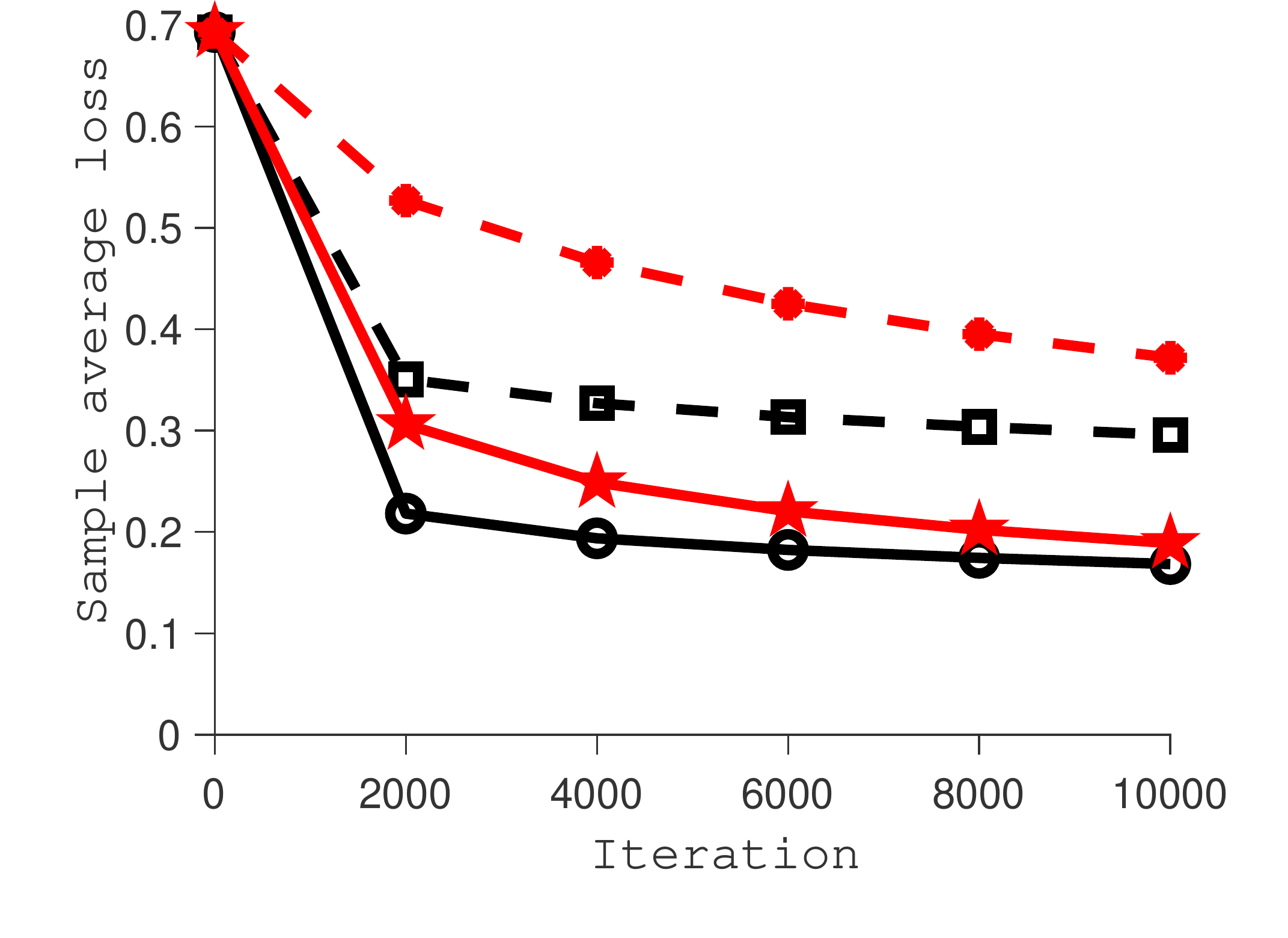}
\end{minipage}
\\
$10^5$
&
\begin{minipage}{.3\textwidth}
\includegraphics[scale=.18, angle=0]{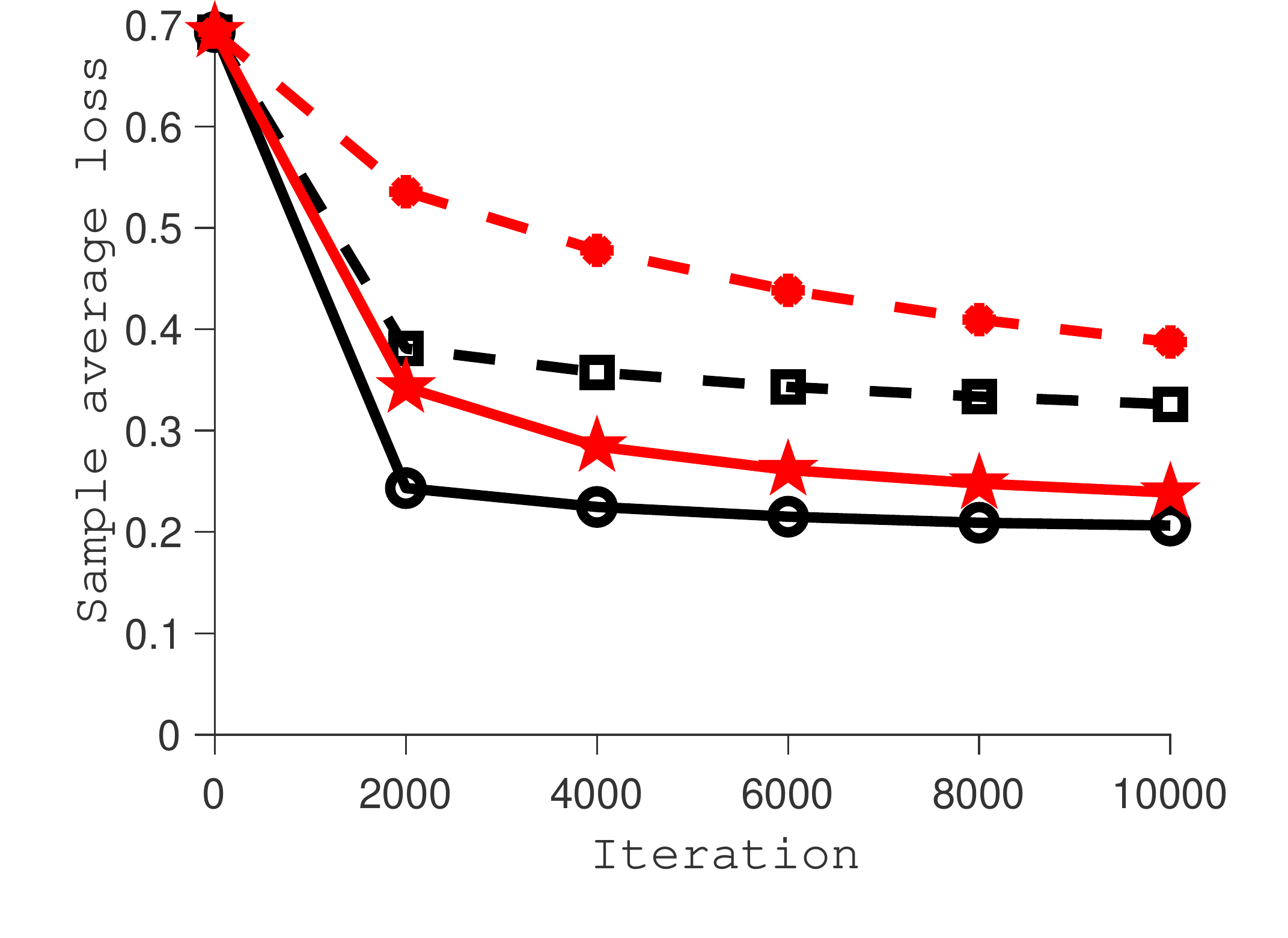}
\end{minipage}
&
\begin{minipage}{.3\textwidth}
\includegraphics[scale=.18, angle=0]{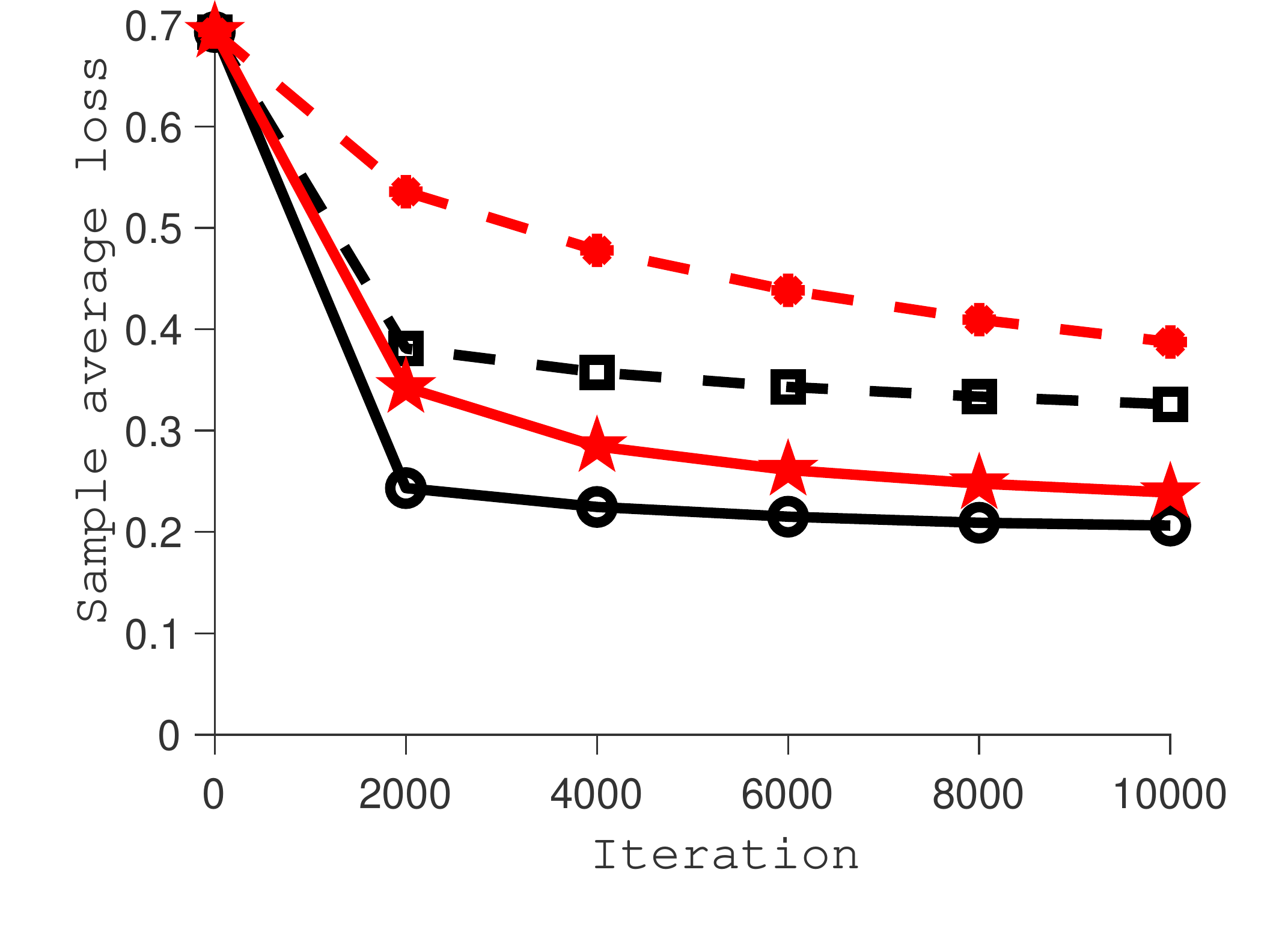}
\end{minipage}
&
\begin{minipage}{.3\textwidth}
\includegraphics[scale=.18, angle=0]{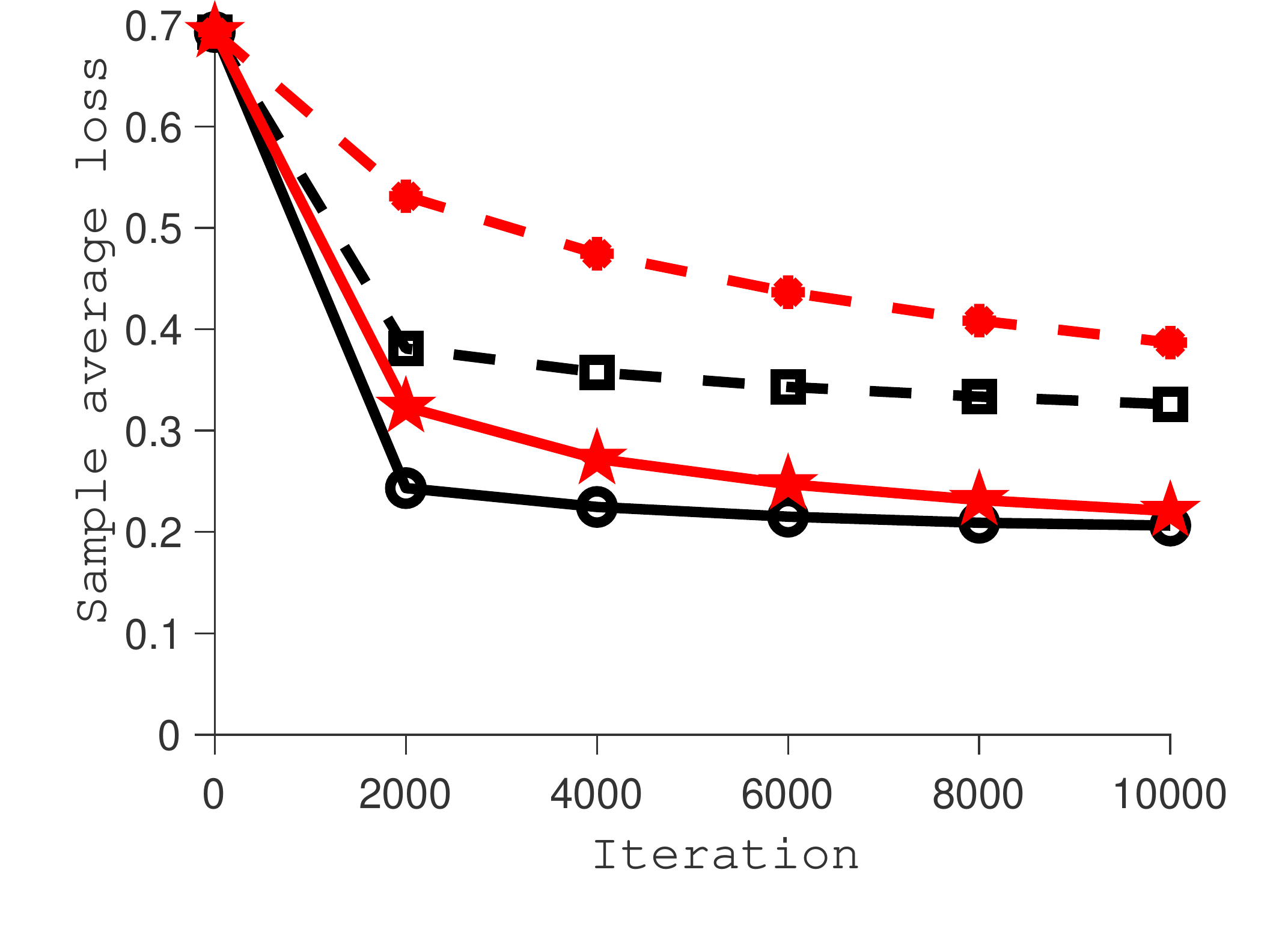}
\end{minipage}
\end{tabular}
\captionof{figure}{\footnotesize{Algorithm \ref{algorithm:IR-S-BFGS} vs. SAGA (with averaging) with different choices of constant stepsize, different sample sizes, and different initial value of the gradient of component functions.\\}}
\label{tableSaga}
\end{table}

\begin{table}[t]
\setlength{\tabcolsep}{0pt}
\centering
 \begin{tabular}{m{1cm} || c  c  c}
$N$ & $\mu_{IAG}=0.1$&$\mu_{IAG}=0.01$& $\mu_{IAG}=0.001$ \\ \hline\\
$1000$
&
\begin{minipage}{.3\textwidth}
\includegraphics[scale=.19, angle=0]{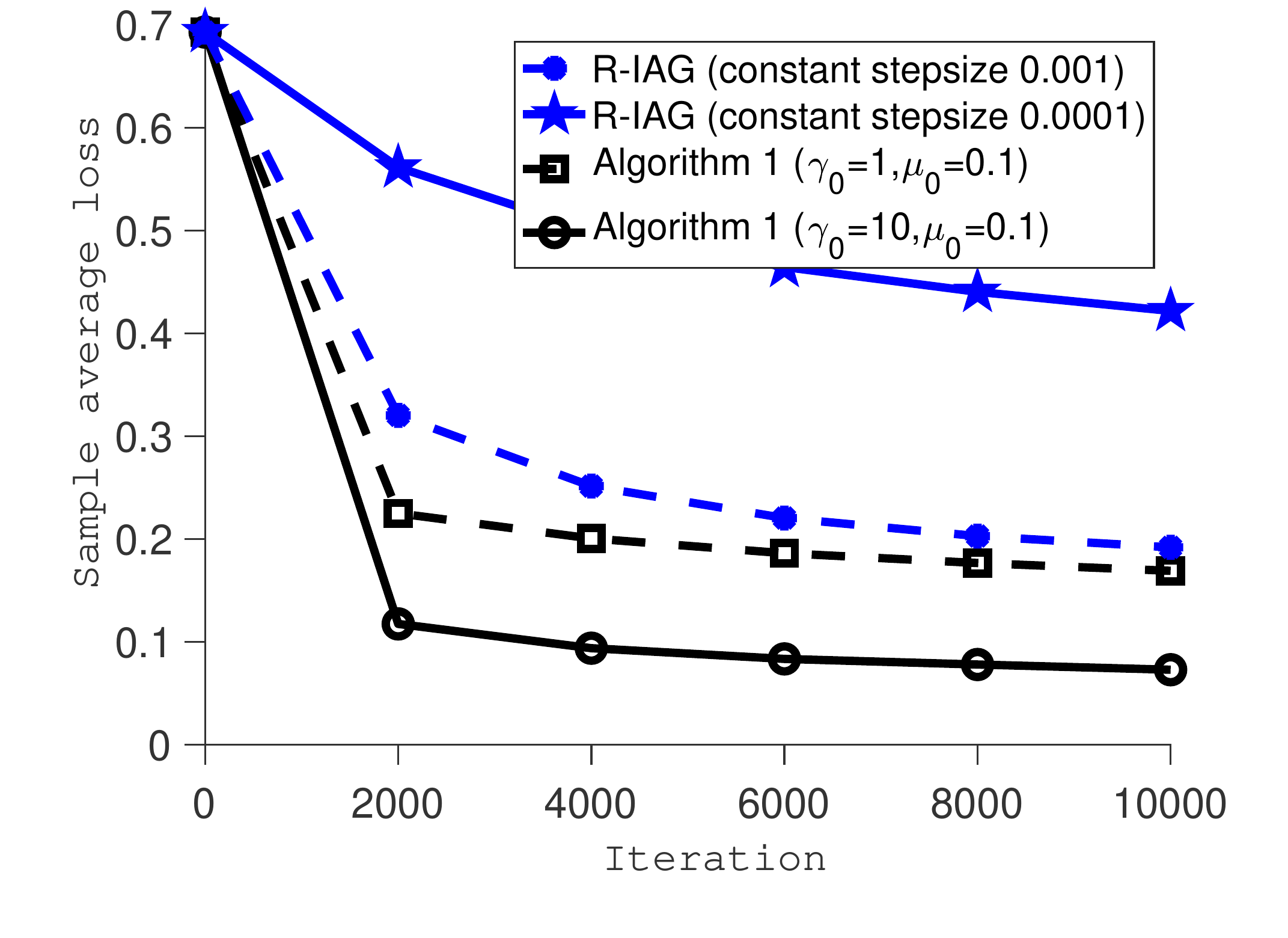}
\end{minipage}
&
\begin{minipage}{.3\textwidth}
\includegraphics[scale=.19, angle=0]{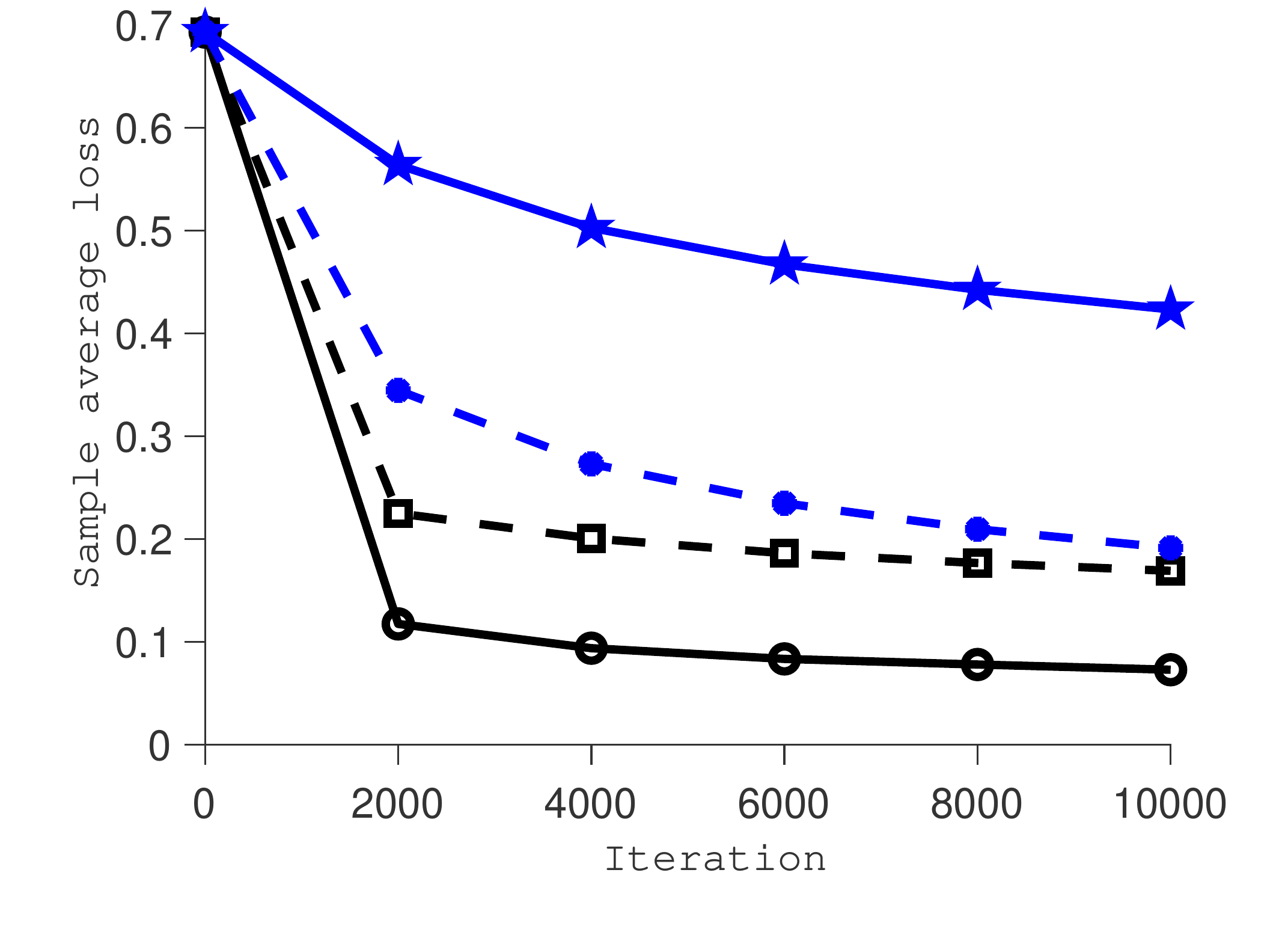}
\end{minipage}
	&
\begin{minipage}{.3\textwidth}
\includegraphics[scale=.19, angle=0]{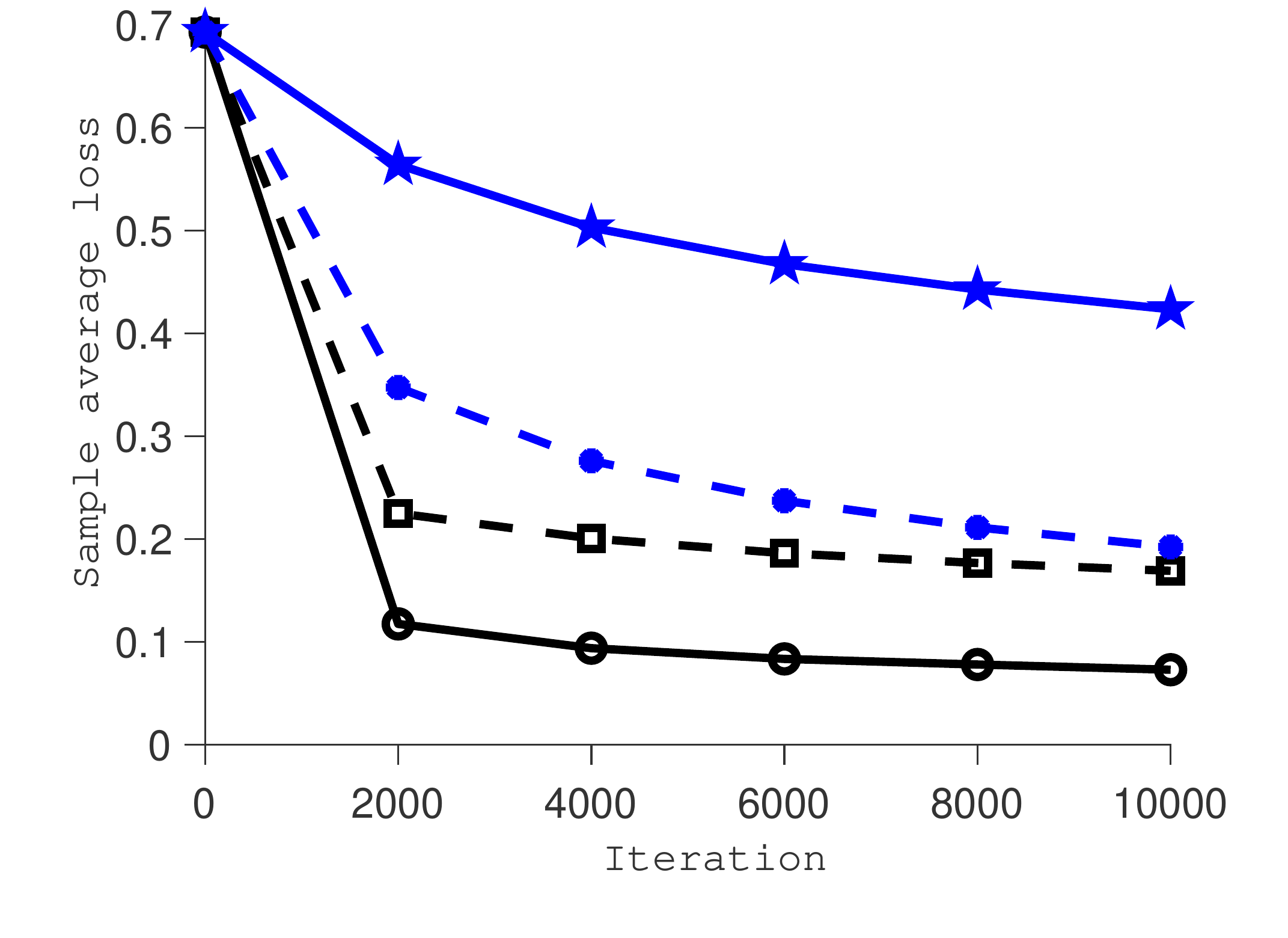}
\end{minipage}
\\
$2000$
&
\begin{minipage}{.3\textwidth}
\includegraphics[scale=.19, angle=0]{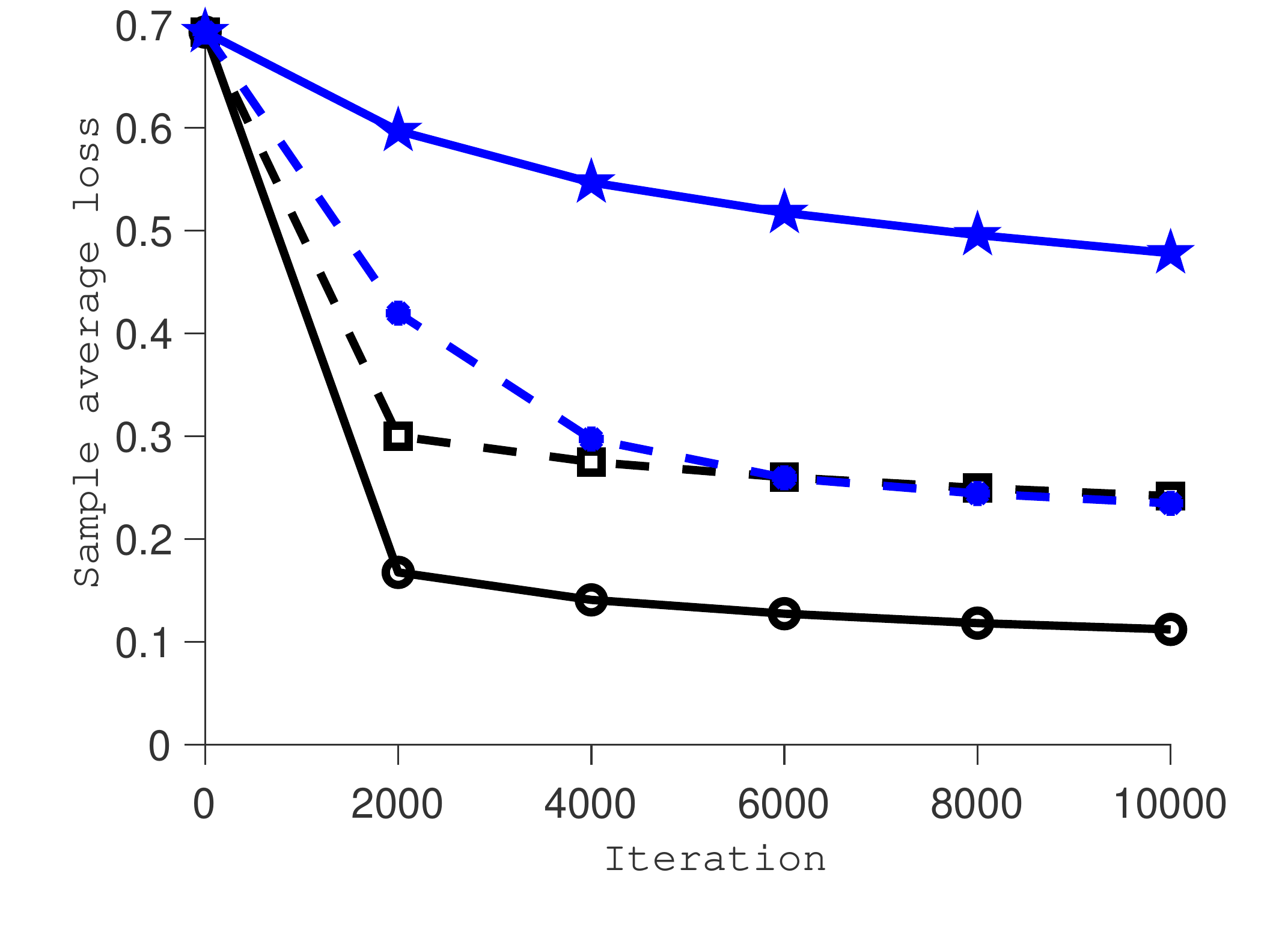}
\end{minipage}
&
\begin{minipage}{.3\textwidth}
\includegraphics[scale=.19, angle=0]{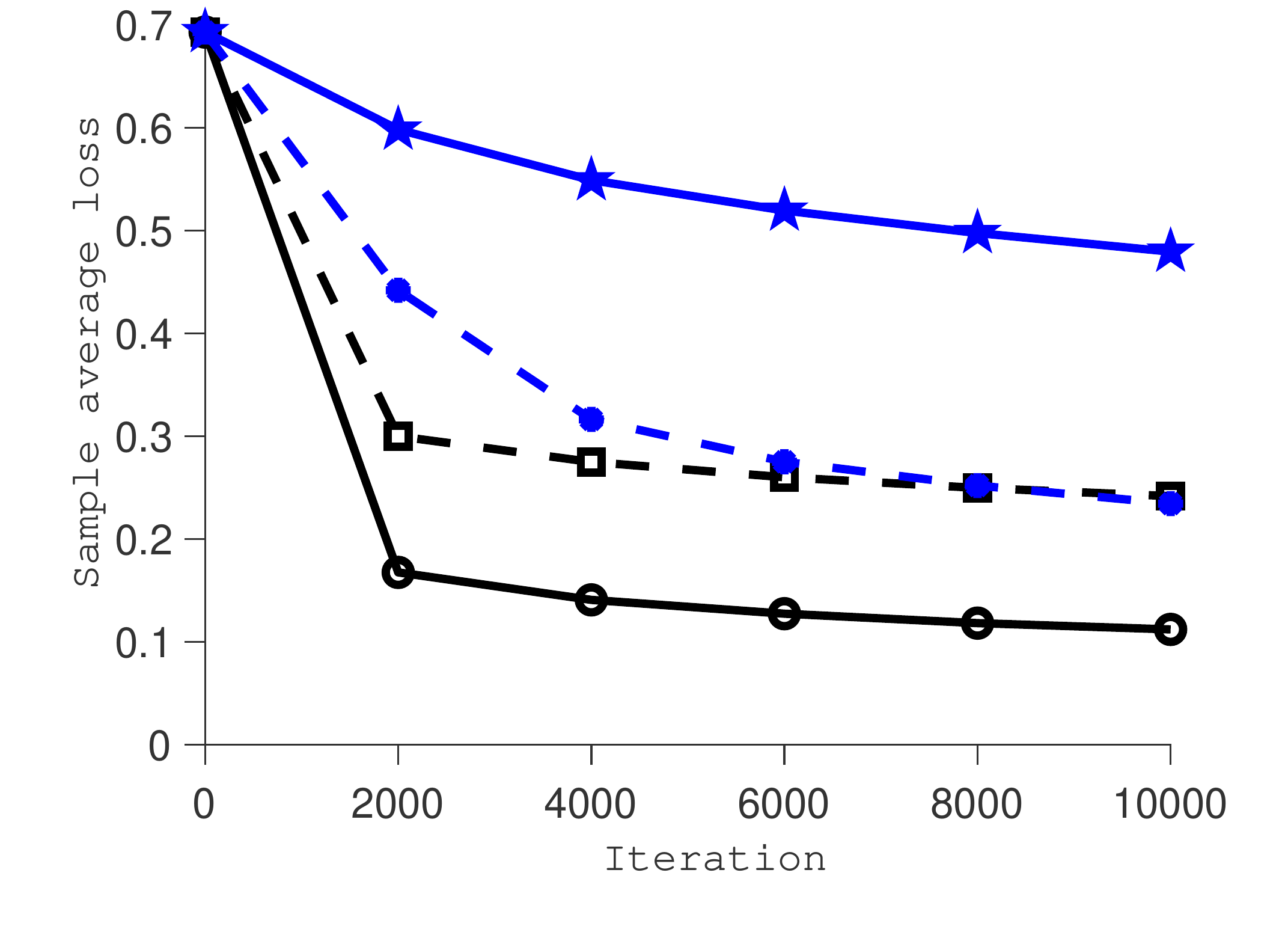}
\end{minipage}
&
\begin{minipage}{.3\textwidth}
\includegraphics[scale=.19, angle=0]{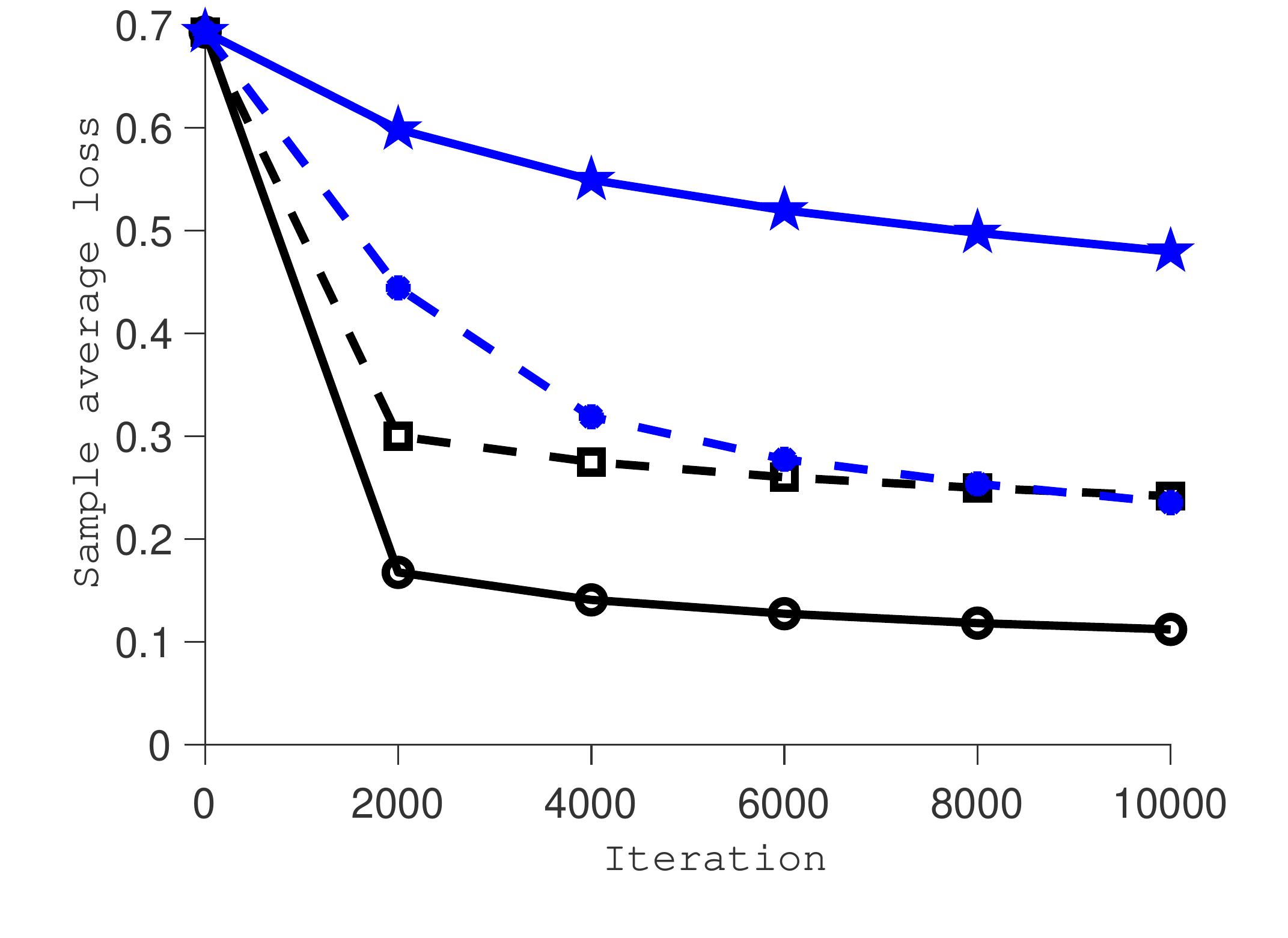}
\end{minipage}
\\
$5000$
&
\begin{minipage}{.3\textwidth}
\includegraphics[scale=.19, angle=0]{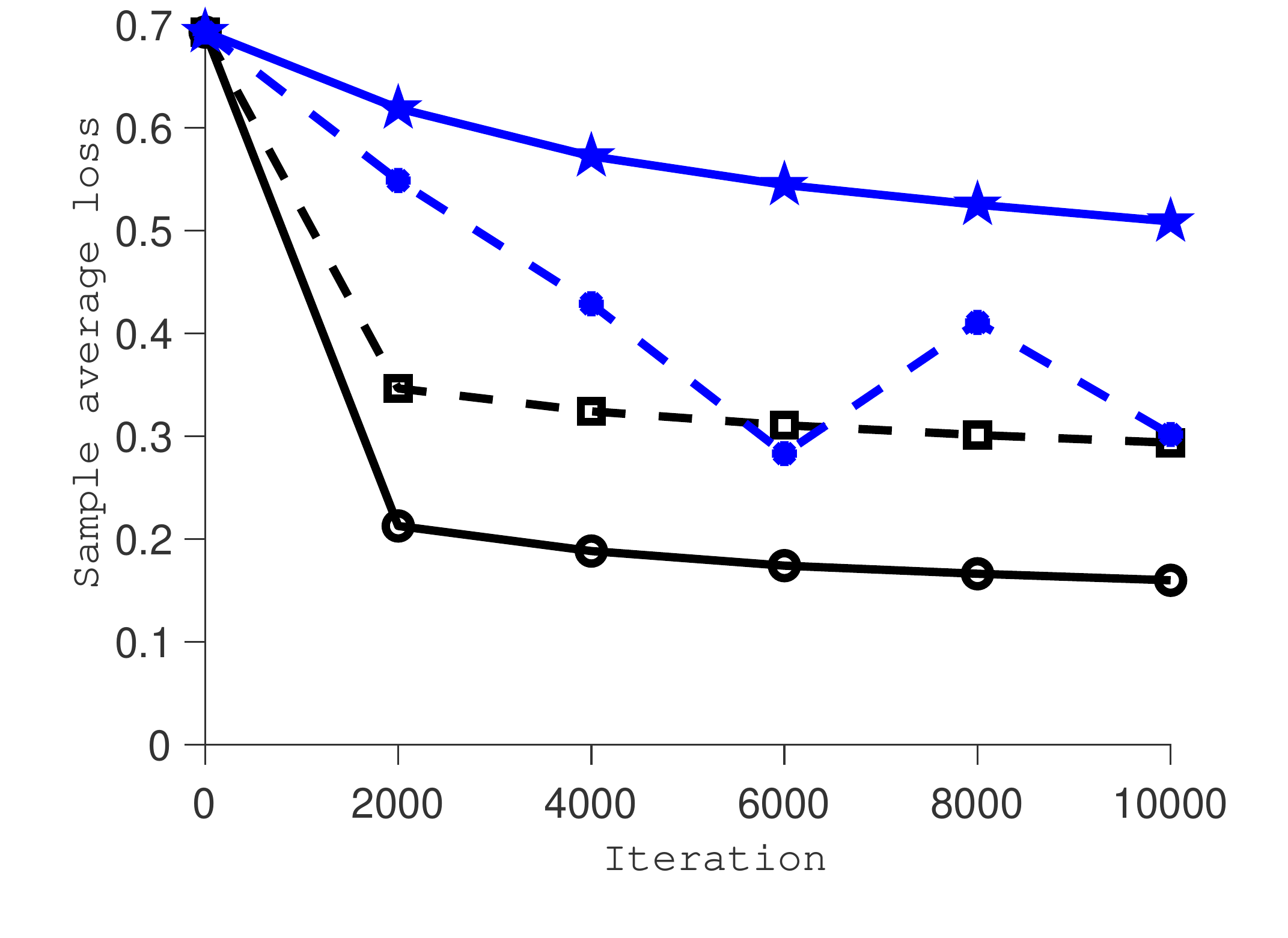}
\end{minipage}
&
\begin{minipage}{.3\textwidth}
\includegraphics[scale=.19, angle=0]{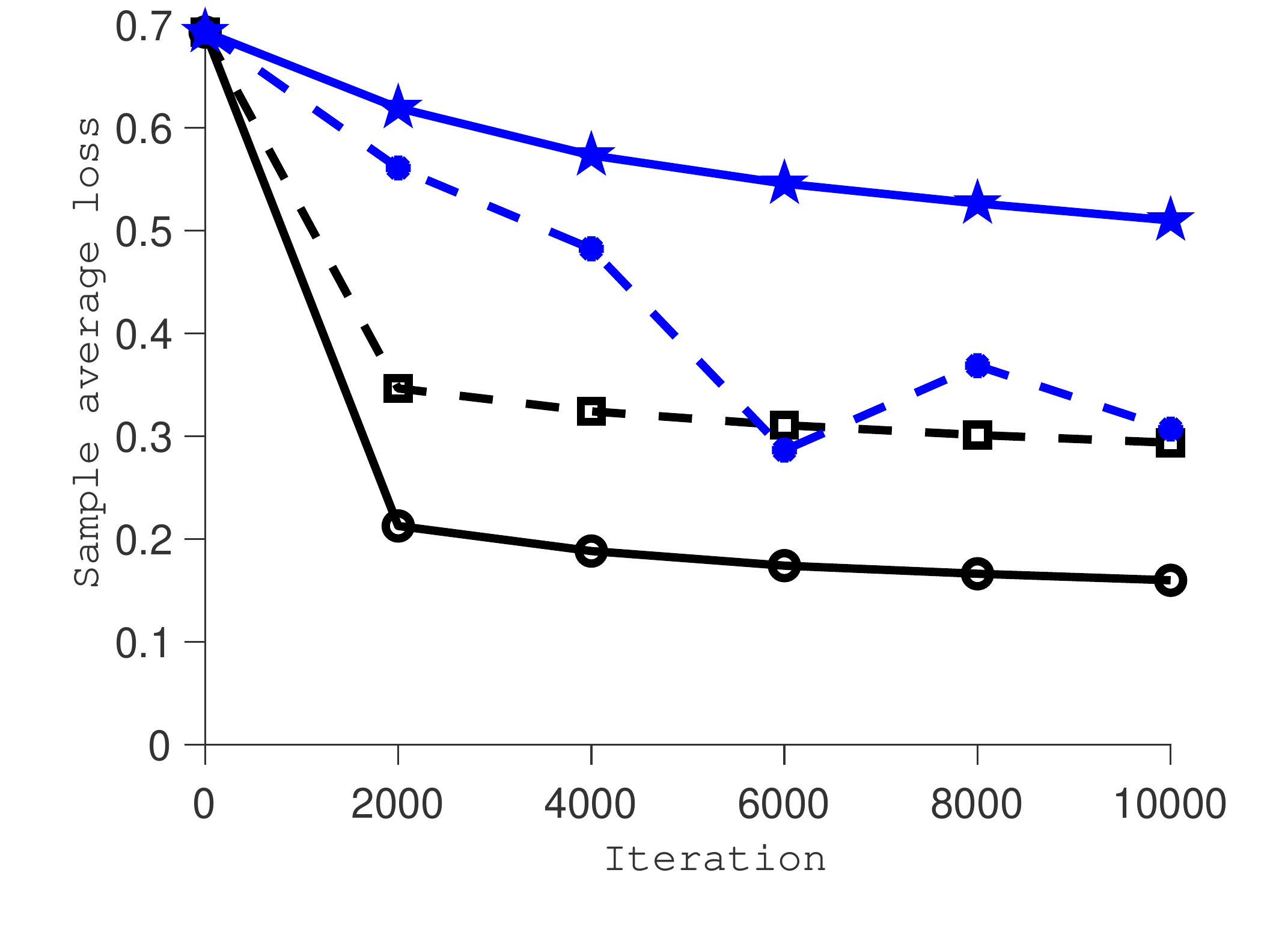}
\end{minipage}
&
\begin{minipage}{.3\textwidth}
\includegraphics[scale=.19, angle=0]{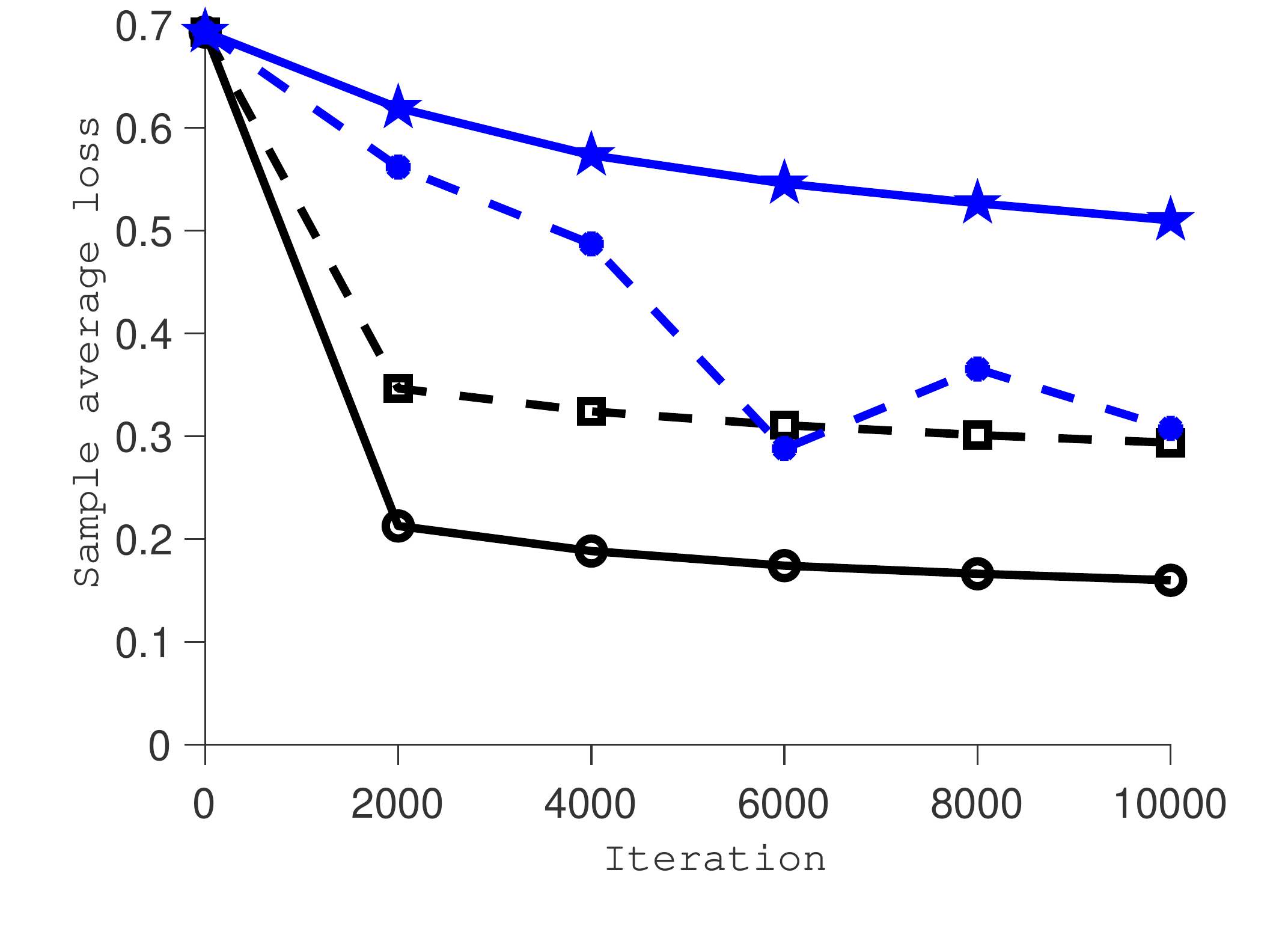}
\end{minipage}
\end{tabular}
\captionof{figure}{\footnotesize{Algorithm \ref{algorithm:IR-S-BFGS} vs. regularized IAG (non-averaging) with different choices of the regularization parameter, different choices of the stepsize, and different sample sizes.\\}}
\label{tableIAG}
\end{table} 

\subsection{Comparison with standard SQN schemes}\label{sec:num-reg}
\fy{To solve problem \eqref{logistic}, the standard LBFGS} methods in~\cite{Mokhtari15,nocedal15} solve an approximate problem of the form 
\begin{align}\label{logisticReg}\tag{Regularized LRM}
\min_{x \in \Real^n} f(x):=\frac{1}{N}\sum_{i=1}^N\ln \left(1+\exp\left(-u_i^Txv_i\right)\right)+\fy{\frac{\eta}{2}\|x\|^2},
\end{align}
where \fy{$\eta>0$} is an arbitrary regularization parameter. \fy{To perform the first experiment, we consider comparison of Algorithm \ref{algorithm:IR-S-BFGS} with three variants of the standard LBFGS schemes, all denoted by RS-LBFGS (see Figure~\ref{m2}). In RS-LBFGS schemes, we use the stepsize of the form $\gamma_k=\frac{\g_0}{{k+1}}$ and drop $\eta$ at epochs of 400 iterations using a decay factor denoted by $\rho \in (0,1]$. Of these, in the first scheme, we assume $\rho=1$, meaning that $\eta$ is kept constant throughout the implementation of the SQN scheme. In the second scheme, we use $\rho=0.5$. This means for example, after every $400$ iterations, we set $\eta:=0.5\eta$. In the third scheme, we use $\rho=0.3$. \fys{We let $\g_0 \in \{10,0.5,0.1\}$, $\eta_0 \in \{1,0.5,0.01\}$, $m \in \{2,5\}$, $N=10^4$, and $x_0$ be the origin.} In all cases, we use five sample paths to calculate the average value of the objective function in \eqref{logistic}.

\textbf{Insights:} We observe that Algorithm \ref{algorithm:IR-S-BFGS} performs uniformly better than the three variants of the standard SQN scheme under different tuning rules for the regularization parameter. This suggests that for merely convex stochastic optimization, SQN schemes using the tuning rules for the stepsize and regularization parameter given as $\g_k \approx 1/\sqrt[3]{k^2}$ and $\mu_k\approx1/\sqrt[3]{k}$ have a faster convergence speed. 

\subsection{Comparison with SAGA \vus{on merely convex
problems}}\label{sec:num-SAGA} Recall that in addressing the finite-sum
minimization problems with merely convex objectives, employing averaging and
under a constant stepsize, SAGA admits a sublinear convergence rate of
$\mathcal{O}\left(\frac{N}{k}\right)$~\cite{Saga14}. The simulation results are
provided in Figure \ref{tableSaga}. These results include different sample
sizes, i.e., $N \in \{10^3,10^4,10^5\}$, different initial conditions for SAGA,
and different choices of the stepsizes and the initial regularization parameter
for Algorithm \ref{algorithm:IR-S-BFGS}. SAGA uses the evaluation of the
gradient map of the component \fyfy{functions} at the starting
point. \fyfy{Here, we use three different values for the evaluated
gradient maps} at the starting point, i.e., the origin, to \fys{study} the
sensitivity of SAGA with respect to the initial conditions. Of these, in
initial condition 3, we use the exact value of the gradient maps, while in
initial condition 2, we  perturb values of the gradient maps. This perturbation
is increased in initial condition 1. 

\textbf{Insights:} From Figure \ref{tableSaga}, we observe that Algorithm \ref{algorithm:IR-S-BFGS} competes well with SAGA. We discuss the comparisons as follows: (i) A computational burden in implementation of SAGA is the memory requirement of this scheme. Generally speaking, SAGA requires storing a matrix of  $\mathcal{O}(Nn)$ at each iteration. Exceptions include the case where the objective function is in terms of a linear regression model function (e.g., in \eqref{logistic}). This is in contrast with Algorithm \ref{algorithm:IR-S-BFGS} where the memory requirement is $\mathcal{O}(mn)$. (ii) As expected, the performance of SAGA deteriorates when the sample size increases. However, the performance of Algorithm \ref{algorithm:IR-S-BFGS} seems to be more robust with respect to the increase in the sample size. (iii) The performance of SAGA seems to be moderately sensitive to the initial conditions. 

\subsection{Comparison with IAG}\label{sec:num-IAG}
Recall that in solving finite-sum minimization problems with $\mu$-strongly convex objectives, using a constant stepsize, (non-averaging) IAG admits a linear convergence rate of $\mathcal{O}\left(\left(1-(\mu/N)^2\right)^{2k}\right)$ where $N$ is the number of component functions (cf.~\cite{Iag17}). Accordingly, to do the numerical comparisons with IAG, we regularize problem \eqref{logistic} with a constant $\mu_{IAG}>0$. Figure \ref{tableIAG} shows the simulation results for different choices of $\mu_{IAG}$, $N$, IAG stepsize, and the initial stepzie and regularization parameter of Algorithm \ref{algorithm:IR-S-BFGS}. 

\textbf{Insights:} (i) Due to the excessive memory requirements of $\mathcal{O}(nN)$ \vus{associated with} IAG, \us{such a scheme becomes challenging to implement when $n$ becomes large as in this case where} $n=138,921$. Consequently, \vus{we use a sample size} $N \in \{1000,2000,5000\}$. However, Algorithm \ref{algorithm:IR-S-BFGS} only requires memory of $\mathcal{O}(nm)$, allowing for implementations with large values of $N$. (ii) Similar to SAGA, the performance of IAG is deteriorated when the sample size increases. However, the performance of Algorithm \ref{algorithm:IR-S-BFGS} seems to be more robust with changes in the sample size. (iii) \fyfy{For} each fixed value of $N$, despite the change in the value of $\mu_{IAG}$, the performance of IAG in terms of the true objective function in \eqref{logistic} does not necessarily improve. \fys{Importantly, this observation suggests that in the standard regularization approach, tuning the regularization parameter could be computationally expensive.}}  

\section{Concluding remarks}\label{sec:conc}
We consider stochastic quasi-Newton (SQN) methods for solving large scale stochastic optimization problems with smooth but unbounded gradients. Much of the past research on convergence rates of these algorithms relies on the strong convexity of the objective function. \fys{We employ an iterative regularization scheme where the regularization parameter is updated iteratively within the algorithm. We establish the convergence in an a.s. sense and a mean sense. Moreover, }we prove that the iterates generated by the \fy{iteratively} regularized stochastic LBFGS scheme converges to an optimal solution at the rate $\mathcal{O}\left(\frac{1}{k^{1/3-\e}}\right)$ for arbitrary small $\e>0$. The deterministic variant of this algorithm achieves the rate $\mathcal{O}\left(\frac{1}{k^{1-\e}}\right)$. \fy{The numerical experiments performed on a large scale classification problem indicate that the proposed LBFGS scheme performs well compared to methods such as standard SQN schemes, and other first-order schemes such as SAGA and IAG}.  

\fy{\section{Appendix}\label{App}
\subsection{Proof of Lemma \ref{sumProductBounds}}\label{app:sumProductBounds}
From $0<a_1 \leq \ldots \leq a_n$, we can write 
\begin{align*}
(n-(i-1))a_i \leq \sum_{j=1}^na_j, \qquad \hbox{for all } i \in \{1,\ldots,n\}.
\end{align*}
Invoking $\sum_{i=1}^n{a_i} \leq S$, we obtain $a_i \leq \frac{S}{n-(i-1)}$, for all $i \in \{1,\ldots,n\}$. From the preceding relation and that $\prod_{j=1}^na_j \geq P$, we can obtain $a_1 \geq (n-1)!P/S^{n-1}.$
}
\subsection{Proof of Lemma \ref{LBFGS-matrix}}\label{app:lemmaLBFGSmatrix}
\begin{proof} \fys{Throughout, we let $\lambda_{k,\min}$, $\lambda_{k,\max}$, and $B_k$ denote the minimum eigenvalue, maximum eigenvalue, and inverse of matrix $H_k$ in \eqref{eqn:H-k-m}, respectively.}
It can be seen, by induction on $k$, that \fy{$H_k$ is symmetric and $\sF_k$ measurable.} We use induction on odd values of \fys{$k\geq 2m-1$ to show that parts (a), (b), and (c) hold.} Suppose \fy{$k\geq 2m-1$} is odd and for any \fys{odd} $t<k$, we have \fy{$s_{\lceil t/2\rceil}^T{y_{\lceil t/2\rceil}} >0$, $H_{t}{y}_{\lceil t/2\rceil}=s_{\lceil t/2\rceil}$}, and \eqref{proof:H_kbounds} for $t$. We show that these statements also hold for $k$ \fys{as well}. First, we \fys{show} that the \fy{curvature condition} holds. We can write 
\begin{align*}
\fys{s_{\lceil k/2\rceil}^Ty_{\lceil k/2\rceil}}&=(x_{k}-x_{k-1})^T(\nabla F(x_k,\xi_{k-1})- \nabla F(x_{k-1},\xi_{k-1})+\fys{\tau}\mu_k^\delta(x_k-x_{k-1}))\\
&=(x_{k}-x_{k-1})^T(\nabla F(x_k,\xi_{k-1})- \nabla F(x_{k-1},\xi_{k-1}))+\fys{\tau}\mu_k^\delta\|x_k-x_{k-1}\|^2\\
&\geq \fys{\tau}\mu_k^\delta\|x_k-x_{k-1}\|^2,
\end{align*}
where the inequality follows from the monotonicity of the gradient map $\nabla F(\cdot,\xi)$. \fys{Next, we show that $\|x_k-x_{k-1}\|^2 >0$.} From the induction hypothesis \fys{and that $k-2$ is odd}, $H_{k-2}$ is positive definite.
\fys{Moreover, from the update rule \eqref{eqn:H-k} and that $k-2$ is odd, we have $H_{k-1}=H_{k-2}$.}  
Therefore, $H_{k-1}$ is \fys{also} positive definite. \fy{Without loss of generality, we assume $\nabla F(x_{k-1},\xi_{k-1})+\mu_{k-1}(x_{k-1}-x_0)\neq 0$}\footnote{If $\nabla F(x_{k},\xi_k)+\mu_{k}(x_k-x_0)=0$\fys{,} then we can draw a new sample of $\xi_k$ to satisfy the relation.}.
Since $H_{k-1}$ is positive definite, 
we have $$H_{k-1}\left(\nabla F(x_{k-1},\xi_{k-1})+\mu_{k-1}\fy{(x_{k-1}-x_0)}\right) \neq 0,$$ implying that 
$x_{k} \neq x_{k-1}$. Hence
$\fy{s_{\lceil k/2\rceil}^T{y_{\lceil k/2\rceil}} }\geq  \fys{\tau}\mu_k^\delta\|x_k-x_{k-1}\|^2 >0,$
where \us{the second inequality is a consequence of} $\fys{\tau},\mu_k>0$.
Thus, the \fy{curvature condition} holds.
Next, we show that \eqref{proof:H_kbounds} holds for $k$. It is well-known that using the Sherman-Morrison-Woodbury formula, $B_k$ is equal to $B_{k,m}$ given by 
\begin{align}\label{equ:B_kLimited}
B_{k,j}=B_{k,j-1}-\frac{B_{k,j-1}s_is_i^TB_{k,j-1}}{s_i^TB_{k,j-1}s_i}+\frac{y_iy_i^T}{y_i^Ts_i}, \quad i:=\fy{{\lceil k/2\rceil}-(m-j)}, \quad 1 \leq j \leq m,
\end{align}
where $s_i$ and $y_i$ are defined by \eqref{equ:siyi-LBFGS} and \fy{$B_{k,0}=\frac{y_{\lceil k/2\rceil}^Ty_{\lceil k/2\rceil}}{s_{\lceil k/2\rceil}^Ty_{\lceil k/2\rceil}}\mathbf{I}$}. \fys{Note that with $j$ varying between $1$ to $m$, the index $i$ takes values in $\left\{\lceil k/2\rceil-m+1,\lceil k/2\rceil-m+2,\ldots,\lceil k/2\rceil\right\}$.} First, we show that for any $i$ \fys{in this range}, \begin{align}\label{equ:boundsForB0}
\fys{\tau}\mu_k^\delta \leq \frac{\|y_i\|^2}{y_i^Ts_i} \leq L+\fys{\tau}\mu_k^\delta,
\end{align}
 where $L$ is the Lipschitzian parameter of the gradient mapping $\nabla F$ \fys{given by Assumption \ref{assum:convex2}(b)}. Let us \fys{define the function $h(x)\triangleq F(x,\xi_{i-1})+\tau\frac{\mu_k^\delta}{2}\|x-x_0\|^2$} for fixed $i$ and $k$. Note that this function is strongly convex and has a gradient mapping of the form $\nabla F+\fys{\tau}\mu_k^\delta(\mathbf{I}-x_0)$ that  is Lipschitz with parameter $L+\fys{\tau}\mu_k^\delta$. For a convex function $h$ with Lipschitz gradient with parameter $L+\fys{\tau}\mu_k^\delta$, the following inequality, referred to as co-coercivity  property, holds for any $x_1,x_2 \in \Real^n$ \fys{(see}\cite{Polyak87}, Pg. 24 , Lemma 2):
 \[\|\nabla h(x_2)-\nabla h(x_1)\|^2 \leq \left(L+\fys{\tau}\mu_k^\delta\right)(x_2-x_1)^T(\nabla h(x_2)-\nabla h(x_1)).\]
Substituting $x_2$ by $x_i$, $x_1$ by $x_{i-1}$, and recalling \eqref{equ:siyi-LBFGS}, the preceding inequality yields
\begin{align}\label{ineq:boundsForB0-1}\|y_i\|^2 \leq \left(L+\fys{\tau}\mu_k^\delta\right)s_i^Ty_i.\end{align}
Note that function $h$ is strongly convex \fy{with} parameter $\fys{\tau}\mu_k^\delta$. Applying the Cauchy-Schwarz inequality, we can write
\[\frac{\|y_i\|^2}{s_i^Ty_i} \geq \frac{\|y_i\|^2}{\|s_i\|\|y_i\|} =\frac{\|y_i\|}{\|s_i\|}\geq  \frac{\|y_i\|\|s_i\|}{\|s_i\|^2} \geq \frac{y_i^Ts_i}{\|s_i\|^2}\geq \fys{\tau}\mu_k^\delta.\]
Combining this relation with \eqref{ineq:boundsForB0-1}, we obtain \eqref{equ:boundsForB0}. Next, we show that the maximum eigenvalue of $B_k$ is bounded. Let $Trace(\cdot)$ denote the trace of a matrix. Taking trace from both sides of \eqref{equ:B_kLimited} and summing up over index $j$, we obtain \fy{for $i:={\lceil k/2\rceil}-(m-j)$,}
\begin{align}\label{ineq:trace}
Trace(B_{k,m})&=Trace(B_{k,0})-\sum_{j=1}^m Trace\left(\frac{B_{k,j-1}s_is_i^TB_{k,j-1}}{s_i^TB_{k,j-1}s_i}\right)+\sum_{j=1}^m Trace\left(\frac{y_iy_i^T}{y_i^Ts_i}\right)\cr
& =Trace\fy{\left(\frac{\|y_{\lceil k/2\rceil}\|^2}{s_{\lceil k/2\rceil}^Ty_{\lceil k/2\rceil}}\mathbf{I}\right)} - \sum_{j=1}^m \frac{\|B_{k,j-1}s_i\|^2}{s_i^TB_{k,j-1}s_i} + \sum_{j=1}^m \frac{\|y_i\|^2}{y_i^Ts_i}\cr 
&\leq n \fy{\frac{\|y_{\lceil k/2\rceil}\|^2}{s_{\lceil k/2\rceil}^Ty_{\lceil k/2\rceil}}} +\sum_{j=1}^m \left(L+\fys{\tau}\mu_k^\delta\right) \fy{\leq} (m+n)\left(L+\fys{\tau}\mu_k^\delta\right),
\end{align}
where the third relation is obtained by positive-definiteness of $B_k$ (this can be seen by induction on \fys{$j$}, and using \eqref{equ:B_kLimited} and $B_{k,0}\succ 0$). Since $B_k=B_{k,m}$, the maximum eigenvalue of the matrix $B_k$ is bounded \fys{by $(m+n)\left(L+\tau\mu_k^\delta\right)$}. As a result, 
\begin{align}\label{proof:lowerbound}
\lambda_{k,\min}\geq \frac{1}{(m+n)\left(L+\fys{\tau}\mu_k^\delta\right)}.\end{align}
 In the next part of the proof, we establish the bound for $\lambda_{k,\max}$. The following relation can be shown (e.g., see Lemma 3 in \cite{Mokhtari15})
\begin{align*}
det(B_{k,m})=det(B_{k,0})\prod_{j=1}^m\frac{s_i^Ty_i}{s_i^TB_{k,j-1}s_i}\fy{, \quad \hbox{for }i:={\lceil k/2\rceil}-(m-j)}.
\end{align*}
Multiplying and dividing by $s_i^Ts_i$, using the strong convexity of the function $h$, and invoking \eqref{equ:boundsForB0} and the result of \eqref{ineq:trace}, we obtain
\begin{align}\label{ineq:detBk}
det(B_{k})&=\fy{det\left(\frac{y_{\lceil k/2\rceil}^Ty_{\lceil k/2\rceil}}{s_{\lceil k/2\rceil}^Ty_{\lceil k/2\rceil}}\mathbf{I}\right)}\prod_{j=1}^m\left(\frac{s_i^Ty_i}{s_i^Ts_i}\right)\left(\frac{s_i^Ts_i}{s_i^TB_{k,j-1}s_i}\right)\cr  
& \geq\fy{\left(\frac{y_{\lceil k/2\rceil}^Ty_{\lceil k/2\rceil}}{s_{\lceil k/2\rceil}^Ty_{\lceil k/2\rceil}}\right)^n}\prod_{j=1}^m\fys{\tau}\mu_k^\delta\left(\frac{s_i^Ts_i}{s_i^TB_{k,j-1}s_i}\right)\cr
& \geq  \fys{\left(\tau\mu_k^\delta\right)^{(n+m)}} \prod_{j=1}^m \frac{1}{(m+n)\left(L+\fys{\tau}\mu_k^\delta\right)} = \frac{\fys{\left(\tau\mu_k^\delta\right)^{(n+m)}}}{(m+n)^{m}\left(L+\fys{\tau}\mu_k^\delta\right)^m}.
\end{align}
Let \fy{$\alpha_{k,1}\leq \alpha_{k,2}\leq\ldots\leq\alpha_{k,n}$} be the eigenvalues of $B_k$ sorted non-decreasingly. Note that since $B_k\succ0$, all the eigenvalues are positive. Taking \eqref{ineq:trace} and \eqref{ineq:detBk} into account, and employing Lemma \ref{sumProductBounds}, we obtain
 \[\alpha_{1,k} \fy{\geq} \frac{(n-1)!\fys{\left(\tau\mu_k^\delta\right)^{(n+m)}}}{(m+n)^{n+m-1}\left(L+\fys{\tau}\mu_k^\delta\right)^{n+m-1}}.\]
This relation and that $\alpha_{k,1}=\lambda_{k,\max}^{-1}$ imply that 
\begin{align}\label{proof:upperbound}
\lambda_{k,\max}\leq \frac{(m+n)^{n+m-1}\left(L+\fys{\tau}\mu_k^\delta\right)^{n+m-1}}{(n-1)!\fys{\left(\tau\mu_k^\delta\right)^{(n+m)}}}.
\end{align}
Therefore, from \eqref{proof:lowerbound} and \eqref{proof:upperbound} and that $\mu_k$ is non-increasing, we conclude that \eqref{proof:H_kbounds} holds for $k$ as well. Next, we show $H_ky_{\lceil k/2\rceil}=s_{\lceil k/2\rceil}$. From \eqref{equ:B_kLimited}, for $j=m$ we obtain
\[B_{k,m}=B_{k,m-1}-\fy{\frac{B_{k,m-1}s_{\lceil k/2\rceil}s_{\lceil k/2\rceil}^TB_{k,m-1}}{s_{\lceil k/2\rceil}^TB_{k,m-1}s_{\lceil k/2\rceil}}+\frac{y_{\lceil k/2\rceil}y_{\lceil k/2\rceil}^T}{y_{\lceil k/2\rceil}^Ts_{\lceil k/2\rceil}}},\]
where we used \fy{$i={\lceil k/2\rceil}-(m-m)={\lceil k/2\rceil}$}. Multiplying both sides of the preceding equation by \fy{$s_{\lceil k/2\rceil}$}, and using $B_k=B_{k,m}$, we have
\fy{$B_{k}s_k=B_{k,m-1}s_{\lceil k/2\rceil}-B_{k,m-1}s_{\lceil k/2\rceil}+y_{\lceil k/2\rceil}=y_{\lceil k/2\rceil}$}. 
Multiplying both sides of the preceding relation by $H_k$ and invoking $H_k=B_k^{-1}$, we conclude that \fy{$H_ky_{\lceil k/2\rceil}=s_{\lceil k/2\rceil}$}. Therefore, we showed that the statements of (a), (b), and (c) hold for \fys{an odd $k$}, assuming that they hold for any odd \fy{$t<k$}. In a similar fashion to this analysis, it can be seen that the statements hold for \fys{$t=2m-1$}. Thus, by induction, we conclude that the statements hold for any odd \fy{$k\geq 2m-1$}. To complete the proof, it is enough to show that \eqref{proof:H_kbounds} holds for any even \fys{$k\geq 2m$}. Let $t=k-1$. Since \fys{$t$} is odd, relation \eqref{proof:H_kbounds} holds. Writing  \eqref{proof:H_kbounds} for $k-1$, and taking into account that $H_k=H_{k-1}$, and $\mu_k<\mu_{k-1}$, we can conclude that \eqref{proof:H_kbounds} holds for any even \fy{$k\geq 2m-1$} and this completes the proof.
\end{proof}

\subsection{Proof of Lemma \ref{lemma:a.s.sequences}}\label{app:feasibleSeqASconv}
\begin{proof}
In the following, we show that the presented class of sequences satisfy each of the conditions listed in Assumption \ref{assum:sequences}. Throughout, we let $\alpha$ denote \fy{$-(m+n)\delta$}.\\
(a) Replacing the sequences by their given rules, we obtain \fy{\begin{align*} \g_k\mu_k^{2\alpha-1} & =\frac{\g_0}{(2^b\mu_0)^{1-2\alpha}}(k+1)^{-a}(k+\kappa)^{(1-2\alpha) b}\leq  \frac{\g_0}{(2^b\mu_0)^{1-2\alpha}}(k+1)^{-a+(1-2\alpha) b}.\end{align*}}
\fy{From the assumption that $\frac{a}{b}>1+2\delta(m+n)$, we obtain$-a+(1-2\alpha) b<0$}. Thus, the preceding term goes to zero verifying Assumption \ref{assum:sequences}(a).\\
(b) Let $k$ be an even number. Thus, $\kappa=2$. From \eqref{equ:seq} we have $\mu_{k}=\mu_{k+1}=\frac{\mu_02^b}{\left(k+2\right)^b}$. Now, let $k$ be an odd number. Again, according to \eqref{equ:seq} can write 
\[\mu_{k+1}=\frac{\mu_02^b}{\left((k+1)+2\right)^b }<\frac{\mu_02^b}{\left(k+1\right)^b}= \frac{\mu_02^b}{\left(k+\kappa\right)^b}=\mu_k.\]
Therefore, $\mu_k$ given by \eqref{equ:seq} satisfies \eqref{eqn:mu-k}. Also, from \eqref{equ:seq} we have $\mu_k \to 0$. Thus, Assumption \ref{assum:sequences}(b) holds.\\
(c) The given rules \eqref{equ:seq} imply that $\g_k$ and $\mu_k$ are both non-increasing sequences. Therefore, we have $\g_k\mu_{k} \leq \g_0\mu_{0}$ for any $k\geq 0$. So, to show that Assumption \ref{assum:sequences}(c) holds, it is enough to show that $\lambda_{\min}\g_0\mu_{0} \leq 1$ where $\lambda_{\min}$ is given by \eqref{equ:valuesForAssumH_k}. Since we assumed that $\g_0\mu_0 \leq L(m+n)$, for any $\delta \in (0,1]$, we have $\g_0\mu_0 \leq (m+n)(L+\mu_0^\delta)$, implying that $\lambda_{\min}\g_0\mu_{0} \leq 1$ and that Assumption \ref{assum:sequences}(c) holds.\\
(d) From \eqref{equ:seq}, we can write $$\sum_{k=0}^\infty \g_k\mu_k =\g_0\mu_02^b\sum_{k=0}^\infty (k+1)^{-a}(k+\kappa)^{-b} \geq \g_0\mu_02^b\sum_{k=0}^\infty (k+2)^{-(a+b)}= \infty,$$
where the last relation is due to $a+b\leq 1$. Therefore, Assumption \ref{assum:sequences}(d) holds.\\
(e) Using \eqref{equ:seq}, it follows
\fy{\begin{align*}&\sum_{k=0}^\infty \g_k\mu_k^2 =  \g_0\mu_0^24^b\sum_{k=0}^\infty (k+1)^{-a}(k+\kappa)^{-2b}\leq   \g_0\mu_0^24^b\sum_{k=0}^\infty (k+1)^{-(a+2b)}<\infty,\end{align*}}
where the last inequality is due to \an{$a+2b>1$}. Therefore, Assumption \ref{assum:sequences}(e) holds.\\
(f) From \eqref{equ:seq}, we have
\begin{align*}&\sum_{k=0}^\infty \g_k^2\mu_k^{2\alpha} 
\leq  \g_0^2(\mu_02^b)^{2\alpha}\left(\sum_{k=0}^1 \frac{(k+\kappa)^{-2\alpha b}}{(k+1)^{2a}}+\sum_{k=2}^\infty \frac{(2k)^{-2\alpha b}}{k^{2a}}\right)<\infty 
\end{align*}
where in the first inequality, we use $\alpha <0$ and in the last inequality, we note that 
$a+\alpha b=a-\delta\left(m+n\right)b>0.5$. Therefore, Assumption \ref{assum:sequences}(f) is verified.
\end{proof}

\subsection{Proof of Lemma \ref{lemma:mean-sequences}}\label{app:feasibleSeqMeanConv}
\begin{proof}
\fy{Throughout, we let $\alpha$ denote $-\delta(m+n)$.} \fys{Assumption \ref{assum:sequences-ms-convergence}(a, b, c) and \eqref{eqn:mu-k} have been already shown in parts (a, b, c) of the proof of Lemma \ref{lemma:a.s.sequences}.}\\
\noindent
(d) It suffices to show there exists \fy{$K_0$} such that for any $k \geq K_0$ and $0<\beta<1$, 
\begin{align}\label{ineq:partd}\frac{\g_{k-1}}{\g_k}\frac{\mu_{k}^{1-2\alpha}}{\mu_{k-1}^{1-2\alpha}}-1\leq \beta \lambda_{\min}\g_k\mu_k.\end{align}
From \eqref{equ:seq} \fy{and the definition of $\alpha$}, we obtain 
\begin{align*}&\frac{\g_{k-1}}{\g_k}\frac{\mu_{k}^{1-2\alpha}}{\mu_{k-1}^{1-2\alpha}}-1
\leq \frac{\g_{k-1}}{\g_k}-1=\left(1+\frac{1}{k}\right)^a-1= 1+\frac{a}{k}+\fy{{o}\left(\frac{1}{k}\right)}-1=\mathcal{O}\left(\frac{1}{k}\right),
\end{align*}
where the first inequality is implied due to $\{\mu_k\}$ is non-increasing, and in the second equation, we used the Taylor's expansion of $\left(1+\frac{1}{k}\right)^a$. Therefore, since the right-hand side of \eqref{ineq:partd} is of the order $\frac{1}{k^{a+b}}$ and that $a+b<1$, the preceding inequality shows that such $K_0$ exists for which Assumption \ref{assum:sequences-ms-convergence}(d) holds for all $0<\beta<1$.\\
\noindent
(e) From \eqref{equ:seq}, we have 
\begin{align*} \frac{\mu_k^{2-2\alpha}}{\g_k}& =\fy{\g_0}^{-1}\left(\mu_02^b\right)^{2-2\alpha}(k+\kappa)^{-b(2-2\alpha)}(k+1)^{a} \leq\frac{\fy{\gamma_0}^{-1}\left(\mu_02^b\right)^{2-2\alpha}}{(k+1)^{-a+(2-2\alpha)b}}\\&\leq \fy{\gamma_0}^{-1}\left(\mu_02^b\right)^{2-2\alpha}=\rho, \end{align*}
where the first inequality is due to \fy{$\alpha <0$, and the second inequality follows by the assumption $a\leq 2b(1+\delta(m+n))$}. Therefore, Assumption \ref{assum:sequences-ms-convergence}(e) is satisfied.
\end{proof}
\bibliographystyle{amsplain}
\bibliography{reference}
\end{document}